\documentclass[12pt]{amsart}
\usepackage{amsmath,amsthm, amssymb}
\usepackage{geometry}
\usepackage{booktabs}
\usepackage[utf8]{inputenc}
\usepackage[T1]{fontenc}
\usepackage[english]{babel}
\usepackage{csquotes}
\usepackage{dsfont}
\usepackage{faktor}
\usepackage[foot]{amsaddr}

\usepackage[
backend=biber,
style=numeric,
sorting=ynt
]{biblatex}
\usepackage[toc,page]{appendix}
\usepackage{bbm}
\usepackage{enumerate}
\usepackage{accents}
\usepackage{theoremref}
\usepackage{empheq}
\usepackage{bm}
\usepackage{pgfplots}
\usepackage{tikz}
\usetikzlibrary{patterns}
\usetikzlibrary{plotmarks}
\usepackage{subcaption}
\usepackage{float}
\pgfplotsset{compat=1.18}
\theoremstyle{plain}
\newtheorem{theorem}{Theorem}[section]
\newtheorem{definition}[theorem]{Definition}
\newtheorem{lemma}[theorem]{Lemma}
\newtheorem{proposition}[theorem]{Proposition}

\newtheorem{remark}[theorem]{Remark}
\newtheorem{algorithm}[theorem]{Algorithm}

\newtheorem{assumption}[theorem]{Assumption}
\numberwithin{equation}{section} 

\DeclareMathOperator*{\esssup}{ess\,sup}
\DeclareMathOperator*{\essinf}{ess\,inf}
\DeclareMathOperator*{\argmin}{arg\,min}
\allowdisplaybreaks

\usepackage{hyperref}
\renewcommand{\geq}{\geqslant}
\renewcommand{\leq}{\leqslant}

\usetikzlibrary{plotmarks}

\begin{document} 

\title{Error estimates for numerical approximations of a nonlinear gradient flow model}
\author{J\'{e}r\^{o}me Droniou$^{1,2}$}%
\author{Kim-Ngan Le$^2$}%
\author{Huateng Zhu$^2$}%
\address{$^1$IMAG, Univ. Montpellier, CNRS, Montpellier, France }
\address{$^2$School of Mathematics, Monash University, Australia}%

\email{jerome.droniou@cnrs.fr,
ngan.le@monash.edu, 
huateng.zhu@monash.edu}%

\thanks{This work was partially supported by the Australian Government through the Australian Research Council’s Discovery
Projects funding scheme (grant number DP220100937).}

    \begin{abstract}
        We perform numerical analysis of a nonlinear gradient flow, which can be regarded as a parabolic minimal surface problem or a regularised total variation flow, using the gradient discretisation method (GDM).
        GDM is a unified convergence analysis framework that covers conforming and nonconforming numerical methods, for instance, conforming and nonconforming finite element, two-point flux approximation, etc..
        In this paper, a fully discretised implicit scheme of the model is proposed, the existence and uniqueness of the solution to the scheme is proved, the stability and consistency of the scheme are analysed, and error estimates are established. Numerical results based on the conforming and nonconforming $\mathbb{P}^1$ finite elements are also provided.
    \end{abstract}  
\maketitle

    \section{Introduction}
Let \(\Omega \subset \mathbb{R}^2\) with Lipschitz boundary \(\partial \Omega\), 
\(\bm{n}\) be the outward normal to \(\Omega\), 
\(T > 0\) be finite, 
and \(\Omega_T := (0,T) \times \Omega\).
Assume \(g:\Omega \rightarrow \mathbb{R}\) and \(u_0: \Omega \rightarrow \mathbb{R}\) are given.
We are interested in the following model:
for \(\lambda \geq 0\) and \(\rho > 0\),
\begin{subequations}\label{model:rtvf}
    \begin{align}
    &\partial_t u = \text{div} \, \bigg( \frac{\nabla u}{\sqrt{\rho^2 + |\nabla u|^2}} \bigg) - \lambda (u - g) &&\text{in \(\Omega_T\),}\label{eq-model:rtvf c0}\\
    &\nabla u \cdot \bm{n} = 0 &&\text{on \((0,T) \times \partial \Omega\),} \label{eq-model:rtvf c1}\\
    &u(0, \cdot) = u_0 &&\text{in \(\Omega\);}\label{eq-model:rtvf c2}
\end{align}\end{subequations} 
which arises formally as a \(L^2\) gradient flow of the energy functional \(\mathcal{J}_{\rho}\)
defined via
\begin{equation}\label{model:rtvf functional}
    \mathcal{J}_{\rho}(u) := \int_{\Omega} \sqrt{\rho^2 + |\nabla u|^2} + \frac{\lambda}{2}|u - g|^2 \, dx.
\end{equation}

    We now present two applications of the model.
    The first one is from the perspective of geometric flow.
    In the case when \(\rho = 1\) and \(\lambda = 0\), 
    the functional \eqref{model:rtvf functional} computes the area of a surface \(u\),
    and its corresponding stationary model (sometimes referred to as \textit{minimal surface problem})
    \begin{equation*}
        \text{div} \, \bigg( \frac{\nabla u}{\sqrt{1 + |\nabla u|^2}} \bigg) = 0
    \end{equation*}
    is satisfied by surfaces with zero curvature and thus with minimal area. Finite element approximations of the minimal surface problem have been studied in \cite{rannacher1977some,johnson1975error}.
    The model \eqref{model:rtvf} (with \(\rho = 1\) and \(\lambda = 0\)) corresponds to a time-dependent minimal surface problem.
    This is an analytically well-understood model (see, e.g. Lichnewsky and Temam \cite{lichnewsky1978pseudosolutions} and Gerhardt \cite{gerhardt1980evolutionary}).
    Deckelnick and Dziuk \cite{deckelnick1995convergence} establish error estimates for continuous in time finite element approximations and use \eqref{model:rtvf} to approximate a non-parametric mean curvature flow.


The second application is to approximate the total variation flow.
The functional \eqref{model:rtvf functional} and the model \eqref{model:rtvf} can serve as a strongly convex regularisation and an approximation to the total variation functional
\begin{equation}\label{model:tvf functional}
    \mathcal{J}(u) := \int_{\Omega} |\nabla u| + \frac{\lambda}{2}|u - g|^2 \, dx,
\end{equation}
and its corresponding \(L^2\) subgradient flow (called the \textit{total variation flow})
\begin{equation}\label{model:tvf}
    \partial_t u = \text{div} \, \bigg( \frac{\nabla u}{ |\nabla u|} \bigg) - \lambda (u - g) \quad \text{in \(\Omega_T\) with conditions \eqref{eq-model:rtvf c1} and \eqref{eq-model:rtvf c2}}.
\end{equation}
The minimisation problem of \eqref{model:tvf functional} and model \eqref{model:tvf} have received considerable attention in the community of image processing in recent decades,
due to its successful application in noise removal, image reconstruction, and segmentation; see, for example, \cite{Rudin1992NonlinearTV, moeller2015learning, wang2008new, unger2008tvseg, Chambolle1997ImageRV,Chan2003OnTR}.

The minimisation problem of \eqref{model:tvf functional} have been extensively studied numerically.
Bartels \cite{bartels2012total} proves that the piecewise linear finite element approximation converges to the exact solution and shows that convergence of piecewise constant approximation cannot be expected in general. 
Subsequent work by Bartels \cite{bartels2015error} yielded an optimal a posteriori error estimate for the finite element approximation leading to a successful adaptive refinement strategy.
A key result was established by Chambolle and Pock \cite{chambolle2020crouzeix} who derived error estimates with a quasi-optimal rate \(\mathcal{O}(h^{1/2})\) using the Crouzeix-Raviart finite element approximation of \eqref{model:tvf functional}, 
a finding later improved by Bartels and Kaltenbach \cite{bartels2022error} to relaxes the required assumption on dual solutions. 
In parallel with the efforts to relax dual solution assumptions,
Bartels, Tovey, and Wassmer \cite{bartels2022singular} devise graded grid approaches using locally refined meshes. These methods achieved improved convergence rates for the approximations despite the occurrence of discontinuities.
We refer to \cite{bartels2012total} for a more thorough review of the minimisation problem.

The well-posedness of \eqref{model:tvf} has been addressed using semigroup theory, see, for example, \cite{andreu2001minimizing, andreu2001dirichlet, Andreu2002SomeQP, bellettini2002total}.
It is however challenging to discretise the model \eqref{model:tvf}, since the model \eqref{model:tvf} demonstrates strong singular phenomenon \cite{Kobayashi1998EquationsWS, giga2010very}; 
that is, once the singularity \(\{|\nabla u| = 0\}\) appears, it would quickly dominate the entire domain.
Bartels, Nochetto, and Salgado \cite{bartels2014discrete} discretise the subgradient flow \eqref{model:tvf} using continuous piecewise linear functions (on conforming regular mesh)
and show the convergence rate \(\mathcal{O}(h^{1/3})\).
In \cite{giga2020new}, a new numerical scheme is proposed for \eqref{model:tvf} that is constrained to a Riemannian manifold; convergence is established without the error estimates.

In the context of image processing,
the model \eqref{model:rtvf} is often referred to as the \textit{regularised total variation flow}.
It is known that the solution of \eqref{model:rtvf} converges to the solution of model \eqref{model:tvf} in the sense of a variational inequality when the regularisation parameter \(\rho\) approaches zero, see \cite{Feng2003AnalysisOT}.
However, the model \eqref{model:tvf} corresponds to the model \eqref{model:rtvf} with \(\rho\) equal to zero, which is outside the scope of problem \eqref{model:rtvf} in the context of numerical experiments,
so it is not surprising to see different behaviours between these two fundamentally different models.
Compare with the total variation flow \eqref{model:tvf}, its regularised version \eqref{model:rtvf} has a more natural variational formulation for classical and modern numerical methods.
The model \eqref{model:rtvf} also provides decent results in image denoising and restoration, see, for instance, \cite{chen2012gradient, chen2000image,Feng2003AnalysisOT, feng2005rate, calder2010image, Chan2003OnTR, handlovivcova1998numerical}.

Convergence of approximations for \eqref{model:rtvf} has been shown using 
finite element schemes \cite{barrett2008p,feng2009finite,chen2012gradient} and a finite difference scheme \cite{hong2021convergence}.
Due to the low coercivity (or ellipticity) of the diffusion coefficient,
numerical approximations of the model \eqref{model:rtvf} are not covered by general diffusion models like the fast diffusion flow 
and the diffusion flow with Leray-Lion operator,
whose numerical approximations and their convergence have been widely studied;
we refer the reader to \cite{droniou2020gradient, droniou2013gradient} and references therein.

Error estimates for semi-discretised schemes for \eqref{model:rtvf} have been studied in both temporal and spatial cases.
Kunstmann, Li, and Lubich \cite{kunstmann2018runge} show the convergence of an implicit Runge-Kutta time discretisation for a diffusion model (that includes \eqref{model:rtvf}) with an assumption on the exact solution that is \((n+1)\)-differentiable in time
and the convergence rate \(\mathcal{O}(\delta t^{n+1})\).
Li, Ueda, and Zhou \cite{li2020second} establish error estimates for a temporal discretisation scheme of a nonlinear parabolic equation using functions in Sobolev space \(W^{n,2}\) (\(n=0,1,2,3\)), with rate \(\mathcal{O}(\delta t^{2-n/2})\).
Deckelnick and Dziuk \cite[Lemma 5.1]{deckelnick1995convergence} show the convergence of \eqref{model:rtvf}(with \(\lambda = 0\)) for the spatial discretisation, using finite elements, with rates  \(\mathcal{O}(h^2|\ln h|^2)\) and \(\mathcal{O}(h|\ln h|)\) for the approximations of solution and gradient respectively; 
however, this result requires smooth boundary \(\partial \Omega \in C^5\) and initial data \(u_0 \in C^5(\overline{\Omega})\).

However, the error estimates for fully discretised schemes for \eqref{model:rtvf} have only been analysed in the context of conforming finite element.
Feng and Prohl \cite{Feng2003AnalysisOT} establish error estimates for approximations of solution and gradient using conforming P1 finite element approximation, with rates \(\mathcal{O}(h^2)\) and \(\mathcal{O}(h)\) respectively.
Li and Sun \cite{li2014linearized} show the convergence of a mixed implicit conforming finite element scheme (with the diffusion coefficient being explicit) for \eqref{model:rtvf} (with \(\lambda = 0\)) using continuous piecewise degree \(n\) polynomials with rate \(\mathcal{O}(h^{n + 1})\).
Bartels, Diening, and Nochetto \cite{bartels2018unconditional} provide error estimates that control the influence of both the regularization and the semi-implicit discretization.
Akrivis and Li \cite{akrivis2021linearization} linearise an implicit finite element scheme for diffusion equation with smooth flux (that includes \eqref{model:rtvf}) using the Newton's method, and proved the convergence rates for the approximate solution and its gradient are close to \(2\) and \(1\) in mesh size respectively.

In this work, we establish error estimates for the model \eqref{model:rtvf} using the gradient discretisation methods (GDM).
GDM is a generic framework for the design and analysis of numerical schemes for diﬀusion models. 
Using three discrete elements (a discrete space, approximations of function and gradient), a generic formulation can be obtained to describe numerical schemes in the framework.
Although we cannot give here an exhausive list of methods in the framework, let us mention a few of them:
the conforming finite element schemes (with or without mass lumping), 
the Crouzeix-Raviart non-conforming finite element (with or without mass lumping), 
the Raviart-Thomas mixed finite elements,
non-conforming virtual element methods,
hybrid mimetic mixed methods (HMM),
mimetic finite difference methods,
and the hybrid high order schemes.
All methods that fit into the framework inherit results that are established using the GDM approach.
For more detailed discussions, we refer to \cite{Droniou2018TheGD}.

    The main result of this paper consists of error estimates for the approximate solution in the \(L^{\infty}(0,T; L^2(\Omega))\)-norm and for the gradient approximation in the \(L^1(0,T;L^1(\Omega))\)-norm.
    Under standard assumptions on the domain and the exact solution, the convergence of the approximations to the exacts holds true for all numerical methods that are in the framework of GDM with the convergence rate 1 with respect to the time step and the mesh size.
    
    In \cite{rannacher1977some,johnson1975error}, a minimal surface projection was introduced using conforming finite element method along with error estimates. 
    In our analysis, we consider a nonlinear modified Riesz projection that is almost identical to (and also includes) the minimal surface projection using GDM, by proving its existence and error estimates.
    The minimal surface projection has been used for error estimates of finite element approximations for geometric flows in several studies, see, for example, \cite{deckelnick1995convergence, deckelnick2000error, deckelnick2006error}.
    Therefore, this projection and our error estimates have the potential in application for nonconforming numerical analysis of geometric problems, such as the minimal surface problem, parabolic minimal surface problem and the non-parametric mean curvature flow.

    It is known that, in the case \(\rho = 0\) (i.e. total variation flow), non-conforming schemes have a better behaviour than conforming schemes; 
    since convergence of conforming finite approximations for \eqref{model:tvf} cannot be expected in general \cite{bartels2012total} and non-conforming finite approximations are known to have quasioptimal convergence rate \(\mathcal{O}(h^{1/2})\), \cite{chambolle2020crouzeix, bartels2022error}. 
    Consequently, it seems natural to also consider non-conforming methods as particularly suitable for the model \eqref{model:rtvf} with small \(\rho > 0\) when it is used to approximate the total variation flow \eqref{model:tvf}.
    Since we use the generic GDM framework, our analysis precisely covers such methods (e.g., non-conforming finite elements) and shows that, in ideal circumstances, these methods converge in \(\mathcal{O}(h)\) (in comparison to \(\mathcal{O}(h^{1/2})\) in \cite{chambolle2020crouzeix, bartels2022error}).
    We however do not, at this stage, establish an error estimate that is robust with respect to \(\rho > 0\), which remains an interesting topic for a future work.

To summarise, the contributions of this paper are twofold: 
\begin{enumerate}
    \item Error estimates for numerical approximations of the model \eqref{model:rtvf} in \(\mathbb{R}^2.\) 
    \item Error estimates of a nonlinear spatial interpolator, in \(\mathbb{R}^2\).
\end{enumerate}
Both of these error estimates are analysed using the generic GDM framework.
To the best of our knowledge, our error estimates are the first for nonconforming numerical approximations of\textbf{} the model \eqref{model:rtvf}.

This paper is structured as follows:
\begin{itemize}
    \item Section \ref{section:notations} introduces the notation for the gradient discretisation method, a gradient scheme, and the main results.
    \item Section \ref{section:interpolation estimate} introduces a non-linear spatial interpolator along with  error estimates. 
    \item Section \ref{section:stability and consistency} establishes the existence and uniqueness of the solution to the gradient scheme, and analyses the stability and consistency of the gradient scheme.
    \item Section \ref{section:main proof} gives a proof for the main theorems.
    \item Section \ref{sec:numerical example} presents numerical examples.
\end{itemize}
A more detailed description of each section is provided at the start of each section.
    \section{Notations and main results}\label{section:notations} 
For \(1 \leq p \leq \infty,\) \(L^p(\Omega)\) denotes the Lebesgue space and \(W^{s,p}(\Omega)\), with \(s\in \mathbb{R}\), denotes the Sobolev space.
For a Banach space \(X\), let \(L^p(0,T; X)\) be Bochner space of Bochner-measurable functions \(u:(0,T) \rightarrow X\) satisfying \(t \mapsto \|u(t)\|_X \in L^p(0,T)\).
Unless otherwise stated, \(\langle f,g \rangle \) denotes \(\int_{\Omega} fg,\) for any \(f,g\).
Throughout the paper, 
\(\rho > 0\) is fixed, 
and 
\begin{equation}\label{eq:def weighted rho norm}
| \cdot |_{\rho} := \sqrt{\rho^2 + |\cdot|^2} \ ;
\end{equation}
\(\Lambda^{(i)}\) \((i = 1, 2, \ldots)\) are used to denote different constants.

\subsection{Gradient discretisation and gradient scheme} 
Notations for the gradient discretisation method (GDM) are briefly recalled here; refer to the book \cite{Droniou2018TheGD} by Droniou \textit{et al.} for detailed discussions.

GDM starts by selecting a finite number of discrete unknowns \(X_{\mathcal{D}}\) describing the finite dimensional space  where the approximate solution is sought.
To reconstruct functions and vector-valued functions over \(\Omega\) from the discrete unknowns \(X_{\mathcal{D}}\), linear operators \(\Pi_{\mathcal{D}}\) and \(\nabla_{\mathcal{D}}\) are needed.
The set of discrete space and operators, \((X_{\mathcal{D}}, \Pi_{\mathcal{D}}, \nabla_{\mathcal{D}})\),  is called a \textit{(spatial) gradient discretisation} (GD).  
A space-time GD is a combination of a spatial GD and a time partition.
See Definition \ref{def:GD definition}.

A scheme can be obtained by replacing the continuous space and operators by \(X_{\mathcal{D}}\) and \((\Pi_{\mathcal{D}}, \nabla_{\mathcal{D}})\) in the weak formulation \eqref{eq-def:weak form of rtvf}.
This scheme is called a \textit{gradient scheme} (GS). See Algorithm \ref{def:GS}.

Two approximation properties of GD are required to enable error estimate between the exact solution \(\overline{u}\) to \eqref{model:rtvf} and its approximation:
The \textit{GD-consistency} states that continuous functions and their gradient over \(\Omega\) can be approximated by elements in \(X_{\mathcal{D}}\); see Definition \ref{def:GD consistency}.
The \textit{limit-conformity} states that, asymptotically, \(\Pi_{\mathcal{D}}\) and \(\nabla_{\mathcal{D}}\) satisfy a divergence theorem; see Definition \ref{def:GD limit conformity}.
The notion of \textit{space size} is introduced in Definition \ref{def:GD space size} to quantitatively measure the approximation properties of GD.

\begin{definition}\label{def:GD definition}
    A \textit{space-time gradient discretisation} \(\mathcal{D}_T\) for homogeneous Neumann boundary conditions is defined by
    \(
        \mathcal{D}_{T} = (X_{\mathcal{D}}, \Pi_{\mathcal{D}}, \nabla_\mathcal{D}, ({t^{(m)}})_{ m = 0}^M)
    \) 
    where:
\begin{enumerate}[(i)]
    \item the set of discrete unknowns \(X_{\mathcal{D}}\) is a finite dimensional vector space on \(\mathbb{R}\);
    \item the \textit{function reconstruction} \(\Pi_{\mathcal{D}} : X_{\mathcal{D}} \rightarrow L^{\infty}(\Omega)\) is linear;
    \item the \textit{gradient reconstruction} \(\nabla_{\mathcal{D}}: X_{\mathcal{D}} \rightarrow L^{\infty}(\Omega)^2\) is linear and satisfies: 
     for any \(u\in X_{\mathcal{D}}\), \(\nabla_{\mathcal{D}} u   \in \mathbb{P}^N(\mathcal{T})\) is piece-wise polynomial of degree \(N\) on a \textit{quasi-uniform} mesh \(\mathcal{T}\) with mesh size \(h_{\mathcal{T}}\) in then sense of Definition \ref{def:mesh};
    \item there exists \(\bm{1}_{\mathcal{D}} \in X_{\mathcal{D}}\) such that \(\Pi_{\mathcal{D}} \bm{1}_{\mathcal{D}} = 1\) and \(\nabla_{\mathcal{D}} \bm{1}_{\mathcal{D}} = 0\);
    \item the operators \(\nabla_{\mathcal{D}}\) and \(\Pi_{\mathcal{D}}\) are such that the following is a norm on \(X_{\mathcal{D}}\):
     \begin{equation}\label{eq-def:GD definition norm}
        \| v \|_{\mathcal{D}}:=  \bigg| \int_{\Omega} \Pi_{\mathcal{D}} v \bigg|+ \|\nabla_{\mathcal{D}} v\|_{L^1(\Omega)};
    \end{equation}
    \item \(t^{(0)} = 0 < t^{(1)} < \cdots  < t^{(M)} = T\). 
\end{enumerate}
For \(\ell = 1, \ldots, M\), the  \textit{\(\ell\)-th time step} \(\delta t^{(\ell)}\) is the difference of \(t^{(\ell)}\) and \(t^{(\ell - 1)}\), 
the \textit{maximum time-step} \(\delta t_{\mathcal{D}}^{\max}\) is the maximum of all time-steps, 
the \textit{minimum time-step} \(\delta t_{\mathcal{D}}^{\min}\) is the minimum of all time-steps, i.e.
\begin{align*}
    \delta t^{(\ell)} = t^{(\ell)} - t^{(\ell - 1)}, \quad {\delta t}_{\mathcal{D}}^{\max}= \max_{1 \leq \ell \leq M} \delta t^{(\ell)}, \quad {\delta t}_{\mathcal{D}}^{\min}= \min_{1 \leq \ell \leq M} \delta t^{(\ell)}.
\end{align*}
For a family \(u  = (u^{(0)}, u^{(1)}, \ldots, u^{(M)}) \in X_{\mathcal{D}}^{M+1}\), 
we define the piecewise constant-in-time functions \(\Pi_{\mathcal{D}} u : [0,T] \rightarrow L^{\infty}( \Omega)\), \(\nabla_{\mathcal{D}} u: [0,T] \rightarrow L^{\infty}(\Omega)^2\), and \textit{discrete time derivative} \(\delta_{\mathcal{D}} u: [0,T] \rightarrow L^{\infty}(\Omega) \) as:
for all \(m = 1, \ldots, M,\) for any \(t\in (t^{(m-1)}, t^{(m)}]\),
\begin{align*} 
    \Pi_{\mathcal{D}}u(t) := \Pi_{\mathcal{D}} u^{(m)}, \ \nabla_{\mathcal{D}}u(t) := \nabla_{\mathcal{D}} u^{(m)}, \ \delta_{\mathcal{D}} u (t) = \delta_{\mathcal{D}}^{(m)} u  := \frac{\Pi_{\mathcal{D}} u^{(m)}  - \Pi_{\mathcal{D}} u^{(m - 1)}}{\delta t^{(m)}}.
\end{align*}
\end{definition}
\begin{remark}
    The standard notion of GD for homogeneous Neumann boundary conditions is given in \cite[Definition 3.1]{Droniou2018TheGD}, here we extend the definition by including the existence of the element \(\bm{1}_{\mathcal{D}}\), which is always satisfied in practical applications.
\end{remark}
\begin{definition}\label{def:mesh}
    A mesh \(\mathcal{T}\) is \textit{quasi-uniform} if there exists a real number \(\sigma \in (0,1)\), for each cell \(\tau \in \mathcal{T}\), \(\sigma h_{\tau}^2 \leq |\tau|_2 \) and \(\sigma h_{\mathcal{T}} \leq  h_{\tau}\), 
    where \(h_{\tau}\) denotes the diameter of \(\tau\),
    \(|\cdot|_2\) denotes the Lebesgue measure in \(\mathbb{R}^2\) and \(h_{\mathcal{T}} := \max_{\tau} h_{\tau}\).
\end{definition}
\begin{algorithm}[Gradient scheme]\label{def:GS}
    Take \(u^{(0)} \in X_{\mathcal{D}}\), and consider  \((u^{(m)})_{m = 1, \ldots, M} \in X_{\mathcal{D}}^M\), such that, for all \(m = 1, \ldots, M\), for any \(v\in X_{\mathcal{D}}\), 
    \begin{equation}\label{eq-def:GS}\begin{aligned}
             & \int_{\Omega} \Pi_{\mathcal{D}} u^{(m)}(x) \, \Pi_{\mathcal{D}}v(x) \, dx 
                + \delta t^{(m)}  \int_{\Omega} \frac{ \nabla_{\mathcal{D}}u^{(m)}(x)  }{|\nabla_{\mathcal{D}}u^{(m)}(x)|_{\rho}}\cdot \nabla_{\mathcal{D}}v(x) \, dx\\
            & + \delta t^{(m)} \int_{\Omega} \lambda (\Pi_{\mathcal{D}} u^{(m)}(x) - g(x)) \, \Pi_{\mathcal{D}}v(x) \, dx =   \int_{\Omega} \Pi_{\mathcal{D}} u^{(m-1)}(x) \, \Pi_{\mathcal{D}}v(x) \, dx.
    \end{aligned}\end{equation} 
\end{algorithm}
\begin{definition}[{GD-Consistency}]\label{def:GD consistency}
    Define  \(S_{\mathcal{D}}: H^1(\Omega) \rightarrow [0, +\infty)\) by 
    \begin{equation*} 
        \forall \varphi \in H^1(\Omega), \quad  S_{\mathcal{D}}(\varphi) := \min_{v\in X_{\mathcal{D}}} \big(\|\Pi_{\mathcal{D}} v - \varphi \|_{L^2(\Omega)} + \|\nabla_{\mathcal{D}} v - \nabla \varphi\|_{L^2(\Omega)}\big).
    \end{equation*}
\end{definition} 
\begin{definition}[{Limit-conformity}]\label{def:GD limit conformity}
    Define \(W_{\mathcal{D}}: H_{\text{div}}(\Omega) \rightarrow [0, \infty)\) by
    \begin{equation*}
        \forall \bm{\varphi} \in H_{\text{div}}(\Omega):= \left\{ \bm{\varphi} \in L^2(\Omega)^2: \text{div} \, \bm{\varphi}  \in L^2(\Omega)\right\}, 
        \quad
        W_{\mathcal{D}} (\bm{\varphi}) := \max_{v\in X_{\mathcal{D}} \setminus \{0\}} {\displaystyle  \frac{|\widetilde{W}_{\mathcal{D}} (\bm{\varphi},v)|}{\|v\|_{\mathcal{D}}} } , 
    \end{equation*}
    where 
    \begin{align*}
        \widetilde{W}_{\mathcal{D}} (\bm{\varphi}, v) =  \int_{\Omega} \nabla_{\mathcal{D}}v(x) \cdot \bm{\varphi}(x) \, dx 
        + \int_{\Omega} \Pi_{\mathcal{D}}v(x) \, \text{div}\,\bm{\varphi}(x) \, dx.    \end{align*}
\end{definition}
\begin{definition}[{Space size}]\label{def:GD space size}
    Let \(W_s\) be a Banach space that is continuously embedded in \(H^1(\Omega)\).
    Let \(\bm{W}_w\) be a Banach space that satisfies:
    \begin{equation*}
        \forall \bm{\psi} \in \bm{W}_w, \quad  \frac{\bm{\psi}}{|\bm{\psi}|_{\rho}} \in H_{\text{div}}(\Omega).
    \end{equation*} 
    Define  
    \begin{equation*}
        s_{\mathcal{D}}(W_s) := \sup_{\varphi \in  W_s \setminus \{0\}}  \frac{S_{\mathcal{D}} (\varphi)}{\|\varphi\|_{W_s}} 
        \quad \text{and} \quad 
        w_{\mathcal{D}}(\bm{W}_w) := \sup_{\bm{\psi} \in \bm{W}_w \setminus \{0\}} \frac{W_{\mathcal{D}}\big({\bm{\psi}}/{ |\bm{\psi}|_{\rho}}\big)} {\|\bm{\psi}\|_{\bm{W}_w}}.
    \end{equation*}
    The \textit{space size} of \(\mathcal{D}_T\) with respect to \(W_s\) and \(\bm{W}_w\) is defined by 
    \begin{align*}
        h_{\mathcal{D}} (W_s; \bm{W}_w) := \max\left(s_{\mathcal{D}}(W_s), w_{\mathcal{D}}(\bm{W}_w)\right).
    \end{align*}
\end{definition}
\begin{remark}
    For all mesh-based lower order methods in the framework, 
    the space size is bounded above by the mesh size (up to a multiplicative constant that depends on the mesh regularity), provided that \(W_s\) and \(\bm{W}_w\) are regular enough (see \cite[Remark 2.24]{Droniou2018TheGD}).
\end{remark}

\subsection{The main result}
    In the rest of the paper,
    \(a \lesssim b\) denotes \(a \leq C b\) if there exists a positive \(C\) that only depends on \(\overline{u}, \Omega, T, \rho\), \(N\) and \(\sigma\);
     \(a \sim b\) denotes \(b \lesssim a \lesssim b\). 
\begin{definition}\label{def:weak form of rtvf}
    A function \(\overline{u}\) is a \textit{weak solution} of \eqref{model:rtvf} if 
    \( \overline{u}\) belongs to \( L^{1}((0,T);BV(\Omega))\) and \(C^0([0,T]; L^2(\Omega))\), \(\partial_t \overline{u} \in L^2(0,T;H^{1}(\Omega)')\), \(\overline{u}(\cdot, 0) = u_0 (\cdot)\), 
    and for any \(\phi\in L^2((0,T);H^1(\Omega)),\)
    \begin{equation}\label{eq-def:weak form of rtvf} \begin{aligned}
            &  \int_{0}^T  \big\langle\partial_t \overline{u}(t) , \phi(t) \big\rangle_{(H^1)',H^1}\, dt 
            + \int_{0}^{T}\int_{\Omega} \frac{\nabla \overline{u}(x,t) }{|\nabla \overline{u} (x,t)|_{\rho}} \cdot \nabla \phi(x,t) \, dx \, dt \\
            & + \int_0^T \int_{\Omega}  \lambda \overline{u}(x,t) \, \phi(x,t) \, dx \, dt 
            = \int_0^T \int_{\Omega} \lambda g(x) \, \phi(x,t) \, dx \, dt.
    \end{aligned}\end{equation} 
\end{definition}
The existence and uniqueness of a variational solution that is equivalent to Definition \ref{def:weak form of rtvf} has been addressed by Feng and Prohl \cite{Feng2003AnalysisOT} using the compactness method.
\begin{assumption}\label{def:GDM Error Estimate Theorem assumption}
    Let \(\mathcal{D}_T\) be a space-time GD for homogeneous Neumann boundary conditions, in the sense of Definition \ref{def:GD definition}
    and let \(W_s\) and \(\bm{W}_w\) be as in Definition \ref{def:GD space size}.
    Let \(h_{\mathcal{D}} := \max\{h_{\mathcal{D}} (W_s; \bm{W}_w), h_{\mathcal{D}}(H^2(\Omega), H^1(\Omega)^2)\}\) be given by Definition \ref{def:GD space size}.
    Let \(P_{\mathcal{D}}\) be the non-linear spatial interpolator in the sense of Lemma \ref{lem:existence of NL interpolator} below.
    Denote 
    \begin{equation}\label{def:eDini}
        e_{\mathcal{D}}^{ini}  := \| \Pi_{\mathcal{D}} u^{(0)} - \Pi_{\mathcal{D}} P_{\mathcal{D}} u_0 \|_{L^2(\Omega)}.
    \end{equation}
    Assume the following:
    \begin{subequations}  
        \begin{equation}\tag{H1}
            \text{The domain \(\Omega \subset \mathbb{R}^2\) is convex, bounded, and polygonal.}
        \end{equation}
        \begin{equation}\tag{H2}
            \text{\(u_0 \in W^{2, \infty}(\Omega) \cap W_s.\)}
        \end{equation}
        \begin{equation}\tag{H3}
        \begin{aligned}
                & \text{The exact solution } \overline{u} \text{ of \eqref{model:rtvf}}  \text{ is Lipschitz-continuous } [0,T] \rightarrow H^2(\Omega) \cap W_s,\\
                & \text{and belongs to \( H^2(0,T; L^2(\Omega)) \cap L^{\infty}(0,T; W^{2, \infty}(\Omega))\).}
        \end{aligned}\end{equation}
        \begin{equation}\tag{H4}
            \text{\(\nabla \overline{u}\) is Lipschitz-continuous \([0,T] \rightarrow L^{\infty}(\Omega) \cap \bm{W}_w\).}
        \end{equation}
        \begin{equation}\tag{H5}\label{eq-def:GDM Error Estimate Theorem assumption 2} 
             h_{\mathcal{D}}  \lesssim h_{\mathcal{T}}.
        \end{equation}
    \end{subequations}
\end{assumption}

Now, we have enough ingredients to state the main theorems.
The first theorem concerns the error estimates with respect to the time step and the space size, in the \(L^1\)-norm on the gradient.
\begin{theorem}\label{thm:GDM Error Estimate Theorem}
    Suppose that Assumption \ref{def:GDM Error Estimate Theorem assumption} holds.
    There exists a unique solution \(u = (u^{(\ell)}: \ell = 1, \ldots, M) \in X_{\mathcal{D}}^{M}\) to the gradient scheme  \eqref{eq-def:GS};  
    this solution satisfies the following inequality:
    \begin{equation}\label{eq:GDM Error Estimate Theorem 1}
             \max_{1 \leq \ell \leq M} \| \Pi_{\mathcal{D}} u^{(\ell)} -   \overline{u}(t^{(\ell)}) \|_{L^2(\Omega)} + \| \nabla_{\mathcal{D}} u - \nabla \overline{u} \|_{L^1(0,T; L^1(\Omega))} \lesssim e_{\mathcal{D}}^{\text{ini}} + (e_{\mathcal{D}}^{\text{ini}})^2 + \delta t_{\mathcal{D}}^{\max} + h_{\mathcal{D}}.
    \end{equation}  
    
\end{theorem}
\begin{proof}
    See Section \ref{section:main proof}.
\end{proof}
The second theorem concerns the error estimates with respect to the mesh size, in the \(L^2\)-norm on the gradient.
\begin{theorem}\label{thm:GDM Error Estimate Theorem 2}
    Suppose that Assumption \ref{def:GDM Error Estimate Theorem assumption} holds.
    In addition, assume that 
    \begin{equation}\label{eq-thm:GDM Error Estimate Theorem 2 assumption}
        e_{\mathcal{D}}^{\text{ini}} + \delta t_{\mathcal{D}}^{\max} \lesssim h_{\mathcal{T}},
    \end{equation}
    the solution \(u \) to the gradient scheme \eqref{eq-def:GS} satisfies the following inequality:
    \begin{equation*}
             \max_{1 \leq \ell \leq M} \| \Pi_{\mathcal{D}} u^{(\ell)} -   \overline{u}(t^{(\ell)}) \|_{L^2(\Omega)} + \| \nabla_{\mathcal{D}} u - \nabla \overline{u} \|_{L^2(0,T; L^2(\Omega))} \lesssim  h_{\mathcal{T}}.
    \end{equation*}  
\end{theorem}
\begin{proof}
    See Section \ref{section:main proof}.
\end{proof}

\begin{remark}
    Assumption \ref{def:GDM Error Estimate Theorem assumption} is quite strong but, since the model we consider is not singular, higher regularity properties on the solution can be establishes in certain situations, see \cite[Theorem 1.2]{Feng2003AnalysisOT}. Such regularity is classically required to establish a convergence rate in terms of mesh size. If the solution is not expected to be smooth, a proof of convergence based on compactness techniques is more appropriate. These techniques have been applied in the GDM framework for time-dependent nonlinear and singular PDEs \cite[Section 6]{Droniou2018TheGD}; extending them to the present gradient-flow model could be the subject of future work.
\end{remark}

\begin{remark}[Spatial dimension]
As in previous works (see, e.g., \cite{Feng2003AnalysisOT}), our analysis strongly relies on the fact that the spatial dimension is $2$.
Having in mind gradient flows associated with minimal surfaces or image denoising, this is actually the relevant dimension to consider.
\end{remark}

\section{The non-linear spatial interpolator}\label{section:interpolation estimate}
The main results of this section is the following: 
\begin{itemize}
    \item An interpolation \(P_{\mathcal{D}}: H^1(\Omega) \rightarrow X_{\mathcal{D}}\) is introduced in Lemma \ref{lem:existence of NL interpolator} with a proof of the existence;
    \item Error estimates of \(P_{\mathcal{D}}\) is provided in Lemma \ref{lem:NL interpolation error estimate} and Lemma \ref{lem:NL interpolation error estimate time derivative}, which are established by borrowing a few tricks from \cite{rannacher1977some,johnson1975error,Droniou2016ImprovedE}.
\end{itemize}

We first recall a linear spatial interpolator and its properties. 
\begin{lemma}
    Define \(I_{\mathcal{D}}:H^1(\Omega) \rightarrow X_{\mathcal{D}}\) via: for any \(\varphi \in H^1(\Omega)\),
    \begin{equation}\label{eq-lem:linear interpolator def}
        I_{\mathcal{D}} \varphi = \argmin_{v \in X_{\mathcal{D}}} \|\Pi_{\mathcal{D}} v - \varphi \|_{L^2(\Omega)}^2 + \|\nabla_{\mathcal{D}} v - \nabla \varphi\|_{L^2(\Omega)}^2.
    \end{equation}
    Then  \(I_{\mathcal{D}} \varphi\) is unique, \(I_{\mathcal{D}} : H^1(\Omega) \rightarrow X_{\mathcal{D}}\) is linear and
    \begin{equation}\label{eq-lem:linear interpolator prop}
         \max\big\{\| \Pi_{\mathcal{D}} I_{\mathcal{D}}  \varphi - \varphi \|_{L^2(\Omega)}, \| \nabla_{\mathcal{D}} I_{\mathcal{D}} \varphi - \nabla \varphi \|_{L^2(\Omega)} \big\} \leq  \sqrt{2}S_{\mathcal{D}}(\varphi)\leq   \sqrt{2}h_{\mathcal{D}} \|\varphi\|_{W_s}.
    \end{equation}
\end{lemma}
\begin{proof}
    Let \(V = \{(\Pi_{\mathcal{D}}w , \nabla_{\mathcal{D}}w): w\in X_{\mathcal{D}}\}\)
    and \(\mathcal{P}: L^2(\Omega) \times L^2(\Omega)^2 \rightarrow V \) be the orthogonal projection.
    Since \(|\int_{\Omega}\Pi_{\mathcal{D}} \cdot|  + \|\nabla_{\mathcal{D}} \cdot\|_{L^2(\Omega)}\) is a norm on \(X_{\mathcal{D}}\), 
    for any \(z \in V\) there exists a unique \(\mathcal{R}z \in X_{\mathcal{D}}\) such that \((\Pi_{\mathcal{D}} \mathcal{R}z, \nabla_{\mathcal{D}} \mathcal{R}z) = z\).
    This defines a linear continuous mapping \(\mathcal{R}: V \rightarrow X_{\mathcal{D}}\)
    and \eqref{eq-lem:linear interpolator def} shows that \(I_{\mathcal{D}} \varphi = \mathcal{R} \circ \mathcal{P}(\varphi, \nabla \varphi)\) for all \(\varphi \in H^1(\Omega)\).
    Hence, \(I_{\mathcal{D}}\varphi\) is uniquely defined and \(I_{\mathcal{D}}\) is linear continuous.

    Taking \(v \in X_{\mathcal{D}}\) that realises the minimum that defines \(S_{\mathcal{D}}(\varphi)\) and using the definition of \(I_{\mathcal{D}}\) shows that 
    \begin{equation*}
        \|\Pi_{\mathcal{D}} I_{\mathcal{D}} \varphi - \varphi \|_{L^2(\Omega)}^2 + \|\nabla_{\mathcal{D}} I_{\mathcal{D}} \varphi - \nabla \varphi\|_{L^2(\Omega)}^2 \leq \|\Pi_{\mathcal{D}} v - \varphi \|_{L^2(\Omega)}^2 + \|\nabla_{\mathcal{D}} v - \nabla \varphi\|_{L^2(\Omega)}^2 \leq 2S_{\mathcal{D}}(\varphi)^2.
    \end{equation*}
    Then the property \eqref{eq-lem:linear interpolator prop} follows from this and Definition \ref{def:GD space size} of the space size.
\end{proof}
Now we define a non-linear spatial interpolator.
\begin{lemma}\label{lem:existence of NL interpolator}
    For all \(u \in W^{1,\infty}(\Omega)\),
    there exists a unique \(P_{\mathcal{D}} u \in X_{\mathcal{D}}\) satisfying
    \begin{subequations}\label{eq-lem:existence of NL interpolator}
        \begin{align}
        \int_{\Omega} \Pi_{\mathcal{D}} P_{\mathcal{D}} u & = \int_{\Omega} u,\label{eq-lem:existence of NL interpolator assumption} \\
        \bigg\langle \frac{\nabla_{\mathcal{D}} P_{\mathcal{D}} u}{ |\nabla_{\mathcal{D}} P_{\mathcal{D}} u|_{\rho}}, \nabla_{\mathcal{D}}w  \bigg\rangle 
        & =
        \bigg\langle  \frac{\nabla  u}{ |\nabla  u|_{\rho} }, \nabla_{\mathcal{D}} w  \bigg\rangle
        \quad 
        \forall w  \in X_{\mathcal{D}}.\label{eq:existence of NL interpolator lemma 2}
        \end{align}
    \end{subequations} 
\end{lemma}

\begin{proof}
    Let \(\ell_u: \nabla_{\mathcal{D}} (X_{\mathcal{D}}) \rightarrow \mathbb{R}\) be the linear form defined by the right-hand side of \eqref{eq:existence of NL interpolator lemma 2}.
    Consider the minimisation problem
    \begin{equation*}
        \min_{v\in X_{\mathcal{D}}} \int_{\Omega} |\nabla_{\mathcal{D}}v|_{\rho} - \ell_u(\nabla_{\mathcal{D}}v) =: \min_{v\in X_{\mathcal{D}}} E(\nabla_{\mathcal{D}}v).
    \end{equation*}
    Then the derivative of \(E(\nabla_{\mathcal{D}} \cdot)\) is: for \(w \in X_{\mathcal{D}}\),
    \begin{equation*}
        DE(\nabla_{\mathcal{D}}v)[\nabla_{\mathcal{D}}w] 
        := 
        \lim_{\alpha \rightarrow 0} \frac{E(\nabla_{\mathcal{D}}v+\alpha \nabla_{\mathcal{D}}w) - E(\nabla_{\mathcal{D}}v)}{\alpha} = \int_{\Omega} \frac{  \nabla_{\mathcal{D}} v }{ |\nabla_{\mathcal{D}}v|_{\rho}}\cdot \nabla_{\mathcal{D}}w - \ell_u(\nabla_{\mathcal{D}}w).
    \end{equation*}
    Since \(E\) is strictly convex and \(X_{\mathcal{D}}\) is finite dimensional, there exists a unique minimiser in \(\nabla_{\mathcal{D}}(X_{\mathcal{D}})\).
    Let \(\nabla_{\mathcal{D}} u_{\mathcal{D}}\) realise the minimisation of \(E\).
    Then for all \(w \in X_{\mathcal{D}}\),
    \begin{equation}\label{eq-proof-lem:existence of NL interpolator 1}
        0 = DE(\nabla_{\mathcal{D}}u_{\mathcal{D}})[\nabla_{\mathcal{D}}w] =
        \int_{\Omega}\frac{  \nabla_{\mathcal{D}} u_{\mathcal{D}}}{ |\nabla_{\mathcal{D}}u_{\mathcal{D}}|_{\rho}}\cdot \nabla_{\mathcal{D}}w - \ell_u(\nabla_{\mathcal{D}}w).
    \end{equation}
    Note that, for any \(C \in \mathbb{R}\), \(\nabla_{\mathcal{D}}(u_{\mathcal{D}} + C\bm{1}_{\mathcal{D}})\) is also a minimizer of \(E\) 
    and hence satisfies \eqref{eq-proof-lem:existence of NL interpolator 1}.
    Choosing 
    \[C = C(u_{\mathcal{D}}) := \frac{1}{|\Omega|_2} \bigg(\int_{\Omega} u - \int_{\Omega} \Pi_{\mathcal{D}} u_{\mathcal{D}} \bigg)\]
    leads to 
    \begin{equation*}
        \int_{\Omega} \Pi_{\mathcal{D}} (u_{\mathcal{D}} + C(u_{\mathcal{D}}) \bm{1}_{\mathcal{D}}) = \int_{\Omega} u.
    \end{equation*}
    Thus \(\overline{u_{\mathcal{D}}} := u_{\mathcal{D}} + C(u_{\mathcal{D}}) \bm{1}_{\mathcal{D}}\) satisfies \eqref{eq-lem:existence of NL interpolator}.
    Now suppose that there exits \(\overline{v_{\mathcal{D}}} \in X_{\mathcal{D}}\) such that  \(\overline{v_{\mathcal{D}}} \) satisfies \eqref{eq-lem:existence of NL interpolator}.
    We have 
    \begin{align*}
         \int_{\Omega} \Pi_{\mathcal{D}}\overline{u_{\mathcal{D}}} - \Pi_{\mathcal{D}}\overline{v_{\mathcal{D}}} = 0
         \quad \text{and} \quad 
         \bigg\langle \frac{\nabla_{\mathcal{D}} \overline{u_{\mathcal{D}}}}{ |\nabla_{\mathcal{D}} \overline{u_{\mathcal{D}}}|_{\rho}} - \frac{\nabla_{\mathcal{D}} \overline{v_{\mathcal{D}}}}{ |\nabla_{\mathcal{D}}\overline{v_{\mathcal{D}}}|_{\rho}}, \nabla_{\mathcal{D}}w  \bigg\rangle 
         = 0
        \quad 
        \forall w  \in X_{\mathcal{D}}. 
    \end{align*}
    This and \eqref{eq-lem:aux 0 3} imply that 
    \begin{equation*}
        \int_{\Omega} \Pi_{\mathcal{D}}(\overline{u_{\mathcal{D}}} - \overline{v_{\mathcal{D}}}) = 0 
        \quad \text{and} \quad \nabla_{\mathcal{D}} (\overline{u_{\mathcal{D}}} - \overline{v_{\mathcal{D}}}) = 0.
    \end{equation*}
    Since \eqref{eq-def:GD definition norm} is a norm on \(X_{\mathcal{D}}\), we have \(\overline{u_{\mathcal{D}}} - \overline{v_{\mathcal{D}}} = 0\). 
    Thus \(\overline{u_{\mathcal{D}}}\) is unique.
    Letting \(P_{\mathcal{D}}u := \overline{u_{\mathcal{D}}} \) completes the proof.
\end{proof}

This following lemma concerns with the error estimates of \(P_{\mathcal{D}}\).
\begin{lemma}\label{lem:NL interpolation error estimate}
    Suppose that Assumption \eqref{def:GDM Error Estimate Theorem assumption} holds.
    Let \(P_{\mathcal{D}}\) be the non-linear spatial interpolator in the sense of Lemma \ref{lem:existence of NL interpolator}.
    Let  \(u \in   W^{2, \infty}(\Omega) \cap W_s\).
    Then,  
    \begin{subequations}
        \begin{align}
            &\|\nabla_{\mathcal{D}} P_{\mathcal{D}} u\|_{L^{\infty}(\Omega)} 
            \lesssim 1,\label{eq-lem:NL interpolation error estimate 1}\\
            &\|\nabla_{\mathcal{D}} P_{\mathcal{D}} u - \nabla u\|_{L^2(\Omega)} 
            \lesssim  h_{\mathcal{D}},\label{eq-lem:NL interpolation error estimate 2}\\
            &\|\Pi_{\mathcal{D}} P_{\mathcal{D}} u - u\|_{L^2(\Omega)} \lesssim  h_{\mathcal{D}}. \label{eq-lem:NL interpolation error estimate 4}
        \end{align}
\end{subequations}
\end{lemma}  
\begin{remark}
    The hidden constants in the estimates \eqref{eq-lem:NL interpolation error estimate 1}, \eqref{eq-lem:NL interpolation error estimate 2}, and \eqref{eq-lem:NL interpolation error estimate 4} depend on \(u\),
    in particular,
    \(\|u\|_{W_s}\), \(\|u\|_{H^2(\Omega)}\), and \(\Lambda^{(1)}\) (see \eqref{eq-proof-lem::NL interpolation error estimate constant 1} below).
    Since the exact solution \(\overline{u}\) is assumed to be Lipschitz-continuous \([0,T] \rightarrow  W^{1, \infty}(\Omega) \cap H^2(\Omega) \cap W_s\), 
    for any \(t \in [0,T]\),
    the estimates \eqref{eq-lem:NL interpolation error estimate 1}, \eqref{eq-lem:NL interpolation error estimate 2}, and \eqref{eq-lem:NL interpolation error estimate 4}  hold for \(u = \overline{u}(t)\) with hidden constants that are independent of \(t\).
\end{remark}
\begin{proof}
    We first show an \(L^1(\Omega)\)-estimate.
    By the Cauchy-Schwarz inequality, we write
    \begin{equation}\label{eq-proof-lem::NL interpolation error estimate 1}
        \| \nabla {u} - \nabla_{\mathcal{D}} P_{\mathcal{D}} {u}\|_{L^1(\Omega)}  
        \leq
        \bigg(\int_{\Omega} \frac{|\nabla {u} - \nabla_{\mathcal{D}} P_{\mathcal{D}} {u}|^2}{ |\nabla_{\mathcal{D}} P_{\mathcal{D}} {u}|_{\rho}}\bigg)^{\frac{1}{2}}
        \bigg(\int_{\Omega}  |\nabla_{\mathcal{D}} P_{\mathcal{D}} {u}|_{\rho}\bigg)^{\frac{1}{2}}. 
    \end{equation}
    Let \(\mathcal{X} := I_{\mathcal{D}}u - {P_{\mathcal{D}}} {u} \).
    We have
    \begin{align*}
        A^2 
        & := \int_{\Omega} \frac{|\nabla {u} - \nabla_{\mathcal{D}} P_{\mathcal{D}} {u}|^2}{ |\nabla_{\mathcal{D}} P_{\mathcal{D}} {u}|_{\rho}} \\
        & = \int_{\Omega} \frac{(\nabla {u} - \nabla_{\mathcal{D}} P_{\mathcal{D}} {u}) \cdot \nabla_{\mathcal{D}} \mathcal{X}}{ |\nabla_{\mathcal{D}} P_{\mathcal{D}} {u}|_{\rho}} 
        + \int_{\Omega} \frac{(\nabla {u} - \nabla_{\mathcal{D}} P_{\mathcal{D}} {u}) \cdot (\nabla {u} - \nabla_{\mathcal{D}} I_{\mathcal{D}}u )}{ |\nabla_{\mathcal{D}} P_{\mathcal{D}} {u}|_{\rho}} 
         =: I_1  + I_2.
    \end{align*}

    Consider the term \(I_1\). By \eqref{eq:existence of NL interpolator lemma 2}, we write 
    \begin{align*}
        I_1 
        = 
            \int_{\Omega}  \nabla_{\mathcal{D}} \mathcal{X} \cdot\frac{ \nabla {u}}{ |\nabla_{\mathcal{D}} P_{\mathcal{D}} {u}|_{\rho}} 
        - 
            \int_{\Omega} \nabla_{\mathcal{D}} \mathcal{X} \cdot \frac{\nabla {u}}{ |\nabla  {u}|_{\rho}}.
    \end{align*}
    Then, under some algebraic manipulation, and uses of the Cauchy-Schwarz inequality,
    we write
    \begin{align*}
        I_1 
        & =
            \int_{\Omega} \nabla_{\mathcal{D}} \mathcal{X} \cdot\frac{ \nabla {u} \, (|\nabla {u}|^2 - |\nabla_{\mathcal{D}} P_{\mathcal{D}} {u}|^2)}{ |\nabla_{\mathcal{D}} P_{\mathcal{D}} {u}|_{\rho}  |\nabla  {u}|_{\rho} ( |\nabla_{\mathcal{D}} P_{\mathcal{D}} {u}|_{\rho} +  |\nabla  {u}|_{\rho})} \\
        & =  
            \int_{\Omega} (\nabla_{\mathcal{D}} \mathcal{X} \cdot \nabla {u})  
            \bigg(\frac{\nabla {u} - \nabla_{\mathcal{D}} P_{\mathcal{D}} {u}}{ |\nabla_{\mathcal{D}} P_{\mathcal{D}} {u}|_{\rho} |\nabla  {u}|_{\rho}} \cdot \frac{\nabla {u}+\nabla_{\mathcal{D}} P_{\mathcal{D}} {u}}{ ( |\nabla_{\mathcal{D}} P_{\mathcal{D}} {u}|_{\rho} +  |\nabla  {u}|_{\rho})}\bigg) \\
        & \leq   
            \int_{\Omega} 
            | \nabla_{\mathcal{D}} \mathcal{X} | 
            \bigg|  \frac{\nabla {u}}{ |\nabla  {u}|_{\rho}} \bigg| 
            \bigg|\frac{\nabla {u} - \nabla_{\mathcal{D}} P_{\mathcal{D}} {u}}{ |\nabla_{\mathcal{D}} P_{\mathcal{D}} {u}|_{\rho}} \bigg|  
            \bigg|\frac{\nabla {u}+\nabla_{\mathcal{D}} P_{\mathcal{D}} {u}}{ |\nabla_{\mathcal{D}} P_{\mathcal{D}} {u}|_{\rho} +  |\nabla  {u}|_{\rho} } \bigg|.
    \end{align*}
    Since \(\nabla u \in L^{\infty}(\Omega)\), we have
    \begin{equation}\label{eq-proof-lem::NL interpolation error estimate constant 1}
        \Lambda^{(1)} 
        :=
        \esssup_{x\in \Omega} \frac{|\nabla {u}(x)| }{ |\nabla {u}(x)|_{\rho}}  < 1.
    \end{equation} 
    Also notice the relation 
    \begin{equation*}
        \frac{|u + v|}{ |u|_{\rho} +  |v|_{\rho}} \
        \leq \frac{|u| + |v|}{|u| + |v|} 
        = 1, 
        \quad 
        \forall u, v \in \mathbb{R}^2.
    \end{equation*}
    Then the Cauchy-Schwarz inequality yields
    \begin{align*}
        I_1 
        \leq 
            \Lambda^{(1)} \int_{\Omega} |\nabla_{\mathcal{D}} \mathcal{X}| 
            \, \frac{|\nabla {u} - \nabla_{\mathcal{D}}P_{\mathcal{D}}{u}|}{ |\nabla_{\mathcal{D}}P_{\mathcal{D}}{u}|_{\rho}} 
        \leq 
            \Lambda^{(1)}\bigg(\int_{\Omega} \frac{|\nabla_{\mathcal{D}} \mathcal{X}|^2}{ |\nabla_{\mathcal{D}}P_{\mathcal{D}}{u}|_{\rho}} \bigg)^{\frac{1}{2}}  
            \bigg(\int_{\Omega} \frac{|\nabla {u} - \nabla_{\mathcal{D}}P_{\mathcal{D}}{u}|^2}{ |\nabla_{\mathcal{D}}P_{\mathcal{D}}{u}|_{\rho}} \bigg)^{\frac{1}{2}}.
    \end{align*}
    Set \[
    N(\cdot) :=  \bigg(\int_{\Omega} \frac{|\cdot|^2}{ |\nabla_{\mathcal{D}}P_{\mathcal{D}}{u}|_{\rho}} \bigg)^{\frac{1}{2}}, 
    \]
    which is a weighted norm.
    The above estimate for \(I_1\) can be rewritten in terms of \(N(\nabla_{\mathcal{D}} \mathcal{X})\) and \(A\).
    Expanding \(\mathcal{X}\), introducing \(\nabla {u}\) and compensating in \(N\), then using the triangle inequality, we obtain 
    \begin{align*}
        I_1 
        \leq 
            \Lambda^{(1)} N( \nabla_{\mathcal{D}} \mathcal{X}) A 
        \leq \, &
            \Lambda^{(1)} \big[ N(\nabla_{\mathcal{D}} I_{\mathcal{D}}u - \nabla  {u}) + N(\nabla {u} - \nabla_{\mathcal{D}}P_{\mathcal{D}}{u}) \big] A \\
        \leq \, &
            \Lambda^{(1)} \big( \rho^{-\frac{1}{2}} \|\nabla_{\mathcal{D}} I_{\mathcal{D}}u - \nabla {u}\|_{L^2(\Omega)} + A \big) A.
    \end{align*}

    Now consider the term \(I_2\). By the Cauchy-Schwarz inequality,
    \begin{equation*}
        I_2 
        \leq  
            \bigg(\int_{\Omega} \frac{|\nabla {u} - \nabla_{\mathcal{D}} P_{\mathcal{D}} {u}|^2 }{ |\nabla_{\mathcal{D}} P_{\mathcal{D}} {u}|_{\rho}} \bigg)^{\frac{1}{2}} 
            \bigg(\int_{\Omega} \frac{|\nabla {u} - \nabla_{\mathcal{D}} I_{\mathcal{D}}u|^2}{ |\nabla_{\mathcal{D}} P_{\mathcal{D}} {u}|_{\rho}} \bigg)^{\frac{1}{2}}  
        \leq 
            A \rho^{ - \frac{1}{2}}\|\nabla_{\mathcal{D}} I_{\mathcal{D}}u - \nabla {u}\|_{L^2(\Omega)}. 
    \end{equation*}
    Therefore, combining the estimates on \(I_1\) and \(I_2\) leads to
    \begin{equation*}
        A^2 
        \leq
            \Lambda^{(1)}A \big( \rho^{-\frac{1}{2}} \|\nabla_{\mathcal{D}} I_{\mathcal{D}}u - \nabla {u}\|_{L^2(\Omega)} + A \big) + A \rho^{-\frac{1}{2}} \|\nabla_{\mathcal{D}} I_{\mathcal{D}}u - \nabla {u}\|_{L^2(\Omega)},
    \end{equation*}
    which implies 
    \begin{equation*} 
        A 
        \leq 
            \rho^{-\frac{1}{2}} \frac{1 + \Lambda^{(1)}}{1 - \Lambda^{(1)}} \|\nabla_{\mathcal{D}} I_{\mathcal{D}}u - \nabla {u} \|_{L^2(\Omega)}. 
    \end{equation*}
    Then it follows from \eqref{eq-lem:linear interpolator prop} that 
    \begin{equation}\label{eq-proof-lem::NL interpolation error estimate 3}
        A 
        \leq
            \rho^{-\frac{1}{2}} \frac{1+\Lambda^{(1)}}{ 1-\Lambda^{(1)}} \| {u} \|_{W_s} h_{\mathcal{D}} . 
    \end{equation}
    Now we show that \(\int_{\Omega} |\nabla_{\mathcal{D}}P_{\mathcal{D}} {u}|_{\rho} \) is bounded. 
    Let \(w = P_{\mathcal{D}} {u}\) in \eqref{eq:existence of NL interpolator lemma 2},
    \begin{align*}
        \int_{\Omega} \frac{|\nabla_{\mathcal{D}}P_{\mathcal{D}}{u}|^2}{ |\nabla_{\mathcal{D}}P_{\mathcal{D}}{u}|_{\rho} } 
        =
            \int_{\Omega} \frac{\nabla {u} \cdot \nabla_{\mathcal{D}}P_{\mathcal{D}}{u}}{ |\nabla {u}|_{\rho} }  
        \leq 
            \Lambda^{(1)} \int_{\Omega}|\nabla_{\mathcal{D}}P_{\mathcal{D}}{u}|.
    \end{align*} 
    We infer that
    \begin{equation*}
                \int_{\Omega} |\nabla_{\mathcal{D}}P_{\mathcal{D}}{u}|_{\rho} 
         = 
            \int_{\Omega} \frac{|\nabla_{\mathcal{D}}P_{\mathcal{D}}{u}|^2}{ |\nabla_{\mathcal{D}}P_{\mathcal{D}}{u}|_{\rho}} + \int_{\Omega}\frac{\rho^2}{ |\nabla_{\mathcal{D}}P_{\mathcal{D}}{u}|_{\rho}} 
        \leq 
            \Lambda^{(1)} \int_{\Omega} |\nabla_{\mathcal{D}}P_{\mathcal{D}}{u}|_{\rho} + \rho|\Omega|_2,
    \end{equation*}
    where the first inequality holds by bounding the integrand of the second summand by \(\rho\).
    Hence, 
    \begin{equation}\label{eq-proof-lem::NL interpolation error estimate 4}
        \bigg(\int_{\Omega}  |\nabla_{\mathcal{D}}P_{\mathcal{D}}{u}|_{\rho} \bigg)^{\frac{1}{2}}
        \leq
            \Big( \frac{\rho|\Omega|_2}{1-\Lambda^{(1)}}\Big)^{\frac{1}{2}}.
    \end{equation}

    From \eqref{eq-proof-lem::NL interpolation error estimate 1}, combining estimates \eqref{eq-proof-lem::NL interpolation error estimate 3} and \eqref{eq-proof-lem::NL interpolation error estimate 4} yields 
    \begin{equation*}
        \| \nabla_{\mathcal{D}} P_{\mathcal{D}} {u} - \nabla {u} \|_{L^1(\Omega)} 
        \leq
            |\Omega|_2^{\frac{1}{2}} \frac{1 + \Lambda^{(1)}}{(1-\Lambda^{(1)})^{3/2}} \|{u}\|_{W_s} h_{\mathcal{D}}. 
    \end{equation*} 
    
    \noindent\textbf{Proof of \eqref{eq-lem:NL interpolation error estimate 1}.}
    Let  \(\mathcal{T}\) be the mesh from (iii) in Definition \ref{def:GD definition} .
    From \eqref{eq-proof-lem::NL interpolation error estimate 3}, for each cell \(\tau \in \mathcal{T}\),
    \begin{equation}\label{eq-proof-lem::NL interpolation error estimate 3 1}
        \theta_{\mathcal{D}} :=  \int_{\tau} \frac{|\nabla {u} - \nabla_{\mathcal{D}} P_{\mathcal{D}} {u} |^2}{ |\nabla_{\mathcal{D}} P_{\mathcal{D}} {u}|_{\rho}} \leq A^2 
        \leq 
            \Big(\frac{1 + \Lambda^{(1)}}{1 - \Lambda^{(1)}} \rho^{-\frac{1}{2}} \| {u}\|_{W_s} h_{\mathcal{D}}\Big)^2.
    \end{equation}
    Under some algebraic manipulation, the following can be obtained 
     \begin{align*}
        \int_{\tau} | \nabla_{\mathcal{D}} P_{\mathcal{D}} {u} |
        &\leq  
            \int_{\tau} \frac{|\nabla_{\mathcal{D}} P_{\mathcal{D}} {u}|^2}{ |\nabla_{\mathcal{D}} P_{\mathcal{D}} {u}|_{\rho}} + \frac{\rho^2}{ |\nabla_{\mathcal{D}} P_{\mathcal{D}} {u}|_{\rho}} \\
        &\leq  
            \int_{\tau} \frac{|\nabla_{\mathcal{D}} P_{\mathcal{D}} {u}|^2}{ |\nabla_{\mathcal{D}} P_{\mathcal{D}} {u}|_{\rho}} + \rho |\tau|_2 \\
        & =  
            \int_{\tau} \frac{(|\nabla {u} | - |\nabla_{\mathcal{D}} P_{\mathcal{D}} {u} |)^2}{ |\nabla_{\mathcal{D}} P_{\mathcal{D}} {u}|_{\rho}} 
            + \int_{\tau} \frac{ 2 |\nabla {u} | |\nabla_{\mathcal{D}} P_{\mathcal{D}} {u}|}{|\nabla_{\mathcal{D}} P_{\mathcal{D}} {u}|_{\rho}}  - \int_{\tau} \frac{|\nabla {u}|^2}{ |\nabla_{\mathcal{D}} P_{\mathcal{D}} {u}|_{\rho}} + \rho |\tau|_2 \\
        & \leq 
            \int_{\tau} \frac{|\nabla {u} - \nabla_{\mathcal{D}} P_{\mathcal{D}} {u} |^2}{ |\nabla_{\mathcal{D}} P_{\mathcal{D}} {u}|_{\rho}} 
            + \int_{\tau} \frac{ 2 |\nabla {u}| |\nabla_{\mathcal{D}} P_{\mathcal{D}} {u}|}{  |\nabla_{\mathcal{D}} P_{\mathcal{D}} {u}|_{\rho}} + \rho|\tau|_2\\
        & \leq 
            \theta_{\mathcal{D}} + 2  \esssup_{x \in \tau} |\nabla {u}(x)| |\tau|_2 + \rho |\tau|_2.
    \end{align*} 
    By the discrete Lebesgue estimates \cite[Lemma 1.25]{Pietro2020TheHH}
    and (iii) in Definition \ref{def:GD definition},
    \begin{equation}\label{eq-proof-lem::NL interpolation error estimate 6}
        \| \nabla_{\mathcal{D}} P_{\mathcal{D}} {u} \|_{L^{\infty}(\tau)} \
        \lesssim \frac{1}{|\tau|_2}\int_{\tau}|\nabla_{\mathcal{D}} P_{\mathcal{D}} u|
        \lesssim
            \frac{\theta_{\mathcal{D}} }{|\tau|_2} + 1. 
    \end{equation}
    We recall that, for any \(\tau \in \mathcal{T}\), \(h_{\mathcal{T}}^2 \lesssim |\tau|_2\).
    Involving \ref{eq-proof-lem::NL interpolation error estimate 3 1} and \(h_{\mathcal{D}} \lesssim h_{\mathcal{T}}\) by the assumption \ref{eq-def:GDM Error Estimate Theorem assumption 2},
    this leads to
    \begin{equation}\label{eq-proof-lem::NL interpolation error estimate 7}
        \frac{\theta_{\mathcal{D}}}{|\tau|_2} 
        \lesssim 
            \frac{h_{\mathcal{D}}^2}{|\tau|_2} 
        \lesssim 
            \frac{h_{\mathcal{T}}^2}{|\tau|_2} 
        \lesssim 1. 
    \end{equation}
    Therefore, combining the estimates \eqref{eq-proof-lem::NL interpolation error estimate 6} and \eqref{eq-proof-lem::NL interpolation error estimate 7} implies
    \begin{equation*}
        \| \nabla_{\mathcal{D}} P_{\mathcal{D}} {u} \|_{L^{\infty}(\tau)} 
        \lesssim 
            1, 
        \quad 
        \forall \tau \in \mathcal{T}.
    \end{equation*}
    Then \eqref{eq-lem:NL interpolation error estimate 1} follows.

    \noindent \textbf{Proof of  \eqref{eq-lem:NL interpolation error estimate 2}.}
    Use \eqref{eq-lem:NL interpolation error estimate 1} to write
    \begin{equation*}
        \| \nabla {u} - \nabla_{\mathcal{D}} P_{\mathcal{D}} {u} \|_{L^2({\Omega})}^2  
        \lesssim 
            \big(\rho + \|\nabla_{\mathcal{D}} P_{\mathcal{D}} {u}\|_{L^{\infty}(\Omega)} \big)  \int_{\Omega} \frac{|\nabla {u} - \nabla_{\mathcal{D}} P_{\mathcal{D}} {u} |^2}{|\nabla_{\mathcal{D}} P_{\mathcal{D}} {u}|_{\rho}}.
    \end{equation*}
    Using the estimate \eqref{eq-proof-lem::NL interpolation error estimate 3} on the integral and \eqref{eq-lem:NL interpolation error estimate 1} concludes the proof of \eqref{eq-lem:NL interpolation error estimate 2}.
    
    \noindent \textbf{Proof of \eqref{eq-lem:NL interpolation error estimate 4}.}
    We proceed using the Aubin-Nitsche trick.
    Denote \(F(\cdot) := \cdot / |\cdot|_{\rho}\). 
    
Let \(g \in L^2(\Omega)\) with \(\|g\|_{L^2({\Omega})} \leq 1\) and \(\int_{\Omega} g = 0\). 
Let \(\varphi_g \in H^1(\Omega)\) with \(\int_{\Omega} \varphi_g = 0\) be the solution of 
    \begin{subequations}
        \begin{equation}\label{eq-proof-lem::NL interpolation error estimate Aubin trick 3a}
            \begin{aligned}
                -\text{div} (DF(\nabla u) \nabla \varphi_g) & = g \text{ in \(\Omega\)},\\
                \nabla \varphi_g \cdot \bm{n} & = 0 \text{ on \(\partial \Omega\)}.
            \end{aligned}
        \end{equation}
        Let \(\varphi_{g, \mathcal{D}} \in X_{\mathcal{D}}\) be the solution of \(\int_{\Omega} \Pi_{\mathcal{D}} \varphi_{g, \mathcal{D}} = 0\) and 
        \begin{equation}\label{eq-proof-lem::NL interpolation error estimate Aubin trick 3b}
            \int_{\Omega} DF(\nabla u) \nabla_{\mathcal{D}} \varphi_{g, \mathcal{D}} \cdot \nabla_{\mathcal{D}} v_{\mathcal{D}} 
            =
                \int_{\Omega} g \, \Pi_{\mathcal{D}} v_{\mathcal{D}}, 
            \quad
            \forall v_{\mathcal{D}} \in X_{\mathcal{D}}. 
        \end{equation}     
    \end{subequations} 
    By \(\nabla u \in W^{1, \infty}(\Omega)\) and \eqref{eq-lem:aux 2},  \(DF(\nabla u)\) is Lipschitz-continuous on \(\Omega\), so $\varphi_g\in H^2(\Omega)$.
    By \cite[Theorem 3.30]{Droniou2018TheGD},  we infer
    \begin{equation}\label{eq-proof-lem::NL interpolation error estimate Aubin trick 4}
        \| \nabla_{\mathcal{D}} \varphi_{g, \mathcal{D}} - \nabla \varphi_g \|_{L^2(\Omega)} 
        \lesssim
            \|\varphi_g \|_{H^2(\Omega)} h_{\mathcal{D}} 
        \lesssim
            \|g\|_{L^2(\Omega)} h_{\mathcal{D}} . 
    \end{equation}
    Since \(DF(\nabla u)\) is uniformly elliptic, we have
    \begin{equation*}
        \|\nabla \varphi_g\|_{L^2(\Omega)}  \lesssim  \|g\|_{L^2(\Omega)} \lesssim 1
    \end{equation*}
    and thus
    \begin{equation}\label{eq-proof-lem::NL interpolation error estimate Aubin trick 5}
        \|\nabla_{\mathcal{D}} \varphi_{g, \mathcal{D}}\|_{L^2(\Omega)} \leq \|\nabla_{\mathcal{D}} \varphi_{g, \mathcal{D}} - \nabla \varphi_g\|_{L^2(\Omega)} + \|\nabla \varphi_{g}\|_{L^2(\Omega)} \lesssim h_{\mathcal{D}} + 1 \lesssim 1.
    \end{equation}
    By \eqref{eq-lem:NL interpolation error estimate 2} and \eqref{eq-proof-lem::NL interpolation error estimate Aubin trick 5}, we have
    \begin{equation*}
        \int_{\Omega}DF(\nabla u)[\nabla u - \nabla_{\mathcal{D}} P_{\mathcal{D}} u] \cdot \nabla_{\mathcal{D}} \varphi_{g, \mathcal{D}} 
        \lesssim \|\nabla u - \nabla_{\mathcal{D}} P_{\mathcal{D}} u\|_{L^2{(\Omega)}} \|\nabla_{\mathcal{D}} \varphi_{g, \mathcal{D}}\|_{L^2{(\Omega)}} \lesssim h_{\mathcal{D}}.
    \end{equation*}
    The left-hand side can be written as 
    \begin{align*}
            & \int_{\Omega} DF(\nabla u)[\nabla u - \nabla_{\mathcal{D}} P_{\mathcal{D}} u] \cdot \nabla_{\mathcal{D}} \varphi_{g, \mathcal{D}}\\
            = \, & \int_{\Omega} DF(\nabla u)[\nabla u - \nabla_{\mathcal{D}} I_{\mathcal{D}} u] \cdot \nabla_{\mathcal{D}} \varphi_{g, \mathcal{D}} 
                - \int_{\Omega} DF(\nabla u)[\nabla_{\mathcal{D}}P_{\mathcal{D}} u - \nabla_{\mathcal{D}} I_{\mathcal{D}} u] \cdot \nabla_{\mathcal{D}} \varphi_{g, \mathcal{D}}\\
            = \, & \int_{\Omega} DF(\nabla u)[\nabla u - \nabla_{\mathcal{D}} I_{\mathcal{D}} u] \cdot \nabla_{\mathcal{D}} \varphi_{g, \mathcal{D}} 
                - \int_{\Omega} g \, (\Pi_{\mathcal{D}}P_{\mathcal{D}} u - u) - \int_{\Omega} g \, (u - \Pi_{\mathcal{D}}I_{\mathcal{D}} u ).
    \end{align*}
    Then rearranging terms and applying the Cauchy-Schwarz inequality leads to 
    \begin{align*}
            \int_{\Omega} g (u - \Pi_{\mathcal{D}} P_{\mathcal{D}} u) \lesssim \, & h_{\mathcal{D}} + \int_{\Omega} DF(\nabla u)[\nabla_{\mathcal{D}} I_{\mathcal{D}} u - \nabla u] \cdot \nabla_{\mathcal{D}} \varphi_{g, \mathcal{D}} + \int_{\Omega} g(u - \Pi_{\mathcal{D}} I_{\mathcal{D}} u) \\
            \lesssim \, & h_{\mathcal{D}} + \|\nabla_{\mathcal{D}} I_{\mathcal{D}} u - \nabla u\|_{L^2(\Omega)} \|\nabla_{\mathcal{D}} \varphi_{g, \mathcal{D}}\|_{L^2(\Omega)} 
            + \|g\|_{L^2(\Omega)} \|u - \Pi_{\mathcal{D}} I_{\mathcal{D}} u \|_{L^2(\Omega)}\\
            \lesssim \, & h_{\mathcal{D}} + S_{\mathcal{D}}(u)
            \lesssim  (1 + \|u \|_{W_s}) h_{\mathcal{D}} \lesssim h_{\mathcal{D}}.
    \end{align*}
    Letting 
    \begin{equation*}
        g = \frac{ u - \Pi_{\mathcal{D}} P_{\mathcal{D}} {u} }{\| {u} - \Pi_{\mathcal{D}} P_{\mathcal{D}} {u}\|_{L^2(\Omega)}} \in L^2(\Omega),
    \end{equation*}
    which satisfies \(\int_{\Omega} g = 0\) by definition of \(P_{\mathcal{D}}u\),
    completes the proof. 
\end{proof}
\begin{lemma}\label{lem:NL interpolation error estimate time derivative}
    Suppose that Assumption \eqref{def:GDM Error Estimate Theorem assumption} holds.
    Let \(P_{\mathcal{D}}\) be the non-linear spatial interpolator in the sense of Lemma \ref{lem:existence of NL interpolator}.
    Let \(u \in L^{\infty}(0,T; W^{2, \infty}(\Omega))\) be Lipschitz-continuous \([0,T] \rightarrow W^{1, \infty}(\Omega) \cap H^2(\Omega) \cap W_s\) 
    and \(\nabla u\) is  Lipschitz-continuous \([0,T] \rightarrow L^{\infty}(\Omega)\).
    Then \(P_{\mathcal{D}} u\) is Lipschitz-continuous \([0,T] \rightarrow X_{\mathcal{D}}\)
    and
    satisfies
    \begin{subequations}
            \begin{align}
        & \esssup_{(0,T)}\|\partial_t \big(\nabla_{\mathcal{D}} P_{\mathcal{D}} u\big) - \partial_t \big(\nabla u\big)\|_{L^2(\Omega)} \lesssim h_{\mathcal{D}}, \label{eq-lem:NL interpolation error estimate time derivative 1} \\
        & \esssup_{(0,T)}\|\partial_t\big( \Pi_{\mathcal{D}} P_{\mathcal{D}} u \big)- \partial_t u\|_{ L^2(\Omega)} \lesssim h_{\mathcal{D}}. \label{eq-lem:NL interpolation error estimate time derivative 2}
    \end{align}
    \end{subequations}
\end{lemma}
\begin{proof}
    Since \(u\) is Lipschitz-continuous \([0,T] \rightarrow W^{1, \infty}(\Omega) \cap H^2(\Omega) \cap W_s\), 
    then for any \(t \in (0,T)\), \(u(t)\) satisfies \eqref{eq:existence of NL interpolator lemma 2} and its interpolation properties \eqref{eq-lem:NL interpolation error estimate 1}, \eqref{eq-lem:NL interpolation error estimate 2}, and \eqref{eq-lem:NL interpolation error estimate 4} with hidden constant \(C_u\) that is independent of \(t\) and \(\mathcal{D}\).
    Denote  \(F(\cdot) := \cdot/ |\cdot|_{\rho}\).

    We first show that \(P_{\mathcal{D}}u\) is Lipschitz-continuous \([0,T] \rightarrow X_{\mathcal{D}}\).
    From \eqref{eq-lem:existence of NL interpolator assumption}, for any \(t_1, t_2 \in [0,T]\), 
    \begin{equation}\label{eq-proof-lem:NL interpolation error estimate time derivative 1}
        \bigg|\int_{\Omega} \Pi_{\mathcal{D}} (P_{\mathcal{D}} u(t_1) - P_{\mathcal{D}}u(t_2)) \bigg|
        = \bigg|\int_{\Omega} u(t_1) - u(t_2)\bigg|
        \leq  \| \partial_t u\|_{L^{\infty}(\Omega_T)}|\Omega|_2 |t_1 - t_2|.
    \end{equation}
    From \eqref{eq:existence of NL interpolator lemma 2},
    since \(\nabla u\) is Lipschitz-continuous \([0,T] \rightarrow L^{\infty}(\Omega)\),
    by \eqref{eq-lem:aux 0 2} and the Cauchy-Schwarz inequality,  
    we have,
    for any \(t_1, t_2 \in (0,T)\), 
    \begin{align*}
        &\big\langle F(\nabla_{\mathcal{D}} P_{\mathcal{D}} u(t_1)) - F(\nabla_{\mathcal{D}} P_{\mathcal{D}} u(t_2)),  \nabla_{\mathcal{D}} P_{\mathcal{D}} u(t_1) - \nabla_{\mathcal{D}} P_{\mathcal{D}} u(t_2) \big\rangle \\
        = \, &   \big\langle F(\nabla u(t_1)) - F(\nabla u (t_2)), \nabla_{\mathcal{D}} P_{\mathcal{D}} u(t_1) - \nabla_{\mathcal{D}} P_{\mathcal{D}} u(t_2) \big\rangle\\
        \leq \, &\|\partial_t (\nabla u)\|_{L^{\infty}(\Omega_T)} |\Omega|_2^{1/2} |t_1 - t_2| \|\nabla_{\mathcal{D}} P_{\mathcal{D}} u(t_1)- \nabla_{\mathcal{D}} P_{\mathcal{D}} u(t_2)\|_{L^2{(\Omega)}}.
    \end{align*}
    Since \(\sup_{(0,T)}\|\nabla_{\mathcal{D}} P_{\mathcal{D}} u\|_{L^{\infty}(\Omega)} \lesssim C_u\) by \eqref{eq-lem:NL interpolation error estimate 1},
    by \eqref{eq-lem:aux 0 3}, we have 
    \begin{align*}
            &\big\langle F(\nabla_{\mathcal{D}} P_{\mathcal{D}} u(t_1)) - F(\nabla_{\mathcal{D}} P_{\mathcal{D}} u(t_2)),  \nabla_{\mathcal{D}} P_{\mathcal{D}} u(t_1) - \nabla_{\mathcal{D}} P_{\mathcal{D}} u(t_2) \big\rangle \\
            \geq \, & \int_{\Omega}\bigg(1 - \frac{|\nabla_{\mathcal{D}} P_{\mathcal{D}} u(t_2)|}{ |\nabla_{\mathcal{D}} P_{\mathcal{D}} u(t_2)|_{\rho}}\bigg) \frac{|\nabla_{\mathcal{D}} P_{\mathcal{D}} u(t_1)- \nabla_{\mathcal{D}} P_{\mathcal{D}} u(t_2)|^2}{|\nabla_{\mathcal{D}} P_{\mathcal{D}} u(t_1)|_{\rho}} \\
            \geq \, &   \bigg(1 - \frac{C_u}{|C_u|_{\rho}}\bigg) \frac{1}{|C_u |_{\rho}} \|\nabla_{\mathcal{D}} P_{\mathcal{D}} u(t_1)- \nabla_{\mathcal{D}} P_{\mathcal{D}} u(t_2)\|_{L^2{(\Omega)}}^2.
    \end{align*}
    This implies that there exists a constant \(C\) that depends on \(C_u\) and \(\|\partial_t (\nabla u)\|_{L^{\infty}(\Omega_T)} \) such that 
    \begin{equation*}
        \|\nabla_{\mathcal{D}} P_{\mathcal{D}} u(t_1)- \nabla_{\mathcal{D}} P_{\mathcal{D}} u(t_2)\|_{L^2{(\Omega)}} \leq C |t_1 - t_2|.
    \end{equation*}
    Combining with \eqref{eq-proof-lem:NL interpolation error estimate time derivative 1},
    this gives a constant \(C\) independent of \(t \in (0,T)\) such that 
    \begin{equation*}
        \|P_{\mathcal{D}} u (t_1) - P_{\mathcal{D}} u (t_2)\|_{\mathcal{D}} \leq C |t_1 - t_2|.
    \end{equation*}
    Thus, by the Rademacher theorem, \(P_{\mathcal{D}} u\) is differentiable almost everywhere in \((0,T)\).

    \noindent\textbf{Proof of \eqref{eq-lem:NL interpolation error estimate time derivative 1}.}
    Taking the time derivative of  \eqref{eq:existence of NL interpolator lemma 2}, we arrive at 
    \begin{equation*}
            \int_{\Omega} DF(\nabla_{\mathcal{D}} P_{\mathcal{D}} u)[\partial_t (\nabla_{\mathcal{D}} P_{\mathcal{D}} u)]\cdot \nabla_{\mathcal{D}} w
            = \int_{\Omega} DF(\nabla u)[\partial_t (\nabla u)] \cdot \nabla_{\mathcal{D}} w 
            \quad \forall w \in X_{\mathcal{D}}.
    \end{equation*}
    This implies that 
    \begin{equation}\label{eq-proof-lem:NL interpolation error estimate time derivative 1 11}
        \begin{aligned}
           & \int_{\Omega} DF(\nabla_{\mathcal{D}} P_{\mathcal{D}} u)[\partial_t (\nabla u) - \partial_t (\nabla_{\mathcal{D}} P_{\mathcal{D}} u)]\cdot \nabla_{\mathcal{D}} w \\
            = \, & \int_{\Omega} \big(DF(\nabla_{\mathcal{D}} P_{\mathcal{D}} u) - DF(\nabla u)\big)[\partial_t (\nabla u)] \cdot \nabla_{\mathcal{D}} w.
        \end{aligned}
    \end{equation}
    Since \(\esssup_{(0,T)}\|\partial_t(\nabla u)\|_{L^{\infty}(\Omega)}\) exists by assumption and \(\sup_{(0,T)}\|\nabla_{\mathcal{D}} P_{\mathcal{D}} u\|_{L^{\infty}(\Omega)} \lesssim C_u\) by \eqref{eq-lem:NL interpolation error estimate 1}, 
    by \eqref{eq-lem:aux 2} and \eqref{eq-lem:NL interpolation error estimate 2}, 
    we find a constant \(C_1\) such that
        \begin{equation}\label{eq-proof-lem:NL interpolation error estimate time derivative 1 2}
        \begin{aligned}
            & \bigg| \int_{\Omega} \big(DF(\nabla_{\mathcal{D}} P_{\mathcal{D}} u) - DF(\nabla u)\big)[\partial_t (\nabla u)] \cdot \nabla_{\mathcal{D}} w \bigg| \\
            \leq \, & C_1 \int_{\Omega} |\nabla_{\mathcal{D}} P_{\mathcal{D}} u - \nabla u| |\partial_t (\nabla u)| |\nabla_{\mathcal{D}}w|\\
            \leq \, & C_1 \|\partial_t (\nabla u) \|_{L^{\infty}(\Omega_T)}\|\nabla_{\mathcal{D}} P_{\mathcal{D}} u - \nabla u \|_{L^2(\Omega)} \| \nabla_{\mathcal{D}}w\|_{L^2(\Omega)}\\
            \leq \, & C_1 C_u h_{\mathcal{D}} \|\partial_t (\nabla u) \|_{L^{\infty}(\Omega_T)}\| \nabla_{\mathcal{D}}w\|_{L^2(\Omega)}.
        \end{aligned}
    \end{equation}

    Next, by \eqref{eq-lem:aux 1}, we find a constant \(C_2\) such that
    \begin{align*}
            & \|\partial_t(\nabla_{\mathcal{D}} P_{\mathcal{D}} u) - \partial_t (\nabla u)\|_{L^2(\Omega)}^2 \\
            \leq \, & C_2 \int_{\Omega} DF(\nabla_{\mathcal{D}} P_{\mathcal{D}} u)[\partial_t(\nabla_{\mathcal{D}} P_{\mathcal{D}} u) - \partial_t (\nabla u)] \cdot (\partial_t(\nabla_{\mathcal{D}} P_{\mathcal{D}} u) - \partial_t (\nabla u))\\
            =  \, & C_2 \int_{\Omega} DF(\nabla_{\mathcal{D}} P_{\mathcal{D}} u)[\partial_t(\nabla_{\mathcal{D}} P_{\mathcal{D}} u) - \partial_t (\nabla u)] \cdot (\partial_t(\nabla_{\mathcal{D}} P_{\mathcal{D}} u) - \partial_t (\nabla_{\mathcal{D}}I_{\mathcal{D}} u))\\
            & + C_2 \int_{\Omega} DF(\nabla_{\mathcal{D}} P_{\mathcal{D}} u)[\partial_t(\nabla_{\mathcal{D}} P_{\mathcal{D}} u) - \partial_t (\nabla u)] \cdot (\partial_t(\nabla_{\mathcal{D}} I_{\mathcal{D}} u) - \partial_t (\nabla u))
            =:  T_1 + T_2.
    \end{align*}
    Consider the term \(T_1\). 
    By \eqref{eq-proof-lem:NL interpolation error estimate time derivative 1 11} and \eqref{eq-proof-lem:NL interpolation error estimate time derivative 1 2} with \(w = \partial_t(P_{\mathcal{D}} u - I_{\mathcal{D}} u)\), 
    \begin{align*}
            T_1 \leq \, & C_1 C_2 C_u h_{\mathcal{D}} \|\partial_t (\nabla u) \|_{L^{\infty}(\Omega_T)} \|\partial_t(\nabla_{\mathcal{D}} P_{\mathcal{D}} u) - \partial_t (\nabla_{\mathcal{D}}I_{\mathcal{D}} u)\|_{L^2(\Omega)} \\
            \leq \, & C_1 C_2 C_u h_{\mathcal{D}} \|\partial_t (\nabla u) \|_{L^{\infty}(\Omega_T)}  \|\partial_t(\nabla_{\mathcal{D}} P_{\mathcal{D}} u) - \partial_t (\nabla u)\|_{L^2(\Omega)} \\
            & + C_1 C_2 C_u h_{\mathcal{D}} \|\partial_t (\nabla u) \|_{L^{\infty}(\Omega_T)}  \| \partial_t (\nabla_{\mathcal{D}}I_{\mathcal{D}} u) - \partial_t (\nabla u)\|_{L^2(\Omega)}  
            =:  T_{1, 1} + T_{1, 2}.
    \end{align*}
    By the Young inequality, for any \(\xi > 0\),
    \begin{equation*}
        T_{1, 1} \leq  C_1 C_2 C_u \|\partial_t (\nabla u) \|_{L^{\infty}(\Omega_T)} \bigg(\frac{\xi}{2}h_{\mathcal{D}}^2  + \frac{1}{2\xi} \|\partial_t(\nabla_{\mathcal{D}} P_{\mathcal{D}} u) - \partial_t (\nabla u)\|_{L^2(\Omega)}^2  \bigg).
    \end{equation*}
    Since the interpolator \(I_{\mathcal{D}}\) is linear, we write \(\partial_t (\nabla_{\mathcal{D}}I_{\mathcal{D}} u) = \nabla_{\mathcal{D}} I_{\mathcal{D}} (\partial_t u)\) and thus 
    \begin{equation}\label{eq-proof-lem:NL interpolation error estimate time derivative 1 3}
        \| \partial_t (\nabla_{\mathcal{D}}I_{\mathcal{D}} u) - \partial_t (\nabla u)\|_{L^2(\Omega)} = \| \nabla_{\mathcal{D}}I_{\mathcal{D}} (\partial_t u) - \nabla (\partial_t u)\|_{L^2(\Omega)} \leq S_{\mathcal{D}}(\partial_t u).
    \end{equation}
    This implies that 
    \begin{equation*}
        T_{1,2} 
        \leq C_1 C_2 C_u h_{\mathcal{D}} \|\partial_t (\nabla u) \|_{L^{\infty}(\Omega_T)} S_{\mathcal{D}}(\partial_t u) 
        \leq C_1 C_2 C_u h_{\mathcal{D}}^2 \|\partial_t (\nabla u) \|_{L^{\infty}(\Omega_T)} \|\partial_t u\|_{W_s}.
    \end{equation*}
    Consider now the term \(T_2\). 
    The estimate \eqref{eq-proof-lem:NL interpolation error estimate time derivative 1 3} gives a constant \(C_3\) such that 
    \begin{equation*}
        \begin{aligned}
            T_2 
            \leq \, & C_3 \|\partial_t(\nabla_{\mathcal{D}} P_{\mathcal{D}} u) - \partial_t (\nabla u)\|_{L^2(\Omega)}S_{\mathcal{D}}(\partial_t u)\\
            \leq \, & C_3 \|\partial_t(\nabla_{\mathcal{D}} P_{\mathcal{D}} u) - \partial_t (\nabla u)\|_{L^2(\Omega)}\|\partial_t u\|_{W_s} h_{\mathcal{D}}\\
            \leq \, & C_3 \|\partial_t u\|_{W_s} \bigg(\frac{\xi}{2}h_{\mathcal{D}}^2  + \frac{1}{2\xi} \|\partial_t(\nabla_{\mathcal{D}} P_{\mathcal{D}} u) - \partial_t (\nabla u)\|_{L^2(\Omega)}^2 \bigg).
        \end{aligned}
    \end{equation*}
    Collecting estimates for \(T_1\) and \(T_2\) and choosing a suitable \(\xi\) yields \eqref{eq-lem:NL interpolation error estimate time derivative 1}.
 
    \noindent\textbf{Proof of \eqref{eq-lem:NL interpolation error estimate time derivative 2}.}
    In this proof, we use the Aubin-Nitsche trick again.
    First, we recall the notion of the dual problem.
     Let \(g \in L^2(\Omega)\) with \(\|g\|_{L^2({\Omega})} \leq 1\) and \(\int_{\Omega} g = 0\). 
     Let \(\varphi_g \in H^2(\Omega)\) with \(\int_{\Omega} \varphi_g = 0\), be the solution of \eqref{eq-proof-lem::NL interpolation error estimate Aubin trick 3a},
     and \(\varphi_{g, \mathcal{D}} \in X_{\mathcal{D}}\) be the solution of 
    \eqref{eq-proof-lem::NL interpolation error estimate Aubin trick 3b}.
    Then \(\varphi_{g}\) and \(\varphi_{g, \mathcal{D}}\) satisfy \eqref{eq-proof-lem::NL interpolation error estimate Aubin trick 4} and \eqref{eq-proof-lem::NL interpolation error estimate Aubin trick 5}.
    
    By \eqref{eq-lem:aux 1.1}, the Cauchy-Schwarz inequality, \eqref{eq-lem:NL interpolation error estimate time derivative 1} and \eqref{eq-proof-lem::NL interpolation error estimate Aubin trick 5}, 
    there exists \(C\) independent of \({\mathcal{D}}\) such that 
    \begin{equation}\label{eq-proof-lem::NL spatial interpolation error estimate Aubin trick 6}
    \begin{aligned}
        &\int_{\Omega} DF(\nabla u)\big[\partial_t (\nabla u) - \partial_t (\nabla_{\mathcal{D}} P_{\mathcal{D}} u)\big] \cdot \nabla_{\mathcal{D}} \varphi_{g, \mathcal{D}}\\
    \leq \, & C\|\partial_t (\nabla u) - \partial_t (\nabla_{\mathcal{D}} P_{\mathcal{D}} u) \|_{L^2(\Omega)} \|\nabla_{\mathcal{D}} \varphi_{g, \mathcal{D}}\|_{L^2(\Omega)} \lesssim h_{\mathcal{D}}.
    \end{aligned}
    \end{equation}
    By \eqref{eq-proof-lem::NL interpolation error estimate Aubin trick 3b},
    we then write 
    \begin{align*}
            & \int_{\Omega} DF(\nabla u)\big[\partial_t (\nabla u) - \partial_t (\nabla_{\mathcal{D}} P_{\mathcal{D}} u)\big] \cdot \nabla_{\mathcal{D}} \varphi_{g, \mathcal{D}}\\
        = \, & \int_{\Omega} DF(\nabla u)\big[ \partial_t (\nabla u)  - \partial_t (\nabla_{\mathcal{D}} I_{\mathcal{D}} u) \big] \cdot  \nabla_{\mathcal{D}} \varphi_{g,\mathcal{D}}\\
            & - \int_{\Omega} DF(\nabla u)\big[ \partial_t (\nabla_{\mathcal{D}} P_{\mathcal{D}} u) - \partial_t (\nabla_{\mathcal{D}} I_{\mathcal{D}} u) \big] \cdot \nabla_{\mathcal{D}} \varphi_{g,\mathcal{D}} \\
            = \, & \int_{\Omega} DF(\nabla u)\big[ \partial_t (\nabla u)  - \partial_t (\nabla_{\mathcal{D}} I_{\mathcal{D}} u) \big] \cdot \nabla_{\mathcal{D}} \varphi_{g,\mathcal{D}}\\
            & - \int_{\Omega} g \big( \partial_t (\Pi_{\mathcal{D}} P_{\mathcal{D}} u) -  \partial_tu\big)
            + \int_{\Omega} g \big( \partial_t (\Pi_{\mathcal{D}} I_{\mathcal{D}} u) - \partial_t u\big).
    \end{align*}
    This, together with \eqref{eq-proof-lem::NL spatial interpolation error estimate Aubin trick 6}, imply that 
    \begin{align*}
            \int_{\Omega} g \big( \partial_tu - \partial_t (\Pi_{\mathcal{D}} P_{\mathcal{D}} u) \big)
            \lesssim \, &  h_{\mathcal{D}} 
            + 
            \int_{\Omega} g \big( \partial_t (\Pi_{\mathcal{D}} I_{\mathcal{D}} u) - \partial_t u\big)\\
            &+
            \int_{\Omega} DF(\nabla u)\big[\partial_t (\nabla_{\mathcal{D}} I_{\mathcal{D}} u)  - \partial_t (\nabla u) \big] \cdot \nabla_{\mathcal{D}} \varphi_{g,\mathcal{D}} \\
            \lesssim \, & h_{\mathcal{D}} + \|g\|_{L^2(\Omega)}\|\partial_t (\Pi_{\mathcal{D}} I_{\mathcal{D}} u) - \partial_t u\|_{L^2(\Omega)} \\
            & + \|\nabla u\|_{L^{\infty}(\Omega_T)} \|\partial_t (\nabla_{\mathcal{D}} I_{\mathcal{D}} u) - \partial_t (\nabla  u) \|_{L^2(\Omega} \|\nabla_{\mathcal{D}} \varphi_{g,\mathcal{D}}\|_{L^2(\Omega)} \\
            \lesssim \, & h_{\mathcal{D}} +  S_{\mathcal{D}}(\partial_t u)
            \lesssim   (1 + \|\partial_t u\|_{W_s})h_{\mathcal{D}}.
    \end{align*}
    Letting 
    \[g = \frac{ \partial_t u - \partial_t (\Pi_{\mathcal{D}} P_{\mathcal{D}} u) }{\| \partial_t u - \partial_t (\Pi_{\mathcal{D}} P_{\mathcal{D}} u)  \|_{L^2(\Omega)}}\]
    completes the proof
\end{proof}

    \section{Analysis of the scheme}\label{section:stability and consistency}
In this section, we analyse the scheme \eqref{eq-def:GS} by investigating:
\begin{itemize}
    \item Existence, uniqueness, and stability of the solution for the scheme in Lemma \ref{lem:GS existence and uniqueness}.
    \item Stability of the scheme in Lemma \ref{lem:GS stability}.
    \item Consistency of the scheme in Lemma \ref{lem:GS consistency}.
\end{itemize}

This following lemma concerns with the existence and uniqueness of the solution to the scheme \eqref{eq-def:GS}.
The existence follows from a topological degree argument and the uniqueness follows from the strong monotonicity of the discrete counterparts of the flux.
\begin{lemma}\label{lem:GS existence and uniqueness}
    There exists a unique solution to the scheme \eqref{eq-def:GS} and the solution satisfies 
    \begin{equation*}
       \max_{1\leq n\leq N} \| \Pi_{\mathcal{D}} u^{(n)}\|_{L^2(\Omega)} 
       + \|\nabla_{\mathcal{D}} u\|_{L^1(0,T;L^1(\Omega))}
       \lesssim 1.
    \end{equation*} 
\end{lemma}
\begin{proof}
    
    \noindent \textbf{Estimate for \(\Pi_{\mathcal{D}} u\).}
    Letting \(v = u^{(m)}\) in \eqref{eq-def:GS}, applying the Cauchy-Schwarz inequality and the Young inequality, the following can be obtained 
    \begin{equation}\label{eq-proof-lem:apriori estimate of GS 3}
        \begin{aligned}
           & (1 + \delta t^{(m)} \lambda) \| \Pi_{\mathcal{D}} u^{(m)}\|_{L^2(\Omega)}^2 
           + \delta t^{(m)}  \int_{\Omega} \frac{ |\nabla_{\mathcal{D}}u^{(m)}|^2 }{ |\nabla_{\mathcal{D}}u^{(m)}|_{\rho}} dx\\
           \leq \ & 
                \frac{1}{2}\| \Pi_{\mathcal{D}} u^{(m-1)}\|_{L^2(\Omega)}^2 + \frac{\delta t^{(m)} \lambda}{2} \|g\|_{L^2(\Omega)}^2 + \frac{1 + \delta t^{(m)} \lambda}{2}\| \Pi_{\mathcal{D}} u^{(m)}\|_{L^2(\Omega)}^2.
        \end{aligned}
    \end{equation}
    Absorbing the last term into the left-hand side
    and applying the discrete Grönwall inequality, 
    we find 
    \begin{equation}\label{eq-proof-lem:apriori estimate of GS 2}
            \max_{1 \leq n \leq N}\| \Pi_{\mathcal{D}} u^{(n)}\|_{L^2(\Omega)}^2 
            \leq 
                C(T, \lambda, g, \Pi_{\mathcal{D}} u^{(0)}).
    \end{equation}

    \noindent \textbf{Estimate for \(\nabla_{\mathcal{D}} u\).} 
    Applying the relation 
    \begin{equation}\label{eq:relation 2}
        \frac{|\mu|^2}{ |\mu|_{\rho}} =\frac{\rho^2 + |\mu|^2}{ |\mu|_{\rho}} - \frac{\rho^2}{|\mu|_{\rho}} \geq |\mu| - \rho, \quad \forall \mu \in \mathbb{R}^2.
    \end{equation}
    to the second term on the left-hand side of \eqref{eq-proof-lem:apriori estimate of GS 3}, absorbing the last term on the right-hand side to the left,
    summing over \(m\) from \(1\) to \(N\), and using the telescopic sum, we find 
    \begin{equation}\label{eq-proof-lem:apriori estimate of GS 4}
       \| \Pi_{\mathcal{D}} u \|_{L^2(0,T; L^2(\Omega))}^2  + \| \nabla_{\mathcal{D}} u\|_{L^1(0,T;L^1(\Omega))} 
        \leq 
            C(T, \Omega, \rho, \lambda, g, \Pi_{\mathcal{D}}u^{(0)}).
    \end{equation}

    \noindent\textbf{Existence of a unique solution.}
    The following proof proceeds by induction, it suffices to show that, for a given \(u^{(m-1)} \in X_{\mathcal{D}}\), there exists a unique \(u^{(m)} \in X_{\mathcal{D}}\) satisfying the gradient scheme.
    Endow \( X_{\mathcal{D}}\) with product \((\cdot, \cdot)\) and denote \(|\cdot|\) the corresponding norm.
    Let \(F^{(m)}: X_{\mathcal{D}} \rightarrow X_{\mathcal{D}}\) be such that, for all \(u,v \in X_{\mathcal{D}}\),
    \begin{equation*}
        (F^{(m)}(u),  v) 
        = 
            \int_{\Omega} (1 + \delta t^{(m)} \lambda ) \, \Pi_{\mathcal{D}}u \, \Pi_{\mathcal{D}}v 
            +  
            \int_{\Omega}\delta t^{(m)}\frac{\nabla_{\mathcal{D}} u  }{ |\nabla_{\mathcal{D}} u |_{\rho}} \cdot \nabla_{\mathcal{D}} v.
    \end{equation*}
    Set \(b^{(m)} \in X_{\mathcal{D}}\) such that, for all \(v\in X_{\mathcal{D}}\), 
    \( (b^{(m)},v) = \int_{\Omega} \Pi_{\mathcal{D}}u^{(m-1)} \, \Pi_{\mathcal{D}}v  + \delta t^{(m)}\lambda \, g \, \Pi_{\mathcal{D}}v \, dx.\)


    Since the non-linear term in \(F^{(m)}\) is continuous (see \eqref{eq-lem:aux 0 2}) and each linear term in \(F^{(m)}\) is continuous, then \(F^{(m)}\) is continuous.

    Existence of solutions for \eqref{eq-def:GS} is a consequence of \cite[Theorem D.1]{Droniou2018TheGD}. 
    Let \(\Phi^{(m)}:X_{\mathcal{D}} \times [0,1] \rightarrow X_{\mathcal{D}}\) be such that 
    for all \(u, v \in X_{\mathcal{D}}\), for any \(\zeta \in [0,1],\)
    \begin{equation}\label{eq-proof-lem:apriori estimate of GS homotopy}\begin{aligned}
        (\Phi^{(m)}(u, \zeta),  v) 
        = \, & 
            \int_{\Omega} (1 + \delta t^{(m)} \lambda ) \, \Pi_{\mathcal{D}}u \, \Pi_{\mathcal{D}}v \, dx   
            + \int_{\Omega} \delta t^{(m)} \, \frac{\nabla_{\mathcal{D}} u  \cdot \nabla_{\mathcal{D}} v}{\sqrt{\rho^2 + \zeta |\nabla_{\mathcal{D}} u |^2}} 
            \, dx \\
            & - \int_{\Omega} \Pi_{\mathcal{D}}u^{(m-1)} \, \Pi_{\mathcal{D}}v  - \int_{\Omega} \delta t^{(m)}\lambda \, g \, \Pi_{\mathcal{D}}v \, dx.
    \end{aligned}\end{equation}
    We proceed the proof by checking conditions of the theorem:
    \begin{itemize}
        \item In \eqref{eq-proof-lem:apriori estimate of GS homotopy}, all non-linear terms are obviously continuous, and it has no singularity for all \((u, \zeta) \in X_{\mathcal{D}} \times [0,1]\). 
        Hence, \(\Phi^{(m)}\) is continuous.
        \item When \(\zeta = 1\), we have \(\Phi^{(m)}(u,1) = F^{(m)}(u) - b\).
        \item Suppose that \(\Phi^{(m)}(u,\zeta) = 0\).  
        Letting \(v = u\) in \eqref{eq-proof-lem:apriori estimate of GS homotopy}, 
        applying the relations
        \begin{equation*}
            \frac{|\mu|^2}{\sqrt{\rho^2 + \zeta |\mu|^2}} \geq
            \frac{|\mu|^2}{ |\mu|_{\rho}}, \quad \forall \zeta \in [0,1],\, \mu \in \mathbb{R}^2;
        \end{equation*}
        and \eqref{eq:relation 2} to the second term on the right-hand side, 
        the following estimate can be obtained analogously to \eqref{eq-proof-lem:apriori estimate of GS 2} and \eqref{eq-proof-lem:apriori estimate of GS 4}
        \begin{equation*}\begin{aligned}
            & \| \Pi_{\mathcal{D}} u\|_{L^2(\Omega)}^2 
            + \delta t^{(m)}  \|\nabla_{\mathcal{D}} u\|_{L^1(\Omega)} 
            \lesssim \,  \| \Pi_{\mathcal{D}} u^{(m-1)}\|_{L^2(\Omega)}^2 +  1.
        \end{aligned}\end{equation*}
        This implies that \(\|u\|_{\mathcal{D}}\) is bounded uniformly with respect to \(\zeta\). 
        Therefore, there exists a sufficiently large \(R > 0\) such that \(\|u\|_{\mathcal{D}} \neq R\).
    \item When \(\zeta = 0\), \(\Phi^{(m)}(u, 0)\) is affine. 
    Since any solution to  \(\Phi^{(m)}(\cdot, 0)\) is bounded, 
    this implies that the linear part of \(\Phi^{(m)}(\cdot, 0)\) has a trivial kernel and hence is injective.
    Since \(X_{\mathcal{D}}\) is finite dimensional, it means that this linear part is an isomorphism and so \(\Phi^{(m)}(\cdot, 0)\) has a unique solution. 
    \end{itemize}
    Hence, there exists at least one \(u \in X_{\mathcal{D}}\) such that \(F^{(m)}(u) = b.\)

    The uniqueness follows from the strict monotonicity of \(F^{(m)}\): take distinct \(u, v\in X_{\mathcal{D}}\) and, in \ref{eq-lem:aux 0 3}, let \(u = \nabla_{\mathcal{D}} u\) and \(v = \nabla_{\mathcal{D}} v\), 
        \begin{align*}
             \big( F^{(m)}(u) - F^{(m)}(v), u-v \big)
             \geq \ & 
                (1 + \lambda \delta t^{(m)})\| \Pi_{\mathcal{D}} u - \Pi_{\mathcal{D}}v \|_{L^2(\Omega)}^2 \\
                & + \delta t^{(m)} \int_{\Omega} \left(1 - \frac{|\nabla_{\mathcal{D}} v|}{ |\nabla_{\mathcal{D}} v|_{\rho}}\right) \frac{|\nabla_{\mathcal{D}} u - \nabla_{\mathcal{D}}v|^2}{ |\nabla_{\mathcal{D}} u|_{\rho}} > 0.
        \end{align*}
    This shows that \(F^{(m)}(u) \neq F^{(m)}(v)\) thus \(F^{(m)}\) is injective. 
    So, the solution of \(F^{(m)}\) is unique for every \(m\).
\end{proof}

This lemma concerns the stability of the scheme, which stipulates that if the data of the problem is perturbed by a certain amount, then the solution of the perturbed problem should remain close to the solution of the non-perturbed problem. 
\begin{lemma}[Stability]\label{lem:GS stability}
    Suppose that \(u_1, u_2 \in X_{\mathcal{D}}^{M+1}\) satisfies: for any \(m = 1, \ldots, M,\)
    \begin{equation}\label{eq-lem:GS stability 1}
    \begin{aligned}
        & (1 + \delta t^{(m)} \lambda)\int_{\Omega} \big(\Pi_{\mathcal{D}} u_{1}^{(m)} - \Pi_{\mathcal{D}} u_{2}^{(m)}\big) \, \Pi_{\mathcal{D}}v   
         + \delta t^{(m)} \int_{\Omega} \bigg(\frac{\nabla_{\mathcal{D}}u_1^{(m)} }{ \big|\nabla_{\mathcal{D}}u_1^{(m)} \big|_{\rho} } - \frac{\nabla_{\mathcal{D}}u_2^{(m)} }{ \big|\nabla_{\mathcal{D}}u_2^{(m)} \big|_{\rho}} \bigg)\cdot \nabla_{\mathcal{D}}v\\
        & = 
            \int_{\Omega} \Pi_{\mathcal{D}} \big(u_{1}^{(m-1)} - u_{2}^{(m-1)}\big)  \Pi_{\mathcal{D}}v + \delta t^{(m)} R^{(m)}(v) \quad \forall v \in X_{\mathcal{D}},
    \end{aligned}
    \end{equation}
    and given data \(u_1(0)\) and \(u_2(0)\), 
    where \(R^{(m)}: X_{\mathcal{D}} \rightarrow \mathbb{R}\) is linear.
    Then,  for any \(m = 1, \ldots, M,\)
\begin{equation}\label{eq-lem:GS stability 2}\begin{aligned}
    & \| \Pi_{\mathcal{D}}(u_1^{(m)} - u_2^{(m)})\|_{L^2(\Omega)}^2 
    + \sum_{\ell = 1}^m \delta t_{\mathcal{D}}^{(\ell)}\| \Pi_{\mathcal{D}}(u_1^{(\ell)} - u_2^{(\ell)})\|_{L^2(\Omega)}^2\\
    & +  \sum_{\ell = 1}^m (\delta t^{(\ell)})^2 \| \delta_{\mathcal{D}}^{(\ell)} (  u_1 -  u_2  )\|_{L^2(\Omega)}^2
     + \sum_{\ell = 1}^m  \delta t^{(\ell)}  \int_{\Omega} \bigg(1 - \frac{\big|\nabla_{\mathcal{D}} u_2^{(\ell)}\big|}{ \big|\nabla_{\mathcal{D}} u_2^{(\ell)} \big|_{\rho}} \bigg)\frac{\big|\nabla_{\mathcal{D}} u_1^{(\ell)} - \nabla_{\mathcal{D}} u_2^{(\ell)} \big|^2}{ \big|\nabla_{\mathcal{D}} u_1^{(\ell)} \big|_{\rho}}\\
        \leq \ &   
            \| \Pi_{\mathcal{D}}(u_1^{(0)} - u_2^{(0)})\|_{L^2(\Omega)}^2 
            + \sum_{\ell = 1}^m \delta t^{(\ell)} R^{(\ell)} (u_1^{(\ell)} - u_2^{(\ell)}).
\end{aligned}
\end{equation}
\end{lemma}
\begin{proof}
Letting \(u_{12}^{(m)} := u_1^{(m)} - u_2^{(m)} \) and \(v = u_{12}^{(m)}\) in \eqref{eq-lem:GS stability 1},
then rearranging terms yields
\begin{equation*}
    \begin{aligned}
        &\big\langle  
        \Pi_{\mathcal{D}} u_{12}^{(m)} - \Pi_{\mathcal{D}} u_{12}^{(m-1)}, \Pi_{\mathcal{D}} u_{12}^{(m)}
        \big\rangle 
         + \lambda  \delta t^{(m)} \| \Pi_{\mathcal{D}}u_{12}^{(m)}\|_{L^2(\Omega)}^2\\ 
        & + \delta t^{(m)} \Bigg\langle  
        \frac{\nabla_{\mathcal{D}}u_1^{(m)} }{ \big|\nabla_{\mathcal{D}}u_1^{(m)} \big|_{\rho}} -  \frac{\nabla_{\mathcal{D}}u_2^{(m)} }{  \big|\nabla_{\mathcal{D}}u_2^{(m)} \big|_{\rho}}, \nabla_{\mathcal{D}}u_{12}^{(m)}\Bigg\rangle
        \leq 
            \delta t^{(m)} R^{(m)}(u_{12}^{(m)}).
    \end{aligned}
\end{equation*}
Apply the relation 
\begin{equation*}
    2(\mu-\nu)\mu = \mu^2 - \nu^2 + (\mu-\nu)^2, \quad \forall \mu,\nu \in \mathbb{R},
\end{equation*}
to the first term on the left-hand side 
\begin{equation}\label{eq-proof-lem:GS stability 1}
    \begin{aligned}
        & (1 +  2\lambda  \delta t^{(m)}) \| \Pi_{\mathcal{D}}u_{12}^{(m)}\|_{L^2(\Omega)}^2 -  \| \Pi_{\mathcal{D}}u_{12}^{(m-1)}\|_{L^2(\Omega)}^2 
        +  \| \Pi_{\mathcal{D}} u_{12}^{(m)} - \Pi_{\mathcal{D}} u_{12}^{(m-1)}\|_{L^2(\Omega)}^2\\
        &  +  2\delta t^{(m)} \Bigg\langle  
        \frac{\nabla_{\mathcal{D}}u_1^{(m)} }{ \big|\nabla_{\mathcal{D}}u_1^{(m)} \big|_{\rho}} -  \frac{\nabla_{\mathcal{D}}u_2^{(m)} }{ \big|\nabla_{\mathcal{D}}u_2^{(m)} \big|_{\rho}}, \nabla_{\mathcal{D}}u_{12}^{(m)} \Bigg\rangle
        \leq 
            2\delta t^{(m)} R^{(m)}(u_{12}^{(m)}).
    \end{aligned}
\end{equation}
Considering now the non-linear term on the left-hand side, in \eqref{eq-lem:aux 0 3}, 
letting \(u = \nabla_{\mathcal{D}} u_1^{(m)}, v = \nabla_{\mathcal{D}} u_2^{(m)}\), 
we obtain 
\begin{align*}
    \Bigg\langle  \frac{\nabla_{\mathcal{D}}u_1^{(m)} }{\big|\nabla_{\mathcal{D}}u_1^{(m)} \big|_{\rho}} -  \frac{\nabla_{\mathcal{D}}u_2^{(m)} }{ \big|\nabla_{\mathcal{D}}u_2^{(m)} \big|_{\rho}}, \nabla_{\mathcal{D}}u_{12}^{(m)} \Bigg\rangle
    \geq 
        \int_{\Omega} \bigg(1 - \frac{|\nabla_{\mathcal{D}} u_2^{(m)}|}{ \big|\nabla_{\mathcal{D}} u_2^{(m)}\big|_{\rho}}  \bigg)\frac{|\nabla_{\mathcal{D}} u_{12}^{(m)}|^2}{ \big|\nabla_{\mathcal{D}} u_1^{(m)} \big|_{\rho}}.
\end{align*}
Plugging this in \eqref{eq-proof-lem:GS stability 1} and rearranging the terms, 
we obtain 
\begin{align*}
    & (1 + 2\lambda \delta t^{(m)} )\| \Pi_{\mathcal{D}}u_{12}^{(m)} \|_{L^2(\Omega)}^2
    + (\delta t^{(m)})^2 \| \delta_{\mathcal{D}}^{(m)}(u_1 -  u_2)\|_{L^2(\Omega)}^2\\
    & +  2\delta t^{(m)} \int_{\Omega} \Bigg(1 - \frac{|\nabla_{\mathcal{D}} u_2^{(m)}|}{ |\nabla_{\mathcal{D}} u_2^{(m)}|_{\rho}}  \Bigg)\frac{|\nabla_{\mathcal{D}} u_{12}^{(m)} |^2}{ |\nabla_{\mathcal{D}} u_1^{(m)}|_{\rho}}
     \leq  
        \| \Pi_{\mathcal{D}}u_{12}^{(m - 1)}\|_{L^2(\Omega)}^2 + \delta t^{(m)}  R^{(m)}(u_{12}^{(m)}).
\end{align*} 
Then, apply the telescopic sum to get \eqref{eq-lem:GS stability 2}.
\end{proof}

The following lemma establishes an estimate for the consistency of the scheme.
Consistency of the scheme measures the error committed when plugging the interpolation of the exact solution into the scheme.
\begin{lemma}[Consistency]\label{lem:GS consistency}
    Suppose that Assumption \eqref{def:GDM Error Estimate Theorem assumption} holds.
    For \(1 \leq m, \ell \leq M\) and any \(v = (v^{(0)}, v^{(1)}, \ldots, v^{(M)})\in X_{\mathcal{D}}^{M+1}\),
    define 
    \begin{equation}\label{eq-lem:GS consistency 1}\begin{aligned}
             \mathcal{E}(\overline{u}(t^{(\ell)}), v^{(\ell)}) 
             := \, & 
                \int_{\Omega} \delta_{\mathcal{D}}^{(\ell)} P_{\mathcal{D}} \overline{u} \  \Pi_{\mathcal{D}} v^{(\ell)} 
                + \int_{\Omega} \frac{\nabla_{\mathcal{D}} P_{\mathcal{D}} \overline{u}(t^{(\ell)})}{ |\nabla_{\mathcal{D}} P_{\mathcal{D}} \overline{u}(t^{(\ell)})|_{\rho}} \cdot \nabla_{\mathcal{D}} v^{(\ell)}\\ 
                & + \lambda \int_{\Omega} \left[\Pi_{\mathcal{D}}P_{\mathcal{D}}\overline{u}(t^{(\ell)}) - g\right] \Pi_{\mathcal{D}} v^{(\ell)}.
    \end{aligned}\end{equation} 
    Then, for all \(\Lambda^{(2)} > 0\),
    \begin{equation}\label{eq-lem:GS consistency 2}
        \begin{aligned}
         \sum_{\ell = 1}^m \delta t^{(\ell)} \mathcal{E} (\overline{u}(t^{(\ell)}), v^{(\ell)})
        \lesssim \, & 
         \frac{1}{\Lambda^{(2)}}\big[(\delta t_{\mathcal{D}}^{\max})^2 
                + h_{\mathcal{D}}^2 \big] 
         +  \Lambda^{(2)} \sum_{\ell = 1}^m \delta t_{\mathcal{D}}^{(\ell)} \| \Pi_{\mathcal{D}} v^{(\ell)} \|_{L^2(\Omega)}^2\\
        &  + \|\nabla_{\mathcal{D}} v\|_{L^1(0,T; L^1(\Omega))} h_{\mathcal{D}}.
        \end{aligned}
    \end{equation}  
\end{lemma}
\begin{proof}
    Note that \(\overline{u}\) is Lipschitz-continuous \([0,T] \rightarrow H^2(\Omega)\). So, for \(v\in X_{\mathcal{D}}\), with \(\Pi_{\mathcal{D}}v \in L^2(\Omega)\), we have, for \(1 \leq \ell \leq M\),
    \begin{equation}\label{eq-proof-lem:GS consistency rtvf in L2}\begin{aligned}
        \langle \lambda \, g , \Pi_{\mathcal{D}}v \rangle 
        = \bigg\langle \partial_t\overline{u} (t^{(\ell)}) 
        - \text{div}\bigg(\frac{\nabla \overline{u} (t^{(\ell)})}{|\nabla \overline{u}  (t^{(\ell)})|_{\rho}}\bigg)
        + \lambda \, \overline{u}  (t^{(\ell)}),  \Pi_{\mathcal{D}}v \bigg\rangle
    \end{aligned}\end{equation}
    and 
    \begin{equation}\label{eq-proof-lem:GS consistency limit conformity}
        \begin{aligned}
            \int_{\Omega} \text{div}\bigg(\frac{\nabla \overline{u} (t^{(\ell)})}{|\nabla \overline{u}  (t^{(\ell)})|_{\rho}}\bigg) \Pi_{\mathcal{D}}v 
            =   -\int_{\Omega} \frac{\nabla \overline{u} (t^{(\ell)}) }{|\nabla \overline{u}  (t^{(\ell)})|_{\rho}} \cdot \nabla_{\mathcal{D}}v 
                 + \widetilde{W}_{\mathcal{D}}\bigg(\frac{\nabla \overline{u} (t^{(\ell)})}{ |\nabla \overline{u} (t^{(\ell)})|_{\rho}}, v\bigg).
        \end{aligned}
    \end{equation} 
    Let 
    \[
      \forall 1 \leq \ell \leq M, \quad \theta(t^{(\ell)}) := \Pi_{\mathcal{D}} P_{\mathcal{D}} \overline{u} (t^{(\ell)}) - \overline{u}(t^{(\ell)}),
      \quad 
      \delta_{\mathcal{D}}^{(\ell)} \theta := \frac{\theta(t^{(\ell)}) - \theta (t^{(\ell - 1)})}{t^{(\ell)} - t^{(\ell - 1)}}.
    \]
    For \(v\in X_{\mathcal{D}}^{M+1}\) and \(1\leq \ell \leq M\), considering the term \eqref{eq-lem:GS consistency 1}, 
    using
    \eqref{eq-proof-lem:GS consistency rtvf in L2},
    \eqref{eq-proof-lem:GS consistency limit conformity},  
    and \eqref{eq:existence of NL interpolator lemma 2},
    then introducing 
    \[
        \delta_{\mathcal{D}}^{(\ell)} \overline{u}:=\frac{\overline{u}(t^{(\ell)}) - \overline{u}(t^{(\ell-1)})}{t^{(\ell)} - t^{(\ell-1)}}
    \] 
    and recalling Definition \ref{def:GD space size} of \(h_{\mathcal{D}}\),
    we have
\begin{align*}
    \mathcal{E}(\overline{u}(t^{(\ell)}), v^{(\ell)})
   = \, & 
        \int_{\Omega} \left[\delta_{\mathcal{D}}^{(\ell)} P_{\mathcal{D}} \overline{u} - \partial_t \overline{u}(t^{(\ell)})\right] \Pi_{\mathcal{D}} v^{(\ell)} + \lambda \int_{\Omega} \theta(t^{(\ell)}) \ \Pi_{\mathcal{D}} v^{(\ell)} 
         + \widetilde{W}_{\mathcal{D}} \bigg(\frac{\nabla \overline{u} (t^{(\ell)})}{ |\nabla \overline{u} (t^{(\ell)})|_{\rho}}, v^{(\ell)}\bigg)\\
   \leq \, & 
        \int_{\Omega}  \delta_{\mathcal{D}}^{(\ell)} \theta \ \Pi_{\mathcal{D}} v^{(\ell)} + \int_{\Omega} \left[\delta_{\mathcal{D}}^{(\ell)}\overline{u} - \partial_t \overline{u}(t^{(\ell)})\right] \Pi_{\mathcal{D}} v^{(\ell)} + \lambda \int_{\Omega} \theta (t^{(\ell)}) \ \Pi_{\mathcal{D}} v^{(\ell)} \\
        & + \| v^{(\ell)}\|_{\mathcal{D}} \| \nabla \overline{u} (t^{(\ell)}) \|_{\bm{W}_w} h_{\mathcal{D}}.
\end{align*}
Next, multiply by \(\delta t^{(\ell)}\), sum over \(\ell\) from \(1\) to \(m\), for \(1 \leq m \leq M\), use the Cauchy-Schwarz inequality and the Young inequality to write, for any \(\Lambda^{(2)} > 0\),
\begin{align}
    \sum_{\ell = 1}^m \delta t^{(\ell)} \mathcal{E}(\overline{u}(t^{(\ell)}), v^{(\ell)})
    \lesssim \, & \sum_{\ell = 1}^m \delta t^{(\ell)} \| \delta_{\mathcal{D}}^{(\ell)} \theta\|_{L^2(\Omega)} \| \Pi_{\mathcal{D}} v^{(\ell)}\|_{L^2(\Omega)} \nonumber \\
        & + \sum_{\ell = 1}^m \delta t^{(\ell)} \|\delta_{\mathcal{D}}^{(\ell)}\overline{u}- \partial_t \overline{u}(t^{(\ell)})\|_{L^2(\Omega)} \|\Pi_{\mathcal{D}} v^{(\ell)} \|_{L^2(\Omega)} \nonumber \\
        & + \lambda \sum_{\ell = 1}^m \delta t^{(\ell )} \|\theta (t^{(\ell)}) \|_{L^2(\Omega)} \|\Pi_{\mathcal{D}} v^{(\ell)}\|_{L^2(\Omega)} 
         + \sum_{\ell = 1}^m \delta t^{(\ell)} \| v^{(\ell)}\|_{\mathcal{D}} h_{\mathcal{D}} \nonumber \\ 
     \lesssim \, &  
        \frac{1}{\Lambda^{(2)}} T_1 +  \Lambda^{(2)}  T_2 + \sum_{\ell = 1}^m \delta t^{(\ell)} \| v^{(\ell)}\|_{\mathcal{D}} h_{\mathcal{D}}, \label{eq-proof-lem:GS consistency 2}
\end{align}
where \(T_1\) gather all the terms involving the approximations that measure \(\theta\) and \(T_2\) gather the terms involving the test function \(v\):
\begin{align*}
    T_1 := \, & \sum_{\ell = 1}^m \delta t^{(\ell)} \| \delta_{\mathcal{D}}^{(\ell)} \theta\|_{L^2(\Omega)}^2
             + \sum_{\ell = 1}^m \delta t^{(\ell)} \|\delta_{\mathcal{D}}^{(\ell)}\overline{u} 
            - \partial_t \overline{u}(t^{(\ell)})\|_{L^2(\Omega)}^2 
            +  \lambda  \sum_{\ell = 1}^m \delta t^{(\ell )} \|\theta (t^{(\ell)}) \|_{L^2(\Omega)}^2, \\
    T_2 := \, & \sum_{\ell = 1}^m \delta t^{(\ell)} \|\Pi_{\mathcal{D}} v^{(\ell)} \|_{L^2(\Omega)}^2
            + \lambda \sum_{\ell = 1}^m \delta t^{(\ell)} \|\Pi_{\mathcal{D}} v^{(\ell)}\|_{L^2(\Omega)}^2 
            \lesssim \sum_{\ell = 1}^m \delta t^{(\ell)} \|\Pi_{\mathcal{D}} v^{(\ell)} \|_{L^2(\Omega)}^2.
\end{align*}

Consider the term \(T_1\).
For the first term, by the estimate \eqref{eq-lem:NL interpolation error estimate time derivative 2}, 
\begin{equation*}
    \| \delta_{\mathcal{D}}^{(\ell)} \theta\|_{L^2(\Omega)}^2 = \frac{\| \theta^{(\ell)} - \theta^{(\ell -1)}\|_{L^2(\Omega)}^2}{(\delta t^{(\ell)} - \delta t^{(\ell - 1)})^2} 
    \leq 
        \frac{(\delta t^{(\ell)} - \delta t^{(\ell - 1)})^2 \|\partial_t \theta\|_{L^{\infty}(0,T;L^2(\Omega))}^2}{(\delta t^{(\ell)} - \delta t^{(\ell - 1)})^2} 
    \lesssim 
        h_{\mathcal{D}}^2.
\end{equation*}
For the second term, using the fundamental theorem of calculus and the Cauchy-Schwarz inequality, we obtain
    \begin{align*}
         \|\delta_{\mathcal{D}}^{(\ell)}\overline{u} - \partial_t \overline{u}(t^{(\ell)})\|_{L^2(\Omega)}^2 
        = \, & 
            \bigg\|  (\delta t^{(\ell)})^{-1}\int_{t^{(\ell-1)}}^{t^{(\ell)}} (s-t^{(\ell-1)})\, \partial_{tt}\overline{u} \, ds\bigg\|_{L^2(\Omega)}^2\\
        \leq \, &
            \bigg\| (\delta t^{(\ell)})^{-1} \bigg|\int_{t^{(\ell-1)}}^{t^{(\ell)}} (s-t^{(\ell-1)})^2 \, ds \bigg|^{\frac{1}{2}} \bigg| \int_{t^{(\ell-1)}}^{t^{(\ell)}} |\partial_{tt}\overline{u} |^2 \, ds \bigg|^{\frac{1}{2}} \bigg\|_{L^2(\Omega)}^2\\
        \leq \, & 
            \delta t^{(\ell)}  \bigg| \int_{t^{(\ell-1)}}^{t^{(\ell)}} \|\partial_{tt} \overline{u} \|_{L^2(\Omega)}^2 \, ds\bigg| 
        = 
            \delta t^{(\ell)} \|\overline{u}\|_{H^{2}((t^{(\ell-1)}, t^{(\ell)}); L^2(\Omega))}^2.
    \end{align*} 
For the third term, we use the estimate \eqref{eq-lem:NL interpolation error estimate 4} directly.
Collecting these estimates gives
\begin{align*}
    T_1 
    \lesssim  T h_{\mathcal{D}}^2
         + \sum_{\ell = 1}^{m}( \delta t^{(\ell)})^2\|\overline{u} \|_{H^{2}((t^{(\ell-1)}, t^{(\ell)}); L^2(\Omega))}^2 + T \lambda  h_{\mathcal{D}}^2
    \lesssim 
       \| \overline{u}\|_{H^{2}(0,T; L^2(\Omega))}^2 (\delta t_{\mathcal{D}}^{\max})^2 + h_{\mathcal{D}}^2.
\end{align*}

Considering the last term of \eqref{eq-proof-lem:GS consistency 2}, applying the triangle inequality and the Young inequality, and  \eqref{eq-def:GD definition norm}, we obtain
\begin{align*}
    \sum_{\ell = 1}^m \delta t^{(\ell)} \|v^{(\ell)}\|_{\mathcal{D}} h_{\mathcal{D}} 
    \lesssim  
        \sum_{\ell = 1}^m \delta t^{(\ell)} 
        \Big(\Lambda^{(2)} \|\Pi_{\mathcal{D}} v^{(\ell)}\|_{L^2(\Omega)}^2 
        +   \frac{1}{\Lambda^{(2)}} h_{\mathcal{D}}^2 
        + \|\nabla_{\mathcal{D}} v^{(\ell)}\|_{L^1(\Omega)} h_{\mathcal{D}}\Big).
\end{align*}
Combining this and the estimate for \(T_1\) yields \eqref{eq-lem:GS consistency 2}.
\end{proof}
\begin{remark}
      In \cite[Equation (4.27)]{Feng2003AnalysisOT}, the equivalent of the term $\delta_{\mathcal{D}}^{(\ell)} \theta$ is apparently bounded by swapping the discrete derivative and the interpolator (equivalent of $P_{\mathcal D}$). However, this swap does not seem valid since the interpolator is not a linear mapping. As a consequence, it seems to us that there is a gap in the analysis done in this reference, and our treatment of this term in the proof above fixes that gap -- which requires paying a cost in the form of a slightly worse estimate.
\end{remark}
\begin{remark}
    We assume the exact solution \(\overline{u}\) belongs to \(H^2(0,T;L^2(\Omega))\), which is stronger than \(H^2(0,T;H^1(\Omega)')\) in \cite[Theorem 1.7]{Feng2003AnalysisOT}.
    This assumption is used in Lemma \ref{lem:GS consistency} to estimate the term
    \[ \int_{\Omega}\big[\delta_{\mathcal{D}}^{(\ell)}\overline{u} - \partial_t \overline{u}(t^{(\ell)})\big] \, \Pi_{\mathcal{D}} v^{(\ell)}.\]
    This product cannot be written as a duality product between \(H^1(\Omega)'\) and \(H^1(\Omega)\), since our approximation is non-conforming. 
    It could however be possible to weaken the regularity assumption;
    that is, bound the above term using discrete \(H^1\) norm and discrete \((H^1)'\) norm on the time derivative, and bound the discrete norm using a weaker regularity on \(\overline{u}\).
    To improve the legibility of our (already technical) estimates, we are not exploring this path in this work and rely on the stronger regularity assumption on \(\overline{u}\).
\end{remark}
    \section{Proof of the main theorems}\label{section:main proof}
In this section, a proof of Theorem \ref{thm:GDM Error Estimate Theorem} is given and done by splitting the main steps into lemmas.

\begin{lemma}\label{lem:main theorem step1}
    Suppose that Assumption \eqref{def:GDM Error Estimate Theorem assumption} holds.
    Let \(u = (u^{(\ell)}: \ell = 1, \ldots, M) \in X_{\mathcal{D}}^{M}\) be the solution to \eqref{eq-def:GS}.
    Then, for all \(m = 1, \ldots, M\),
    \begin{equation}\label{eq-lem:main theorem step1 1} \begin{aligned}
        & \| \Pi_{\mathcal{D}} u^{(m)} - \Pi_{\mathcal{D}} P_{\mathcal{D}} \overline{u}(t^{(m)})\|_{L^2(\Omega)}^2 
         + \Lambda^{(3)} \sum_{\ell = 1}^m  \delta t^{(\ell)}  \int_{\Omega}  \frac{|\nabla_{\mathcal{D}} u^{(\ell)} - \nabla_{\mathcal{D}}P_{\mathcal{D}} \overline{u} ( t^{(\ell)})|^2}{|\nabla_{\mathcal{D}} u^{(\ell)}|_{\rho}}\\
     \lesssim \, &  
        (e_{\mathcal{D}}^{\text{ini}} )^2 + (\delta t_{\mathcal{D}}^{\max})^2 + h_{\mathcal{D}}^2  + \|\nabla_{\mathcal{D}} u - \nabla_{\mathcal{D}} P_{\mathcal{D}} \overline{u}\|_{L^1(0,T; L^1(\Omega))} h_{\mathcal{D}},
    \end{aligned}\end{equation}
    where $e_{\mathcal{D}}^{\text{ini}}$ is given by \eqref{def:eDini} and
    \begin{equation}\label{eq:def.lambda3}
        \Lambda^{(3)}
        := 
            \min_{1 \leq \ell \leq M}\essinf_{x\in \Omega} \left(1 - \frac{|\nabla_{\mathcal{D}}P_{\mathcal{D}} \overline{u} (x, t^{(\ell)})|}{ |\nabla_{\mathcal{D}}P_{\mathcal{D}} \overline{u} ( x,t^{(\ell)})|_{\rho}}  \right) \gtrsim 1.
    \end{equation}
    As a consequence,  
    \begin{equation}\label{eq-lem:main theorem step1 2} \begin{aligned}
        & \| \Pi_{\mathcal{D}} u^{(m)} - \Pi_{\mathcal{D}} P_{\mathcal{D}} \overline{u}(t^{(m)}) \|_{L^2(\Omega)}^2 
         + \Lambda^{(3)} \sum_{\ell = 1}^m  \delta t^{(\ell)}  \int_{\Omega}  \frac{|\nabla_{\mathcal{D}} u^{(\ell)} - \nabla_{\mathcal{D}}P_{\mathcal{D}} \overline{u} ( t^{(\ell)})|^2}{ |\nabla_{\mathcal{D}} u^{(\ell)}|_{\rho}}\\
        \lesssim \, 
            &  (e_{\mathcal{D}}^{\text{ini}} )^2 
            + (\delta t_{\mathcal{D}}^{\max})^2
            + h_{\mathcal{D}}.
    \end{aligned}
    \end{equation}

\end{lemma}
\begin{proof}
    Let 
\begin{equation} \label{eq-proof-lem:main theorem step1 1}
    \xi := \max_{\\ 1 \leq \ell \leq M}\esssup_{ x\in \Omega} |\nabla_{\mathcal{D}} P_{\mathcal{D}} \overline{u}(x, t^{(\ell)})|.
\end{equation}
By \eqref{eq-lem:NL interpolation error estimate 1}, we have \(\xi \lesssim 1\).
Noticing that
\begin{equation}\label{eq-proof-lem:main theorem step1 increasing fraction}
    \text{for all } \nu \geq 0, \text{ the function }  \mu \in \mathbb{R}^+ \mapsto \frac{\mu -\nu}{|\mu|_{\rho}} \text{ is strictly increasing,} 
\end{equation}
we have
\begin{equation*}
    \Lambda^{(3)} = \min_{1 \leq \ell \leq M}\essinf_{x\in \Omega}  \left(1 - \frac{|\nabla_{\mathcal{D}}P_{\mathcal{D}} \overline{u} (x, t^{(\ell)})|}{ |\nabla_{\mathcal{D}}P_{\mathcal{D}} \overline{u} ( x,t^{(\ell)})|_{\rho}}  \right) = 1 - \frac{\xi}{|\xi|_{\rho}} > 0.
\end{equation*}
From Lemma \ref{lem:GS stability} with  \(  u_1^{(m)} := u^{(m)}\) and \(u_2^{(m)} := P_{\mathcal{D}} \overline{u}(t^{(m)})\), we find
\begin{equation}\label{eq-proof-lem:main theorem step1 3.1}\begin{aligned}
        & \| \Pi_{\mathcal{D}} u^{(m)} - \Pi_{\mathcal{D}} P_{\mathcal{D}} \overline{u}(t^{(m)}) \|_{L^2(\Omega)}^2 
        + \sum_{\ell = 1}^{m} \delta t^{(\ell)} \| \Pi_{\mathcal{D}} u^{(\ell)} - \Pi_{\mathcal{D}} P_{\mathcal{D}} \overline{u}(t^{(\ell)}) \|_{L^2(\Omega)}^2\\
        & + \sum_{\ell = 1}^m (\delta t^{(\ell)})^2 \| \delta_{\mathcal{D}}^{(\ell)}\left(  u -  P_{\mathcal{D}} \overline{u} \right)\|_{L^2(\Omega)}^2 
        + \Lambda^{(3)} \sum_{\ell = 1}^m  \delta t^{(\ell)}  \int_{\Omega}  \frac{|\nabla_{\mathcal{D}} u^{(\ell)} - \nabla_{\mathcal{D}}P_{\mathcal{D}} \overline{u} ( t^{(\ell)})|^2}{|\nabla_{\mathcal{D}} u^{(\ell)}|_{\rho}}\\
        \leq \, &   
            \| \Pi_{\mathcal{D}}u^{(0)} - \Pi_{\mathcal{D}} P_{\mathcal{D}}\overline{u}(t^{(0)})\|_{L^2(\Omega)}^2 + \sum_{\ell = 1}^m \delta t^{(\ell)}  R^{(\ell)} (u^{(\ell)} - P_{\mathcal{D}} \overline{u} (t^{(\ell)}))
\end{aligned} \end{equation}
with \(R^{(\ell)}(\cdot) = \mathcal{E}(\overline{u}(t^{(\ell)}), \cdot).\)
Applying Lemma \ref{lem:GS consistency}, 
with \(v^{(\ell)} = u^{(\ell)} - P_{\mathcal{D}} \overline{u} (t^{(\ell)})\), 
we arrive at,
for any \(\Lambda^{(2)} > 0\),
\begin{equation}\label{eq-proof-lem:main theorem step1 3.2}\begin{aligned}
        & \sum_{\ell = 1}^m \delta t^{(\ell)}  R^{(\ell)}(u^{(\ell)} - P_{\mathcal{D}} \overline{u} (t^{(\ell)})) \\
        \lesssim \, 
        & \frac{1}{\Lambda^{(2)}} \big[ (\delta t_{\mathcal{D}}^{\max})^2 
                + h_{\mathcal{D}}^2 \big]
                +  \Lambda^{(2)}  \sum_{\ell = 1}^m \delta t^{(\ell)} \|\Pi_{\mathcal{D}}u^{(\ell)} - \Pi_{\mathcal{D}} P_{\mathcal{D}}\overline{u}(t^{(\ell)}) \|_{L^2(\Omega)}^2\\
        &+ \|\nabla_{\mathcal{D}}u  - \nabla_{\mathcal{D}} P_{\mathcal{D}}\overline{u}\|_{L^1(0,T; L^1(\Omega))} h_{\mathcal{D}}.
\end{aligned} \end{equation}
Combining \eqref{eq-proof-lem:main theorem step1 3.1} and \eqref{eq-proof-lem:main theorem step1 3.2} and choosing \(\Lambda^{(2)}\) small enough (depending on the hidden constants) yields \eqref{eq-lem:main theorem step1 1}.

From Lemma \ref{lem:GS existence and uniqueness} and \eqref{eq-lem:NL interpolation error estimate 1}, 
we know that
    \begin{equation}\label{eq-proof-lem:main theorem step1 2}
        \| \nabla_{\mathcal{D}} u  -  \nabla_{\mathcal{D}}  P_{\mathcal{D}} \overline{u} \|_{L^1(0,T; L^1(\Omega))} \lesssim 1. 
    \end{equation}
Thus, bound the last term in \eqref{eq-lem:main theorem step1 1} using \eqref{eq-proof-lem:main theorem step1 2} to obtain \eqref{eq-lem:main theorem step1 2}.
\end{proof}

In the following lemma, the estimate \eqref{eq-proof-lem:main theorem step1 2} is improved. 
\begin{lemma}\label{lem:main theorem step2}
Suppose that Assumption \eqref{def:GDM Error Estimate Theorem assumption} holds.
Let \(u = (u^{(\ell)}: \ell = 1, \ldots, M) \in X_{\mathcal{D}}^{M}\) be the solution to \eqref{eq-def:GS}.
Then 
    \begin{equation}\label{eq-lem:main theorem step2} \begin{aligned}
        \|\nabla_{\mathcal{D}} u - \nabla_{\mathcal{D}}P_{\mathcal{D}} \overline{u} \|_{L^1(0,T; L^1(\Omega))} 
        \lesssim \ &  
            e_{\mathcal{D}}^{\text{ini}} + (e_{\mathcal{D}}^{\text{ini}})^2  + \delta t_{\mathcal{D}}^{\max}  + h_{\mathcal{D}}.
    \end{aligned}\end{equation}
\end{lemma}
\begin{proof}    
We start by considering the second term on the left-hand side of \eqref{eq-lem:main theorem step1 1} by analysing 
\begin{equation*}
   1 \leq \ell \leq M, x\in \Omega, 
   \quad 
   \frac{|\nabla_{\mathcal{D}} u^{(\ell)} - \nabla_{\mathcal{D}}P_{\mathcal{D}} \overline{u} (x, t^{(\ell)})|}{ |\nabla_{\mathcal{D}} u^{(\ell)}(x)|_{\rho}} |\nabla_{\mathcal{D}} u^{(\ell)}(x) - \nabla_{\mathcal{D}}P_{\mathcal{D}} \overline{u} (x, t^{(\ell)})|.
\end{equation*}
For \(x\in \{ |\nabla_{\mathcal{D}} u^{(\ell)}| > \xi + 1 \}\) (recall \eqref{eq-proof-lem:main theorem step1 1}),  by the reverse triangle inequality and \eqref{eq-proof-lem:main theorem step1 increasing fraction},
\begin{equation*}
    \frac{|\nabla_{\mathcal{D}} u^{(\ell)}(x) - \nabla_{\mathcal{D}}P_{\mathcal{D}} \overline{u} ( x, t^{(\ell)})|}{ |\nabla_{\mathcal{D}} u^{(\ell)}(x)|_{\rho}} 
    \geq 
        \frac{  |\nabla_{\mathcal{D}} u^{(\ell)}(x)| - |\nabla_{\mathcal{D}}P_{\mathcal{D}} \overline{u} ( x, t^{(\ell)})|  }{ |\nabla_{\mathcal{D}} u^{(\ell)}(x)|_{\rho}} 
    \geq 
        \frac{1}{|\xi + 1|_{\rho}}.
\end{equation*}
Next, for \(x\in \{ |\nabla_{\mathcal{D}} u^{(\ell)}| \leq \xi + 1 \}\), we have 
\begin{equation*}
    \frac{|\nabla_{\mathcal{D}} u^{(\ell)}(x) - \nabla_{\mathcal{D}}P_{\mathcal{D}} \overline{u} ( x, t^{(\ell)})|}{|\nabla_{\mathcal{D}} u^{(\ell)}(x)|_{\rho}} 
    \geq 
        \frac{|\nabla_{\mathcal{D}} u^{(\ell)}(x) - \nabla_{\mathcal{D}}P_{\mathcal{D}} \overline{u} ( x, t^{(\ell)})|}{|\xi + 1|_{\rho}} .
\end{equation*}
So, for \(x\in \Omega\), 
\begin{equation}\label{eq-proof-lem:main theorem step2 1} \begin{aligned}
    &\frac{|\nabla_{\mathcal{D}} u^{(\ell)}(x) - \nabla_{\mathcal{D}}P_{\mathcal{D}} \overline{u} (x, t^{(\ell)})|^2}{|\nabla_{\mathcal{D}} u^{(\ell)}(x)|_{\rho}} \\  
    \geq \, & 
        \mathbbm{1}_{\{|\nabla_{\mathcal{D}} u^{(\ell)}| > \xi + 1 \}} \frac{|\nabla_{\mathcal{D}} u^{(\ell)}(x) - \nabla_{\mathcal{D}}P_{\mathcal{D}} \overline{u} (x, t^{(\ell)})|}{|\xi + 1|_{\rho}} \\
        & + \mathbbm{1}_{\{|\nabla_{\mathcal{D}} u^{(\ell)}| \leq \xi + 1 \}} \frac{ |\nabla_{\mathcal{D}} u^{(\ell)}(x) - \nabla_{\mathcal{D}}P_{\mathcal{D}} \overline{u} (x, t^{(\ell)})|^2}{|\xi + 1|_{\rho}}.
\end{aligned}\end{equation}
We split the following \(L^1\)-norm by indicator functions \(\mathbbm{1}_{\{|\nabla u ^{(\ell)}| > \xi + 1 \}}\) and \(\mathbbm{1}_{\{|\nabla_{\mathcal{D}} u^{(\ell)} | \leq \xi + 1 \}}\), then use \eqref{eq-proof-lem:main theorem step2 1} and the Cauchy-Schwarz inequality.
This leads to
\begin{align*}
    & \int_{\Omega} |\nabla_{\mathcal{D}} u^{(\ell)}  - \nabla_{\mathcal{D}}P_{\mathcal{D}} \overline{u} (t^{(\ell)})| \\
    = \, & 
        \int_{\Omega} \mathbbm{1}_{\{|\nabla_{\mathcal{D}} u^{(\ell)}| > \xi + 1 \}}  |\nabla_{\mathcal{D}} u^{(\ell)}  - \nabla_{\mathcal{D}}P_{\mathcal{D}} \overline{u} (t^{(\ell)})| 
         + \int_{\Omega}\mathbbm{1}_{\{|\nabla_{\mathcal{D}} u^{(\ell)}| \leq \xi + 1 \}} |\nabla_{\mathcal{D}} u^{(\ell)}  - \nabla_{\mathcal{D}}P_{\mathcal{D}} \overline{u} (t^{(\ell)})| \\
    \leq \, & 
        |\xi + 1|_{\rho} \int_{\Omega}   \frac{|\nabla_{\mathcal{D}} u^{(\ell)}  - \nabla_{\mathcal{D}}P_{\mathcal{D}} \overline{u} (t^{(\ell)})|^2}{ |\nabla_{\mathcal{D}} u^{(\ell)} |_{\rho}}
         + |\Omega|_2^{\frac{1}{2}} |\xi + 1|_{\rho}^{\frac{1}{2}} 
        \bigg\{
        \int_{\Omega}  \frac{|\nabla_{\mathcal{D}} u^{(\ell)}  - \nabla_{\mathcal{D}}P_{\mathcal{D}} \overline{u} (t^{(\ell)})|^2}{ |\nabla_{\mathcal{D}} u^{(\ell)} |_{\rho}}
        \bigg\}^{\frac{1}{2}}.
\end{align*}
Then multiply both sides by \(\delta t^{(\ell)}\) and sum over \(\ell\) from \(1\) to \(m\), apply the Cauchy-Schwarz inequality,
\begin{equation}\label{eq-proof-lem:main theorem step2 2} \begin{aligned}
    &\sum_{\ell = 1}^m \delta t^{(\ell)} \|\nabla_{\mathcal{D}} u^{(\ell)}  - \nabla_{\mathcal{D}}P_{\mathcal{D}} \overline{u} (t^{(\ell)})\|_{L^1(\Omega)}\\
    \leq \, &   
        |\xi + 1|_{\rho}    \sum_{\ell = 1}^m \delta t^{(\ell)} \int_{\Omega}  \frac{|\nabla_{\mathcal{D}} u^{(\ell)}  - \nabla_{\mathcal{D}}P_{\mathcal{D}} \overline{u} (t^{(\ell)})|^2}{ |\nabla_{\mathcal{D}} u^{(\ell)} |_{\rho}} \\
        & + |\Omega|_2^{\frac{1}{2}} T^{\frac{1}{2}} |\xi + 1|_{\rho}^{\frac{1}{2}} 
        \bigg\{
        \sum_{\ell = 1}^m \delta t^{(\ell)}\int_{\Omega}  \frac{|\nabla_{\mathcal{D}} u^{(\ell)}  - \nabla_{\mathcal{D}}P_{\mathcal{D}} \overline{u} (t^{(\ell)})|^2}{ |\nabla_{\mathcal{D}} u^{(\ell)} |_{\rho}}
        \bigg\}^{\frac{1}{2}}.
\end{aligned} \end{equation}
Notice that, from \eqref{eq-lem:main theorem step1 1} and \eqref{eq-lem:main theorem step1 2},
\begin{subequations}
\begin{align}
    & \sum_{\ell = 1}^m \delta t^{(\ell)}\int_{\Omega} \frac{|\nabla_{\mathcal{D}} u^{(\ell)}  - \nabla_{\mathcal{D}}P_{\mathcal{D}} \overline{u} (t^{(\ell)})|^2}{ |\nabla_{\mathcal{D}} u^{(\ell)} |_{\rho}} \nonumber\\
    \lesssim \, &  
        (e_{\mathcal{D}}^{\text{ini}} )^2 + (\delta t_{\mathcal{D}}^{\max})^2 +   h_{\mathcal{D}}^2 + \| \nabla_{\mathcal{D}} u -  \nabla_{\mathcal{D}}  P_{\mathcal{D}} \overline{u} \|_{L^1(0,T; L^1(\Omega))} h_{\mathcal{D}}  \label{eq-proof-lem:main theorem step2 3a}  \\
    \lesssim \, &  
        (e_{\mathcal{D}}^{\text{ini}} )^2 + (\delta t_{\mathcal{D}}^{\max})^2 +   h_{\mathcal{D}}. \label{eq-proof-lem:main theorem step2 3b}
\end{align}
\end{subequations}
Using \eqref{eq-proof-lem:main theorem step2 3b} on the first term of the right-hand side of \eqref{eq-proof-lem:main theorem step2 2}, and \eqref{eq-proof-lem:main theorem step2 3a} on the second term, we obtain the following
\begin{align}
    &\|\nabla_{\mathcal{D}} u   - \nabla_{\mathcal{D}}P_{\mathcal{D}} \overline{u} \|_{L^1(0,T; L^1(\Omega))}\notag\\
    \lesssim \, 
        & |\xi + 1|_{\rho}  \big\{  (e_{\mathcal{D}}^{\text{ini}} )^2 +  (\delta t_{\mathcal{D}}^{\max})^2 +   h_{\mathcal{D}} \big\} +  |\xi + 1|_{\rho}^{\frac{1}{2}}  \big\{ e_{\mathcal{D}}^{\text{ini}}  
        +  \delta t_{\mathcal{D}}^{\max}    \notag\\
        &  +  h_{\mathcal{D}} +   \| \nabla_{\mathcal{D}} u^{(\ell)} -  \nabla_{\mathcal{D}}  P_{\mathcal{D}} \overline{u}(t^{(\ell)})\|_{L^1(0,T; L^1(\Omega))}^{\frac{1}{2}}  h_{\mathcal{D}}^{\frac{1}{2}}   \big\} \notag\\
    \lesssim \,  
        & |\xi + 1|_{\rho} \big\{   
            (e_{\mathcal{D}}^{\text{ini}} )^2 
                + (\delta t_{\mathcal{D}}^{\max})^2 
                + h_{\mathcal{D}} \big\} 
                + |\xi + 1|_{\rho}^{\frac{1}{2}}  \bigg\{ e_{\mathcal{D}}^{\text{ini}} 
                +  \delta t_{\mathcal{D}}^{\max} \notag\\
         &   +  h_{\mathcal{D}}  +    \Lambda^{(4)} \| \nabla_{\mathcal{D}} u^{(\ell)} -  \nabla_{\mathcal{D}}  P_{\mathcal{D}} \overline{u}(t^{(\ell)})\|_{L^1(0,T; L^1(\Omega))} +  \frac{1}{\Lambda^{(4)}} h_{\mathcal{D}} \bigg\}
         \label{def:lambda4}
\end{align}
where the last line holds by Young inequality, for any \(\Lambda^{(4)}> 0\).
We then choose \(\Lambda^{(4)}\) depending on \(\xi\) and \(\rho\),
so that \(\| \nabla_{\mathcal{D}} u  -  \nabla_{\mathcal{D}}  P_{\mathcal{D}} \overline{u} \|_{L^1(0,T;L^1(\Omega))} \) can be absorbed in the left-hand side and thus get \eqref{eq-lem:main theorem step2}.
\end{proof}

Now, we prove the main theorem by improving the estimate \eqref{eq-lem:main theorem step1 1} using Lemma \ref{lem:main theorem step2}.
\begin{proof}[Proof of Theorem \ref{thm:GDM Error Estimate Theorem}]
By Lemma \ref{lem:main theorem step1} and Lemma \ref{lem:main theorem step2},
we plug the estimate \eqref{eq-lem:main theorem step2} back into \eqref{eq-lem:main theorem step1 1}
and apply the Young inequality to the last term on the right-hand side: for all \(m = 1, \ldots, M\),
\begin{equation}\label{eq-proof-thm:GDM Error Estimate Theorem 1} \begin{aligned}
    & \| \Pi_{\mathcal{D}} u^{(m)} - \Pi_{\mathcal{D}} P_{\mathcal{D}} \overline{u}(t^{(m)}) \|_{L^2(\Omega)}^2 
    + \Lambda^{(3)} \sum_{\ell = 1}^m  \delta t^{(\ell)}  \int_{\Omega}  \frac{|\nabla_{\mathcal{D}} u^{(\ell)} - \nabla_{\mathcal{D}}P_{\mathcal{D}} \overline{u} ( t^{(\ell)})|^2}{|\nabla_{\mathcal{D}} u^{(\ell)}|_{\rho}}\\
    \lesssim \, & 
        (e_{\mathcal{D}}^{\text{ini}})^2 + (\delta t_{\mathcal{D}}^{\max})^2 +  h_{\mathcal{D}}^2 +  \big\{ e_{\mathcal{D}}^{\text{ini}} + (e_{\mathcal{D}}^{\text{ini}})^2  +  \delta t_{\mathcal{D}}^{\max}  + h_{\mathcal{D}}  \big\} h_{\mathcal{D}} \\
    \lesssim \, & 
        (e_{\mathcal{D}}^{\text{ini}})^2 + (e_{\mathcal{D}}^{\text{ini}})^4 + (\delta t_{\mathcal{D}}^{\max})^2 +  h_{\mathcal{D}}^2.
\end{aligned}\end{equation} 
Thus, by the triangle inequality, 
the \(L^{\infty}(0,T; L^2(\Omega))\)-estimate for the approximation in \eqref{eq:GDM Error Estimate Theorem 1} of the solution follows from 
\eqref{eq-proof-thm:GDM Error Estimate Theorem 1} and \eqref{eq-lem:NL interpolation error estimate 4};
the \(L^1(0,T;L^1(\Omega))\)-estimate for the approximation of its gradient follows from \eqref{eq-lem:main theorem step2} and \eqref{eq-lem:NL interpolation error estimate 2}.
\end{proof}
\begin{proof}[Proof of Theorem \ref{thm:GDM Error Estimate Theorem 2}.]
    Now we prove the \(L^2(0,T; H^1(\Omega))\)-estimate. 
    Consider for any \(1 \leq n \leq N\), and the triangulation \(\mathcal{T}\) in the sense of Definition \eqref{def:mesh},
    \begin{align*}
            &  \Lambda^{(3)} \sum_{\ell = 1}^n \delta t^{(\ell)} \| \nabla_{\mathcal{D}} u^{(\ell)} - \nabla_{\mathcal{D}}P_{\mathcal{D}} \overline{u} ( t^{(\ell)})\|_{L^2(\Omega)}^2\\
            = \, &\sum_{\tau \in \mathcal{T}}\Lambda^{(3)} \sum_{\ell = 1}^n \delta t^{(\ell)} \| \nabla_{\mathcal{D}} u^{(\ell)} - \nabla_{\mathcal{D}}P_{\mathcal{D}} \overline{u} ( t^{(\ell)})\|_{L^2(\tau)}^2\\
            =  \, & \Lambda^{(3)} \sum_{\tau \in \mathcal{T}} \sum_{\ell = 1}^n  \delta t^{(\ell)} \int_{\tau} | \nabla_{\mathcal{D}} u^{(\ell)}|_{\rho} \frac{|\nabla_{\mathcal{D}} u^{(\ell)} - \nabla_{\mathcal{D}}P_{\mathcal{D}} \overline{u} ( t^{(\ell)})|^2}{|\nabla_{\mathcal{D}} u^{(\ell)}|_{\rho}}\\
            \leq \, & \Lambda^{(3)} \sum_{\tau \in \mathcal{T}} \sum_{\ell = 1}^n \big(\rho + \| \nabla_{\mathcal{D}} u^{(\ell)}\|_{L^{\infty}(\tau)}\big) \delta t^{(\ell)} \int_{\tau} \frac{|\nabla_{\mathcal{D}} u^{(\ell)} - \nabla_{\mathcal{D}}P_{\mathcal{D}} \overline{u} ( t^{(\ell)})|^2}{|\nabla_{\mathcal{D}} u^{(\ell)}|_{\rho}}\\
            \leq \, & \Lambda^{(3)}\sum_{\tau \in \mathcal{T}} \sum_{\ell = 1}^n (\rho + \| \nabla_{\mathcal{D}}P_{\mathcal{D}} \overline{u} ( t^{(\ell)})\|_{L^{\infty}(\tau)}) \, \delta t^{(\ell)} \int_{\tau} \frac{|\nabla_{\mathcal{D}} u^{(\ell)} - \nabla_{\mathcal{D}}P_{\mathcal{D}} \overline{u} ( t^{(\ell)})|^2}{|\nabla_{\mathcal{D}} u^{(\ell)}|_{\rho}}\\
            & + \Lambda^{(3)} \sum_{\tau \in \mathcal{T}} \sum_{\ell = 1}^n   \| \nabla_{\mathcal{D}} u^{(\ell)} - \nabla_{\mathcal{D}}P_{\mathcal{D}} \overline{u} ( t^{(\ell)})\|_{L^{\infty}(\tau)} \, \delta t^{(\ell)} \int_{\tau} \frac{|\nabla_{\mathcal{D}} u^{(\ell)} - \nabla_{\mathcal{D}}P_{\mathcal{D}} \overline{u} ( t^{(\ell)})|^2}{|\nabla_{\mathcal{D}} u^{(\ell)}|_{\rho}} \\
            =: & \, T_1 + T_2.
    \end{align*}
    Since \(\rho\) is a fixed constant and \(\| \nabla_{\mathcal{D}}P_{\mathcal{D}} \overline{u} ( t^{(\ell)})\|_{L^{\infty}(\tau)}\) is bounded above for all \(\ell\) by the estimate \eqref{eq-lem:NL interpolation error estimate 1},  by \eqref{eq-proof-thm:GDM Error Estimate Theorem 1}, we have 
    \begin{equation*}
    \begin{aligned}
        T_1 
        \lesssim \, & (e_{\mathcal{D}}^{\text{ini}})^2 + (e_{\mathcal{D}}^{\text{ini}})^4 + (\delta t_{\mathcal{D}}^{\max})^2 +  h_{\mathcal{D}}^2.
    \end{aligned}
    \end{equation*}
    Now we consider the term \(T_2\). 
    By the discrete Lebesgue estimates \cite[Lemma 1.25]{Pietro2020TheHH}, for any \(\tau \in \mathcal{T}\),
    \begin{equation*}
       \| \nabla_{\mathcal{D}} u^{(\ell)} - \nabla_{\mathcal{D}} P_{\mathcal{D}} \overline{u}(t^{(\ell)})\|_{L^{\infty}(\tau)} \leq \frac{\| \nabla_{\mathcal{D}} u^{(\ell)} - \nabla_{\mathcal{D}} P_{\mathcal{D}} \overline{u}(t^{(\ell)})\|_{L^{2}(\tau)}}{|\tau|_2^{1/2}}.
    \end{equation*}
    This and the Young inequality imply that 
    \begin{align*} 
            T_2 \leq \, & \Lambda^{(3)} \sum_{\tau \in \mathcal{T}}  \sum_{\ell = 1}^n \frac{\| \nabla_{\mathcal{D}} u^{(\ell)} - \nabla_{\mathcal{D}} P_{\mathcal{D}} \overline{u}(t^{(\ell)})\|_{L^{2}(\tau)}}{ |\tau|_2^{1/2}} \, \delta t^{(\ell)} \int_{\tau} \frac{|\nabla_{\mathcal{D}} u^{(\ell)} - \nabla_{\mathcal{D}}P_{\mathcal{D}} \overline{u} ( t^{(\ell)})|^2}{|\nabla_{\mathcal{D}} u^{(\ell)}|_{\rho}}\\
             \leq \, & \frac{\Lambda^{(3)}}{2}  \sum_{\tau \in \mathcal{T}}  \sum_{\ell = 1}^n \delta t^{(\ell)} \| \nabla_{\mathcal{D}} u^{(\ell)} - \nabla_{\mathcal{D}} P_{\mathcal{D}} \overline{u}(t^{(\ell)})\|_{L^{2}(\tau)}^2 \\
              & + \frac{\Lambda^{(3)}}{2} \sum_{\tau \in \mathcal{T}}  \sum_{\ell = 1}^n  \delta t^{(\ell)} \frac{1}{{ |\tau|_2}} \bigg(\int_{\tau} \frac{|\nabla_{\mathcal{D}} u^{(\ell)} - \nabla_{\mathcal{D}}P_{\mathcal{D}} \overline{u} ( t^{(\ell)})|^2}{|\nabla_{\mathcal{D}} u^{(\ell)}|_{\rho}}\bigg)^2 =: T_{2,1} + T_{2,2}.
    \end{align*}
    The term \(T_{2,1}\) can be written as
    \begin{equation*}
        T_{2,1} =\frac{\Lambda^{(3)}}{2}   \sum_{\ell = 1}^n \delta t^{(\ell)} \| \nabla_{\mathcal{D}} u^{(\ell)} - \nabla_{\mathcal{D}} P_{\mathcal{D}} \overline{u}(t^{(\ell)})\|_{L^{2}(\Omega)}^2.
    \end{equation*}
    Now consider the term \(T_{2,2}\). 
    Recalling that \(h_{\mathcal{T}}^2 \lesssim |\tau|_2\), 
    by \eqref{eq-proof-thm:GDM Error Estimate Theorem 1},  
    we have 
    \begin{align*} 
            T_{2,2} \lesssim \, & \frac{\Lambda^{(3)}}{2} \sum_{\tau \in \mathcal{T}}  \sum_{\ell = 1}^n  \delta t^{(\ell)} \frac{1}{h_{\mathcal{T}}^2} \bigg(\int_{\tau} \frac{|\nabla_{\mathcal{D}} u^{(\ell)} - \nabla_{\mathcal{D}}P_{\mathcal{D}} \overline{u} ( t^{(\ell)})|^2}{|\nabla_{\mathcal{D}} u^{(\ell)}|_{\rho}}\bigg)^2\\
            \lesssim \, & \frac{\Lambda^{(3)}}{2}   \frac{1}{h_{\mathcal{T}}^2} \bigg( \sum_{\ell = 1}^n  \delta t^{(\ell)} \int_{\Omega} \frac{|\nabla_{\mathcal{D}} u^{(\ell)} - \nabla_{\mathcal{D}}P_{\mathcal{D}} \overline{u} ( t^{(\ell)})|^2}{|\nabla_{\mathcal{D}} u^{(\ell)}|_{\rho}}\bigg)^2 \\
            \lesssim \, & \frac{\Lambda^{(3)}}{2} \frac{(e_{\mathcal{D}}^{\text{ini}})^4 + (e_{\mathcal{D}}^{\text{ini}})^8 + (\delta t_{\mathcal{D}}^{\max})^4 +  h_{\mathcal{D}}^4}{h_{\mathcal{T}}^2}.
    \end{align*}
    Collecting estimates for \(T_1\) and \(T_2\), rearranging the terms, we arrive at 
    \begin{align*} 
        & \Lambda^{(3)} \sum_{\ell = 1}^n \delta t^{(\ell)} \| \nabla_{\mathcal{D}} u^{(\ell)} - \nabla_{\mathcal{D}}P_{\mathcal{D}} \overline{u} ( t^{(\ell)})\|_{L^2(\tau)}^2 \\
        \lesssim \, &  (e_{\mathcal{D}}^{\text{ini}})^2 + (e_{\mathcal{D}}^{\text{ini}})^4 + (\delta t_{\mathcal{D}}^{\max})^2 +  h_{\mathcal{D}}^2 + \frac{(e_{\mathcal{D}}^{\text{ini}})^4 + (e_{\mathcal{D}}^{\text{ini}})^8 + (\delta t_{\mathcal{D}}^{\max})^4 +  h_{\mathcal{D}}^4}{h_{\mathcal{T}}^2}.
    \end{align*}
    Under the assumptions \eqref{eq-def:GDM Error Estimate Theorem assumption 2} and \eqref{eq-thm:GDM Error Estimate Theorem 2 assumption}, 
    we deduce
    \begin{equation*}
        \Lambda^{(3)} \sum_{\ell = 1}^n \delta t^{(\ell)} \| \nabla_{\mathcal{D}} u^{(\ell)} - \nabla_{\mathcal{D}}P_{\mathcal{D}} \overline{u} ( t^{(\ell)})\|_{L^2(\Omega)}^2 \lesssim h_{\mathcal{T}}^2. \qedhere
    \end{equation*}
\end{proof}

    \section{Numerical examples}\label{sec:numerical example}
In this section, we present numerical examples of gradient discretisations for the model \eqref{model:rtvf} via:
\begin{itemize}
    \item \textbf{Conforming $\mathbb{P}^1$ FEM (CP1FEM).} Let \(\mathcal{T}\) be a conforming triangular mesh in the sense of Definition \ref{def:mesh} of the domain \(\Omega\) with mesh size \(h\).
Each \(v_{\mathcal{D}} \in X_{\mathcal{D}}\) is a vector of values at the vertices of the mesh \(\mathcal{T}\), 
\(\Pi_{\mathcal{D}} v_{\mathcal{D}}\) is the continuous piecewise linear function on \(\mathcal{T}\) that takes its values at the vertices,
and \(\nabla_{\mathcal{D}} v_{\mathcal{D}} = \nabla(\Pi_{\mathcal{D}} v_{\mathcal{D}})\) is the piecewise constant function on \(\mathcal{T}\). Then the gradient scheme \eqref{def:GS} is the standard $\mathbb{P}^1$ FEM scheme for the model \eqref{model:rtvf}.
    \item \textbf{Nonconforming $\mathbb{P}^1$ FEM (NCP1FEM).} Under the mesh \(\mathcal{T}\), each  \(v_{\mathcal{D}} \in X_{\mathcal{D}}\) is a vector of values at the centre of mass of the edges of \(\mathcal{T}\), 
\(\Pi_{\mathcal{D}} v_{\mathcal{D}}\) is piecewise linear function on \(\mathcal{T}\) that takes its values at these centre of mass,
and \(\nabla_{\mathcal{D}} v_{\mathcal{D}} = \nabla_{\mathcal{T}}(\Pi_{\mathcal{D}} v_{\mathcal{D}})\) is the broken gradient of \(\Pi_{\mathcal{D}} v_{\mathcal{D}}\). Then the gradient scheme \eqref{def:GS} gives a nonconforming $\mathbb{P}^1$ finite element approximation for the model \eqref{model:rtvf}.
\end{itemize} 
These methods are known to satisfy the properties of GDM under standard regularity of mesh, we refer to \cite{Droniou2018TheGD} for detailed analysis of these methods.
We emphasize that GDM serves purely as an analytical tool for studying the schemes. The actual numerical implementations remain scheme-specific and utilize standard procedures for each method.

The numerical simulations for the model \eqref{model:rtvf} are carried out on a unit square \(\Omega = (0,1)^2\) (with triangular mesh data given in Table \ref{tab:tri_mesh}) and final time \(T = 1\).
Since the diffusion term is non-linear, a non-linear iterative method is required to compute the approximate solution. 
Newton's method is a conventional choice due to its quadratic convergence.
At each time step, the initial guess in the Newton iterations is the approximate solution obtained in the previous time step,
and the iterations stop once the residue is less than the tolerance \(\textnormal{tol}\).
To validate our theoretical results, the relative errors between the approximate solution \(U\) and the manufactured solution \(\overline{u}\) are computed:
\begin{align*}
    E_{1} = \frac{\| \Pi_{\mathcal{D}}U - \overline{u}\|_{L^{\infty}(0,T; L^2(\Omega))}}{\| \overline{u}\|_{L^{\infty}(0,T; L^2(\Omega))}}, 
    E_{2} = \frac{\| \nabla_{\mathcal{D}} U - \nabla \overline{u}\|_{L^{1}(0,T; L^1(\Omega))}}{\| \nabla\overline{u}\|_{L^{1}(0,T; L^1(\Omega))}},  
    E_{3} = \frac{\| \nabla_{\mathcal{D}} U - \nabla \overline{u}\|_{L^{2}(0,T; L^2(\Omega))}}{\| \nabla\overline{u}\|_{L^{2}(0,T; L^2(\Omega))}}.
\end{align*}
The implementations are available at \url{https://github.com/HuatengZhu/matlab-NLGF}.
\begin{table}[ht!]
    \centering
    \caption{Test 1: Data for the triangular meshes.}%
    \label{tab:tri_mesh}
    \begin{tabular}{|c|c|c|c|c|}
    \hline
     Mesh& Size& Nb. cells & Nb. edges&Nb. vertices\\
    \hline
        mesh1&    0.250&   56&     92&     37\\
        mesh2&    0.125&   224&    352&    129\\
        mesh3&    0.063&   896&    1376&   481\\
        mesh4&    0.050&   1400&   2140&   741\\
        mesh5&    0.031&   3584&   5440&   1857\\
        mesh6&    0.016&   14336&  21632&  7297\\
    \hline
    \end{tabular}%
\end{table}%

\subsection*{Test 1} 
The aim of this test is to demonstrate and validate the convergence rate of CP1FEM and NCP1FEM for the model \eqref{model:rtvf} using a smooth manufactured solution. 
For this purpose,
we set the regularisation parameter to  \(\rho = 1\), 
the time step to \(\delta t = h\), 
the tolerance to \(\textnormal{tol} = 10^{-8}\),
employing the manufactured solution \[\overline{u} = \cos(t) \cos(\pi x) \cos(\pi y), \quad (x,y) \in \Omega.\]
A comparison of efficiency showed that CP1FEM required approximately 14 Newton iterations across all meshes, whereas NCP1FEM was significantly more efficient, requiring only around 2 Newton iterations on average; see Table \ref{tab:complexity FEM}.
The log-log convergence plots (see Figure \ref{Fig:orderofcv}) confirm the theoretical results:
the convergence rates for \(E_1\), \(E_2\), and \(E_3\) are around 1,
as predicted by Theorems \ref{thm:GDM Error Estimate Theorem} and \ref{thm:GDM Error Estimate Theorem 2}. 

\begin{figure}[t]
    \centering
    \begin{subfigure}[t]{0.32\textwidth}
        \centering
        \resizebox{1\textwidth}{!}{
%
%
\definecolor{mycolor1}{rgb}{0.06600,0.44300,0.74500}%
\definecolor{mycolor2}{rgb}{0.12941,0.12941,0.12941}%
\definecolor{mycolor3}{rgb}{0.23100,0.66600,0.19600}%
\begin{tikzpicture}

\begin{axis}[%
width=2.965in,
height=2.762in,
at={(1.806in,0.447in)},
scale only axis,
xmode=log,
xmin=0.005,
xmax=0.5,
xminorticks=false,
xlabel style={font=\Large\color{mycolor2}},
xlabel={$h$},
ymode=log,
ymin=0.0005,
ymax=0.7,
yminorticks=false,
axis background/.style={fill=white},
legend style={at={(0.03,0.97)}, anchor=north west, legend columns=2, legend cell align=left, align=left}
]

\addplot [color=mycolor3, mark=*, mark options={solid, mycolor3}]
  table[row sep=crcr]{%
0.25	0.232717114471374\\
0.125	0.0521203572816063\\
0.0625	0.0168776430035039\\
0.03125	0.00697323525968433\\
0.015625	0.00294381837579355\\
0.0078125	0.00113349679947479\\
};
\addlegendentry{CP1FEM}
\addplot [color=mycolor1, mark=asterisk, mark options={solid, mycolor1}]
  table[row sep=crcr]{%
0.25	0.137308567114891\\
0.125	0.0817506875369685\\
0.0625	0.0466461626145092\\
0.03125	0.0243668281716981\\
0.015625	0.0133145981396147\\
0.0078125	0.0061284215123162\\
};
\addlegendentry{NCP1FEM}

\addplot [color=black, line width=1.0pt, forget plot]
  table[row sep=crcr]{%
0.09	0.006\\
0.27	0.006\\
};
\addplot [color=black, line width=1.0pt, forget plot]
  table[row sep=crcr]{%
0.27	0.006\\
0.27	0.018\\
};
\addplot [color=black, line width=1.0pt, forget plot]
  table[row sep=crcr]{%
0.09	0.006\\
0.27	0.018\\
};
\node[centered, align=center, inner sep=0, font=\color{mycolor2}]
at (axis cs:0.156,0.0045) {1};
\node[right, align=left, inner sep=0, font=\color{mycolor2}]
at (axis cs:0.297,0.01) {1};
\end{axis}

\end{tikzpicture}%
}
        \caption{Error \(E_1\)}
    \end{subfigure}%
    \hfill
    \begin{subfigure}[t]{0.32\textwidth}
        \centering
        \resizebox{1\textwidth}{!}{
%
%
\definecolor{mycolor1}{rgb}{0.06600,0.44300,0.74500}%
\definecolor{mycolor2}{rgb}{0.12941,0.12941,0.12941}%
\definecolor{mycolor3}{rgb}{0.23100,0.66600,0.19600}%
\begin{tikzpicture}

\begin{axis}[%
width=2.965in,
height=2.762in,
at={(5.708in,0.447in)},
scale only axis,
xmode=log,
xmin=0.005,
xmax=0.5,
xminorticks=false,
xlabel style={font=\Large\color{mycolor2}},
xlabel={$h$},
ymode=log,
ymin=0.0005,
ymax=0.7,
yminorticks=false,
axis background/.style={fill=white},
legend style={at={(0.03,0.97)}, anchor=north west, legend columns=2, legend cell align=left, align=left}
]

\addplot [color=mycolor3, mark=*, mark options={solid, mycolor3}]
  table[row sep=crcr]{%
0.25	0.238531456046486\\
0.125	0.0571162115477063\\
0.0625	0.0206125087491065\\
0.03125	0.00924057733037063\\
0.015625	0.00427902042108399\\
0.0078125	0.0019423997941794\\
};
\addlegendentry{CP1FEM}

\addplot [color=mycolor1, mark=asterisk, mark options={solid, mycolor1}]
  table[row sep=crcr]{%
0.25	0.302007274810191\\
0.125	0.130552531785647\\
0.0625	0.0740396691756868\\
0.03125	0.0357285257838499\\
0.015625	0.0179482808696559\\
0.0078125	0.00913738832872458\\
};
\addlegendentry{NCP1FEM}

\addplot [color=black, line width=1.0pt, forget plot]
  table[row sep=crcr]{%
0.09	0.006\\
0.27	0.006\\
};
\addplot [color=black, line width=1.0pt, forget plot]
  table[row sep=crcr]{%
0.27	0.006\\
0.27	0.018\\
};
\addplot [color=black, line width=1.0pt, forget plot]
  table[row sep=crcr]{%
0.09	0.006\\
0.27	0.018\\
};
\node[centered, align=center, inner sep=0, font=\color{mycolor2}]
at (axis cs:0.156,0.0045) {1};
\node[right, align=left, inner sep=0, font=\color{mycolor2}]
at (axis cs:0.297,0.01) {1};
\end{axis}

\end{tikzpicture}%
}
        \caption{Error \(E_2\)}
    \end{subfigure}
     \hfill
    \begin{subfigure}[t]{0.32\textwidth}
        \centering
        \resizebox{1\textwidth}{!}{
%
%
\definecolor{mycolor1}{rgb}{0.06600,0.44300,0.74500}%
\definecolor{mycolor2}{rgb}{0.12941,0.12941,0.12941}%
\definecolor{mycolor3}{rgb}{0.23100,0.66600,0.19600}%
\begin{tikzpicture}

\begin{axis}[%
width=2.965in,
height=2.762in,
at={(9.61in,0.447in)},
scale only axis,
xmode=log,
xmin=0.005,
xmax=0.5,
xminorticks=false,
xlabel style={font= \Large\color{mycolor2}},
xlabel={$h$},
ymode=log,
ymin=0.0005,
ymax=0.7,
yminorticks=false,
axis background/.style={fill=white},
legend style={at={(0.03,0.97)}, anchor=north west, legend columns=2, legend cell align=left, align=left}
]
\addplot [color=mycolor3, mark=*, mark options={solid, mycolor3}]
  table[row sep=crcr]{%
0.25	0.23604378034304\\
0.125	0.0588325730873486\\
0.0625	0.0214876196761341\\
0.03125	0.00962314451519746\\
0.015625	0.00445469908688483\\
0.0078125	0.00202796167243279\\
};
\addlegendentry{CP1FEM}

\addplot [color=mycolor1, mark=asterisk, mark options={solid, mycolor1}]
  table[row sep=crcr]{%
0.25	0.443623549625686\\
0.125	0.152950341046065\\
0.0625	0.0868476032178225\\
0.03125	0.0415435232558743\\
0.015625	0.0206843304201662\\
0.0078125	0.0106652398439837\\
};
\addlegendentry{NCP1FEM}

\addplot [color=black, line width=1.0pt, forget plot]
  table[row sep=crcr]{%
0.09	0.006\\
0.27	0.006\\
};
\addplot [color=black, line width=1.0pt, forget plot]
  table[row sep=crcr]{%
0.27	0.006\\
0.27	0.018\\
};
\addplot [color=black, line width=1.0pt, forget plot]
  table[row sep=crcr]{%
0.09	0.006\\
0.27	0.018\\
};
\node[centered, align=center, inner sep=0, font=\color{mycolor2}]
at (axis cs:0.156,0.0045) {1};
\node[right, align=left, inner sep=0, font=\color{mycolor2}]
at (axis cs:0.297,0.01) {1};
\end{axis}

\end{tikzpicture}%
}
        \caption{Error \(E_3\)}
    \end{subfigure}
    \caption{Test 1:  Errors w.r.t.~$h$: CP1FEM vs. NCP1FEM.}
    \label{Fig:orderofcv}
\end{figure}
\begin{table}   
    \centering
    \caption{Test 1: Computational complexity.}
    \subfloat[CP1FEM]{%
    \begin{tabular}{|c|c|c|c|}
    \hline
        Mesh  &  DoF & Avg. Newt. & Avg. Cond.\\ 
    \hline
          mesh1 & 37 & 15  & 5.06e+01 \\    
         mesh2 &  129 & 15 & 1.09e+02  \\
          mesh3 & 481& 15  &  2.23e+02\\
          mesh4 & 741& 14  &  4.51e+02\\ 
         mesh5 &  1857& 13 &  8.92e+02 \\ 
         mesh6 &  7297& 11 &  1.75e+03 \\ 
    \hline
    \end{tabular}
    }
    \hfill
    \subfloat[NCP1FEM]{
    \begin{tabular}{|c|c|c|c|}
    \hline
         Mesh &  DoF & Avg. Newt. & Avg. Cond. \\ 
    \hline
         mesh1 & 92& 2  & 2.32e+02\\    
         mesh2 & 352& 2 & 4.53e+02 \\
         mesh3 &  1376& 2  & 9.98e+02 \\
         mesh4 &  2140& 2  & 2.12e+03 \\ 
         mesh5 &  5440& 2 & 4.25e+03  \\ 
         mesh6 & 21632&  2 & 8.47e+03 \\ 
    \hline
    \end{tabular}
    }
    \label{tab:complexity FEM}
\end{table}

\subsection*{Test 2} 

The aim of this test is to explore the error of CP1FEM and NCP1FEM for the model \eqref{model:rtvf} with various values of \(\rho \in \{1, 0.1, 0.01, 0.001, h^{1/2}, h\}\), using a different manufactured solution.
In this case, we choose the time step \(\delta t = h\) and the manufactured solution 
\[\overline{u} = \cos(t) x^3 (1-x)^2 y^3 (1-y)^2, \quad (x,y) \in \Omega.\]
We denote by \(E_i^{(\rho)}\) as relative error \(E_i\) of the model \eqref{model:rtvf} for a given \(\rho\).

Numerical results reveal that NCP1FEM is significantly more efficient,
demonstrating a typical run-time of approximately 400 seconds, compared to the 1400 seconds required by CP1FEM,
while the errors \(E_2\) and \(E_3\) of CP1FEM (see Figures \ref{Fig:orderofcv w.r.t. h & rho CFEM} and \ref{Fig:orderofcv w.r.t. rho CFEM 1}) are slightly lower than those of  NCP1FEM (see  Figures \ref{Fig:orderofcv w.r.t. h & rho NCFEM} and \ref{Fig:orderofcv w.r.t. rho NCFEM 1}).

A key observation related to the parameter \(\rho\) is its impact on the error magnitude and the convergence rate.
As \(\rho\) decreases,
the magnitude of \(E_1^{(\rho)}\) increases for both schemes, 
yet the magnitudes of \(E_2^{(\rho)}\) and \(E_3^{(\rho)}\) remain nearly unchanged.
Moreover, the convergence rate of \(E_1^{(\rho)}\) is directly influenced by \(\rho\), see Figures \ref{Fig:orderofcv w.r.t. h & rho CFEM} (A), \ref{Fig:orderofcv w.r.t. h & rho NCFEM} (A);
when \(\rho = 1\), the convergence rate for \(E_1^{(1)}\)  is around order 2 before saturation;
when \(\rho = h^{1/2}\), the convergence rate for \(E_1^{(1)}\) is around order 1.5;
when \(\rho = h\), the convergence rate for \(E_1^{(1)}\) is around order 1.
This dependency suggests that, for this test case, the hidden constant in the upper bound of \(E_1^{(\rho)}\) may depend on \(\rho^{-1}\).
With furthur investigations on \(\rho\)-dependency of the errors, we observe that the convergence rate for \(E_1^{(\rho)}\) with respect to \(\rho\) is around order \(-1\); see Figures \ref{Fig:orderofcv w.r.t. rho CFEM 1} (A), \ref{Fig:orderofcv w.r.t. rho NCFEM 1} (A).
While the influence of \(\rho\) on the convergence rate of \(E_1^{(\rho)}\) is clear, its effect on  the convergence rates of \(E_2^{(\rho)}\) and \(E_3^{(\rho)}\) seems to only be visible for very small values of $\rho$.

\begin{figure}[t]
    \centering
    \begin{subfigure}[t]{0.32\textwidth}
        \centering
         \resizebox{1\textwidth}{!}{
%
%
\definecolor{mycolor1}{rgb}{0.06600,0.44300,0.74500}%
\definecolor{mycolor2}{rgb}{0.12941,0.12941,0.12941}%
\definecolor{mycolor3}{rgb}{0.23100,0.66600,0.19600}%
\definecolor{mycolor4}{rgb}{0.86600,0.32900,0.00000}%
\begin{tikzpicture}

\begin{axis}[%
width=7cm,
height=7cm,
scale only axis,
xmode=log,
xmin=0.0013,
xmax=1.36,
xminorticks=false,
xlabel style={font=\color{mycolor2}},
xlabel={$h$},
ymode=log,
ymin=0.0001,
ymax=8.58942065609222,
yminorticks=false,
axis background/.style={fill=white},
legend style={at={(0.03,0.97)}, anchor=north west, legend cell align=left, align=left}
]

\addplot [color=mycolor1, mark=*, mark options={solid, mycolor1}]
  table[row sep=crcr]{%
0.25	0.665032335602349\\
0.125	0.147922194469749\\
0.0625	0.028773438909269\\
0.03125	0.00329772928337885\\
0.015625	0.00145691653710782\\
0.0078125	0.000168\\
};
\addlegendentry{$\rho\text{ = 1}$}

\addplot [color=black, line width=1.0pt, forget plot]
  table[row sep=crcr]{%
0.09	0.006\\
0.27	0.006\\
};
\addplot [color=black, line width=1.0pt, forget plot]
  table[row sep=crcr]{%
0.27	0.006\\
0.27	0.018\\
};
\addplot [color=black, line width=1.0pt, forget plot]
  table[row sep=crcr]{%
0.09	0.006\\
0.27	0.018\\
};
\node[centered, align=center, inner sep=0, font=\color{mycolor2}]
at (axis cs:0.156,0.0045) {1};
\node[right, align=left, inner sep=0, font=\color{mycolor2}]
at (axis cs:0.297,0.01) {1};

\addplot [color=mycolor3, mark=*, mark options={solid, mycolor3}]
  table[row sep=crcr]{%
0.25	1.38996700080975\\
0.125	0.473573767910165\\
0.0625	0.162155170375077\\
0.03125	0.0550694152085344\\
0.015625	0.0182977095827042\\
0.0078125	0.0058666927028307\\
};
\addlegendentry{$\rho\text{ = h}^{\text{1/2}}$}

\addplot [color=mycolor4, mark=*, mark options={solid, mycolor4}]
  table[row sep=crcr]{%
0.25	2.84978864125111\\
0.125	1.40151972647524\\
0.0625	0.69952071917909\\
0.03125	0.351017955179886\\
0.015625	0.176045090941556\\
0.0078125	0.0881852743545491\\
};
\addlegendentry{$\rho\text{ = h}$}

\end{axis}
\end{tikzpicture}%
}
        \caption{Error \(E_1^{(\rho)}\)}
    \end{subfigure}%
    \hfill
    \begin{subfigure}[t]{0.32\textwidth}
        \centering
         \resizebox{1\textwidth}{!}{
%
%
\definecolor{mycolor1}{rgb}{0.06600,0.44300,0.74500}%
\definecolor{mycolor2}{rgb}{0.12941,0.12941,0.12941}%
\definecolor{mycolor3}{rgb}{0.23100,0.66600,0.19600}%
\definecolor{mycolor4}{rgb}{0.86600,0.32900,0.00000}%
\begin{tikzpicture}

\begin{axis}[%
width=7cm,
height=7cm,
scale only axis,
xmode=log,
xmin=0.0013,
xmax=1.36,
xminorticks=false,
xlabel style={font=\color{mycolor2}},
xlabel={$h$},
ymode=log,
ymin=0.000197782957908523,
ymax=3.66299164134895,
yminorticks=false,
axis background/.style={fill=white},
legend style={at={(0.03,0.97)}, anchor=north west, legend cell align=left, align=left}
]
\addplot [color=mycolor1, mark=*, mark options={solid, mycolor1}]
  table[row sep=crcr]{%
0.25	0.206152324223097\\
0.125	0.057389370747214\\
0.0625	0.0237878031236634\\
0.03125	0.0112143448836325\\
0.015625	0.00552078119280779\\
0.0078125	0.00274933403465416\\
};
\addlegendentry{$\rho\text{ = 1}$}

\addplot [color=black, line width=1.0pt, forget plot]
  table[row sep=crcr]{%
0.09	0.006\\
0.27	0.006\\
};
\addplot [color=black, line width=1.0pt, forget plot]
  table[row sep=crcr]{%
0.27	0.006\\
0.27	0.018\\
};
\addplot [color=black, line width=1.0pt, forget plot]
  table[row sep=crcr]{%
0.09	0.006\\
0.27	0.018\\
};
\node[centered, align=center, inner sep=0, font=\color{mycolor2}]
at (axis cs:0.156,0.0045) {1};
\node[right, align=left, inner sep=0, font=\color{mycolor2}]
at (axis cs:0.297,0.01) {1};
\addplot [color=mycolor3, mark=*, mark options={solid, mycolor3}]
  table[row sep=crcr]{%
0.25	0.209973429145426\\
0.125	0.0582167205290612\\
0.0625	0.0238867002507547\\
0.03125	0.0112257937391868\\
0.015625	0.00552058849965087\\
0.0078125	0.00274800751518197\\
};
\addlegendentry{$\rho\text{ = h}^{\text{1/2}}$}

\addplot [color=mycolor4, mark=*, mark options={solid, mycolor4}]
  table[row sep=crcr]{%
0.25	0.211909365928811\\
0.125	0.0585617851731707\\
0.0625	0.0239347589302422\\
0.03125	0.0112408605235547\\
0.015625	0.00553209127300957\\
0.0078125	0.00276389444707746\\
};
\addlegendentry{$\rho\text{ = h}$}

\end{axis}

\end{tikzpicture}%
}
        \caption{Error \(E_2^{(\rho)}\)}
    \end{subfigure}
     \hfill
    \begin{subfigure}[t]{0.32\textwidth}
        \centering
         \resizebox{1\textwidth}{!}{
%
%
\definecolor{mycolor1}{rgb}{0.06600,0.44300,0.74500}%
\definecolor{mycolor2}{rgb}{0.12941,0.12941,0.12941}%
\definecolor{mycolor3}{rgb}{0.23100,0.66600,0.19600}%
\definecolor{mycolor4}{rgb}{0.86600,0.32900,0.00000}%
\begin{tikzpicture}
\begin{axis}[%
width=7cm,
height=7cm,
scale only axis,
xmode=log,
xmin=0.0013,
xmax=1.36,
xminorticks=false,
xlabel style={font=\color{mycolor2}},
xlabel={$h$},
ymode=log,
ymin=0.000156926514358767,
ymax=2.90631971773852,
yminorticks=false,
axis background/.style={fill=white},
legend style={at={(0.03,0.97)}, anchor=north west, legend cell align=left, align=left}
]
\addplot [color=mycolor1, mark=*, mark options={solid, mycolor1}]
  table[row sep=crcr]{%
0.25	0.21187428478186\\
0.125	0.0589090902677681\\
0.0625	0.0244481287110046\\
0.03125	0.0117828169913396\\
0.015625	0.00585171158822876\\
0.0078125	0.00292188932666602\\
};
\addlegendentry{$\rho\text{ = 1}$}

\addplot [color=mycolor3, mark=*, mark options={solid, mycolor3}]
  table[row sep=crcr]{%
0.25	0.214570706730641\\
0.125	0.0596169302800949\\
0.0625	0.0245552154514822\\
0.03125	0.0117938543859436\\
0.015625	0.00585132936638554\\
0.0078125	0.00292103069926452\\
};
\addlegendentry{$\rho\text{ = h}^{\text{1/2}}$}

\addplot [color=mycolor4, mark=*, mark options={solid, mycolor4}]
  table[row sep=crcr]{%
0.25	0.21606633969126\\
0.125	0.0599418086889611\\
0.0625	0.0246100849081296\\
0.03125	0.0118115843803502\\
0.015625	0.00586596814433315\\
0.0078125	0.0029424904210736\\
};
\addlegendentry{$\rho\text{ = h}$}

\addplot [color=black, line width=1.0pt, forget plot]
  table[row sep=crcr]{%
0.09	0.006\\
0.27	0.006\\
};
\addplot [color=black, line width=1.0pt, forget plot]
  table[row sep=crcr]{%
0.27	0.006\\
0.27	0.018\\
};
\addplot [color=black, line width=1.0pt, forget plot]
  table[row sep=crcr]{%
0.09	0.006\\
0.27	0.018\\
};
\node[centered, align=center, inner sep=0, font=\color{mycolor2}]
at (axis cs:0.156,0.0045) {1};
\node[right, align=left, inner sep=0, font=\color{mycolor2}]
at (axis cs:0.297,0.01) {1};
\end{axis}

\end{tikzpicture}%
}
        \caption{Error \(E_3^{(\rho)}\)}
    \end{subfigure}
    \caption{Test 2: Errors w.r.t.~$h$ of CP1FEM with \(\rho = 1, h^{1/2}, h\).}
    \label{Fig:orderofcv w.r.t. h & rho CFEM}
\end{figure}

\begin{figure}[t]
    \centering
    \begin{subfigure}[t]{0.32\textwidth}
        \centering
         \resizebox{1\textwidth}{!}{
%
%
\definecolor{mycolor1}{rgb}{0.06600,0.44300,0.74500}%
\definecolor{mycolor2}{rgb}{0.86600,0.32900,0.00000}%
\definecolor{mycolor3}{rgb}{0.92900,0.69400,0.12500}%
\definecolor{mycolor4}{rgb}{0.52100,0.08600,0.81900}%
\definecolor{mycolor5}{rgb}{0.12941,0.12941,0.12941}%
\begin{tikzpicture}

\begin{axis}[%
width=7cm,
height=7cm,
scale only axis,
xmode=log,
xmin=0.0005,
xmax=5,
xminorticks=false,
xlabel style={font=\color{mycolor5}},
xlabel={$\rho$},
ymode=log,
ymin=0.0001,
ymax=500,
yminorticks=false,
axis background/.style={fill=white},
legend style={legend cell align=left, align=left}
]
\addplot [color=mycolor1, mark=asterisk, mark options={solid, mycolor1}]
  table[row sep=crcr]{%
0.001	54.4\\
0.01	1.13\\
0.1	0.104\\
1	0.0033\\
};
\addlegendentry{mesh4}

\addplot [color=mycolor2, mark=square, mark options={solid, mycolor2}]
  table[row sep=crcr]{%
0.001	5.27\\
0.01	0.278\\
0.1	0.0239\\
1	0.00146\\
};
\addlegendentry{mesh5}

\addplot [color=mycolor3, mark=triangle, mark options={solid, mycolor3}]
  table[row sep=crcr]{%
0.001	0.772\\
0.01	0.0684\\
0.1	0.00494\\
1	0.000168\\
};
\addlegendentry{mesh6}

\addplot [color=mycolor4]
  table[row sep=crcr]{%
0.001	100\\
0.01	10\\
0.1	1\\
1	0.1\\
};
\addlegendentry{slope \(-1\)}

\end{axis}

\end{tikzpicture}%
}
        \caption{Error \(E_1^{(\rho)}\)}
    \end{subfigure}%
    \hfill
    \begin{subfigure}[t]{0.32\textwidth}
        \centering
         \resizebox{1\textwidth}{!}{
%
%
\definecolor{mycolor1}{rgb}{0.06600,0.44300,0.74500}%
\definecolor{mycolor2}{rgb}{0.86600,0.32900,0.00000}%
\definecolor{mycolor3}{rgb}{0.92900,0.69400,0.12500}%
\definecolor{mycolor4}{rgb}{0.52100,0.08600,0.81900}%
\definecolor{mycolor5}{rgb}{0.12941,0.12941,0.12941}%
\begin{tikzpicture}

\begin{axis}[%
width=7cm,
height=7cm,
scale only axis,
xmode=log,
xmin=0.0005,
xmax=5,
xminorticks=false,
xlabel style={font=\color{mycolor5}},
xlabel={$\rho$},
ymode=log,
ymin=0.0005,
ymax=500,
yminorticks=false,
axis background/.style={fill=white},
legend style={legend cell align=left, align=left}
]
\addplot [color=mycolor1, mark=asterisk, mark options={solid, mycolor1}]
  table[row sep=crcr]{%
0.001	0.0539\\
0.01	0.0113\\
0.1	0.0112\\
1	0.0112\\
};
\addlegendentry{mesh4}

\addplot [color=mycolor2, mark=square, mark options={solid, mycolor2}]
  table[row sep=crcr]{%
0.001	0.0103\\
0.01	0.00555\\
0.1	0.00552\\
1	0.00552\\
};
\addlegendentry{mesh5}

\addplot [color=mycolor3, mark=triangle, mark options={solid, mycolor3}]
  table[row sep=crcr]{%
0.001	0.00386\\
0.01	0.00276\\
0.1	0.00275\\
1	0.00275\\
};
\addlegendentry{mesh6}

\end{axis}

\end{tikzpicture}%
}
        \caption{Error \(E_2^{(\rho)}\)}
    \end{subfigure}
    \hfill
    \begin{subfigure}[t]{0.32\textwidth}
        \centering
         \resizebox{1\textwidth}{!}{
%
%
\definecolor{mycolor1}{rgb}{0.06600,0.44300,0.74500}%
\definecolor{mycolor2}{rgb}{0.86600,0.32900,0.00000}%
\definecolor{mycolor3}{rgb}{0.92900,0.69400,0.12500}%
\definecolor{mycolor4}{rgb}{0.52100,0.08600,0.81900}%
\definecolor{mycolor5}{rgb}{0.12941,0.12941,0.12941}%
\begin{tikzpicture}

\begin{axis}[%
width=7cm,
height=7cm,
scale only axis,
xmode=log,
xmin=0.0005,
xmax=5,
xminorticks=false,
xlabel style={font=\color{mycolor5}},
xlabel={$\rho$},
ymode=log,
ymin=0.0005,
ymax=500,
yminorticks=false,
axis background/.style={fill=white},
legend style={legend cell align=left, align=left}
]
\addplot [color=mycolor1, mark=asterisk, mark options={solid, mycolor1}]
  table[row sep=crcr]{%
0.001	0.0716\\
0.01	0.0119\\
0.1	0.0118\\
1	0.0118\\
};
\addlegendentry{mesh4}

\addplot [color=mycolor2, mark=square, mark options={solid, mycolor2}]
  table[row sep=crcr]{%
0.001	0.0124\\
0.01	0.00589\\
0.1	0.00585\\
1	0.00585\\
};
\addlegendentry{mesh5}

\addplot [color=mycolor3, mark=triangle, mark options={solid, mycolor3}]
  table[row sep=crcr]{%
0.001	0.00443\\
0.01	0.00293\\
0.1	0.00292\\
1	0.00292\\
};
\addlegendentry{mesh6}

\end{axis}

\end{tikzpicture}%
}
        \caption{Error \(E_3^{(\rho)}\)}
    \end{subfigure}
\caption{Test 2: Errors w.r.t.~\(\rho\) of CP1FEM on mesh4, mesh5, and mesh6.}
\label{Fig:orderofcv w.r.t. rho CFEM 1}
\end{figure}

\begin{figure}[t]
    \centering
    \begin{subfigure}[t]{0.32\textwidth}
        \centering
         \resizebox{1\textwidth}{!}{
%
%
\definecolor{mycolor1}{rgb}{0.06600,0.44300,0.74500}%
\definecolor{mycolor2}{rgb}{0.12941,0.12941,0.12941}%
\definecolor{mycolor3}{rgb}{0.23100,0.66600,0.19600}%
\definecolor{mycolor4}{rgb}{0.86600,0.32900,0.00000}%
\begin{tikzpicture}

\begin{axis}[%
width=7cm,
height=7cm,
scale only axis,
separate axis lines,
every outer x axis line/.append style={mycolor2},
every x tick label/.append style={font=\color{mycolor2}},
every x tick/.append style={mycolor2},
xmode=log,
xmin=0.0013,
xmax=1.36,
xminorticks=false,
xlabel style={font=\color{mycolor2}},
xlabel={$h$},
every outer y axis line/.append style={mycolor2},
every y tick label/.append style={font=\color{mycolor2}},
every y tick/.append style={mycolor2},
ymode=log,
ymin=0.0001,
ymax=7.43675595834041,
yminorticks=false,
axis background/.style={fill=white},
legend style={at={(0.03,0.97)}, anchor=north west, legend cell align=left, align=left, draw=mycolor2, fill=white}
]
\addplot [color=mycolor1, mark=asterisk, mark options={solid, mycolor1}]
  table[row sep=crcr]{%
0.25	0.621238308546587\\
0.125	0.144943247720334\\
0.0625	0.028561831496851\\
0.03125	0.00324607426353482\\
0.015625	0.00144776085624211\\
0.0078125	0.000361\\
};
\addlegendentry{$\rho\text{ = 1}$}

\addplot [color=black, line width=1.0pt, forget plot]
  table[row sep=crcr]{%
0.09	0.006\\
0.27	0.006\\
};
\addplot [color=black, line width=1.0pt, forget plot]
  table[row sep=crcr]{%
0.27	0.006\\
0.27	0.018\\
};
\addplot [color=black, line width=1.0pt, forget plot]
  table[row sep=crcr]{%
0.09	0.006\\
0.27	0.018\\
};
\node[centered, align=center, inner sep=0, font=\color{mycolor2}]
at (axis cs:0.156,0.0045) {1};
\node[right, align=left, inner sep=0, font=\color{mycolor2}]
at (axis cs:0.297,0.01) {1};
\addplot [color=mycolor3, mark=asterisk, mark options={solid, mycolor3}]
  table[row sep=crcr]{%
0.25	1.29676113978006\\
0.125	0.464276160434116\\
0.0625	0.161325676789106\\
0.03125	0.0549971150314164\\
0.015625	0.0182914419997227\\
0.0078125	0.00586430038733626\\
};
\addlegendentry{$\rho\text{ = h}^{\text{1/2}}$}

\addplot [color=mycolor4, mark=asterisk, mark options={solid, mycolor4}]
  table[row sep=crcr]{%
0.25	2.65776302165729\\
0.125	1.37403215999192\\
0.0625	0.69597974985572\\
0.03125	0.350570086062636\\
0.015625	0.175988842134239\\
0.0078125	0.0881782210412607\\
};
\addlegendentry{$\rho\text{ = h}$}
\end{axis}
\end{tikzpicture}%
}
        \caption{Error \(E_1^{(\rho)}\)}
    \end{subfigure}%
    \hfill
    \begin{subfigure}[t]{0.32\textwidth}
        \centering
         \resizebox{1\textwidth}{!}{
%
%
\definecolor{mycolor1}{rgb}{0.06600,0.44300,0.74500}%
\definecolor{mycolor2}{rgb}{0.12941,0.12941,0.12941}%
\definecolor{mycolor3}{rgb}{0.23100,0.66600,0.19600}%
\definecolor{mycolor4}{rgb}{0.86600,0.32900,0.00000}%
\begin{tikzpicture}
\begin{axis}[%
width=7cm,
height=7cm,
scale only axis,
separate axis lines,
every outer x axis line/.append style={mycolor2},
every x tick label/.append style={font=\color{mycolor2}},
every x tick/.append style={mycolor2},
xmode=log,
xmin=0.0013,
xmax=1.36,
xminorticks=false,
xlabel style={font=\color{mycolor2}},
xlabel={$h$},
every outer y axis line/.append style={mycolor2},
every y tick label/.append style={font=\color{mycolor2}},
every y tick/.append style={mycolor2},
ymode=log,
ymin=0.000159723101372596,
ymax=2.95811323404695,
yminorticks=false,
axis background/.style={fill=white},
legend style={at={(0.03,0.97)}, anchor=north west, legend cell align=left, align=left, draw=mycolor2, fill=white}
]
\addplot [color=mycolor1, mark=asterisk, mark options={solid, mycolor1}]
  table[row sep=crcr]{%
0.25	0.219750588117238\\
0.125	0.0805802852749884\\
0.0625	0.0358239894249939\\
0.03125	0.0173028835163314\\
0.015625	0.00857431361694942\\
0.0078125	0.00427659803021997\\
};
\addlegendentry{$\rho\text{ = 1}$}
\addplot [color=black, line width=1.0pt, forget plot]
  table[row sep=crcr]{%
0.09	0.006\\
0.27	0.006\\
};
\addplot [color=black, line width=1.0pt, forget plot]
  table[row sep=crcr]{%
0.27	0.006\\
0.27	0.018\\
};
\addplot [color=black, line width=1.0pt, forget plot]
  table[row sep=crcr]{%
0.09	0.006\\
0.27	0.018\\
};
\node[centered, align=center, inner sep=0, font=\color{mycolor2}]
at (axis cs:0.156,0.0045) {1};
\node[right, align=left, inner sep=0, font=\color{mycolor2}]
at (axis cs:0.297,0.01) {1};
\addplot [color=mycolor3, mark=asterisk, mark options={solid, mycolor3}]
  table[row sep=crcr]{%
0.25	0.222895583132694\\
0.125	0.0813046707841766\\
0.0625	0.0359644738585385\\
0.03125	0.0173300658585328\\
0.015625	0.00857960347028611\\
0.0078125	0.00427764541292005\\
};
\addlegendentry{$\rho\text{ = h}^{\text{1/2}}$}
\addplot [color=mycolor4, mark=asterisk, mark options={solid, mycolor4}]
  table[row sep=crcr]{%
0.25	0.224492883122536\\
0.125	0.081589598117694\\
0.0625	0.0360122415615734\\
0.03125	0.0173523206995859\\
0.015625	0.00862272401782041\\
0.0078125	0.00438203603180168\\
};
\addlegendentry{$\rho\text{ = h}$}
\end{axis}
\end{tikzpicture}%
}
        \caption{Error \(E_2^{(\rho)}\)}
    \end{subfigure}
     \hfill
    \begin{subfigure}[t]{0.32\textwidth}
        \centering
         \resizebox{1\textwidth}{!}{
%
%
\definecolor{mycolor1}{rgb}{0.06600,0.44300,0.74500}%
\definecolor{mycolor2}{rgb}{0.12941,0.12941,0.12941}%
\definecolor{mycolor3}{rgb}{0.23100,0.66600,0.19600}%
\definecolor{mycolor4}{rgb}{0.86600,0.32900,0.00000}%
\begin{tikzpicture}

\begin{axis}[%
width=7cm,
height=7cm,
scale only axis,
separate axis lines,
every outer x axis line/.append style={mycolor2},
every x tick label/.append style={font=\color{mycolor2}},
every x tick/.append style={mycolor2},
xmode=log,
xmin=0.0013,
xmax=1.36,
xminorticks=false,
xlabel style={font=\color{mycolor2}},
xlabel={$h$},
every outer y axis line/.append style={mycolor2},
every y tick label/.append style={font=\color{mycolor2}},
every y tick/.append style={mycolor2},
ymode=log,
ymin=0.000139054922982024,
ymax=2.57533321352773,
yminorticks=false,
axis background/.style={fill=white},
legend style={at={(0.03,0.97)}, anchor=north west, legend cell align=left, align=left, draw=mycolor2, fill=white}
]
\addplot [color=mycolor1, mark=asterisk, mark options={solid, mycolor1}]
  table[row sep=crcr]{%
0.25	0.228195972456008\\
0.125	0.0808641930196969\\
0.0625	0.0356088936153063\\
0.03125	0.0171067164470741\\
0.015625	0.00845741348182169\\
0.0078125	0.0042152169940852\\
};
\addlegendentry{$\rho\text{ = 1}$}

\addplot [color=black, line width=1.0pt, forget plot]
  table[row sep=crcr]{%
0.09	0.006\\
0.27	0.006\\
};
\addplot [color=black, line width=1.0pt, forget plot]
  table[row sep=crcr]{%
0.27	0.006\\
0.27	0.018\\
};
\addplot [color=black, line width=1.0pt, forget plot]
  table[row sep=crcr]{%
0.09	0.006\\
0.27	0.018\\
};
\node[centered, align=center, inner sep=0, font=\color{mycolor2}]
at (axis cs:0.156,0.0045) {1};
\node[right, align=left, inner sep=0, font=\color{mycolor2}]
at (axis cs:0.297,0.01) {1};
\addplot [color=mycolor3, mark=asterisk, mark options={solid, mycolor3}]
  table[row sep=crcr]{%
0.25	0.230662182379782\\
0.125	0.0814751574703414\\
0.0625	0.0357180900847803\\
0.03125	0.017124928491235\\
0.015625	0.00846019280753135\\
0.0078125	0.0042154895989794\\
};
\addlegendentry{$\rho\text{ = h}^{\text{1/2}}$}
 
\addplot [color=mycolor4, mark=asterisk, mark options={solid, mycolor4}]
  table[row sep=crcr]{%
0.25	0.232016967288908\\
0.125	0.0817322630277832\\
0.0625	0.0357536978263536\\
0.03125	0.0171356060469992\\
0.015625	0.0084853817240661\\
0.0078125	0.00429890714275244\\
};
\addlegendentry{$\rho\text{ = h}$}
 
\end{axis}
\end{tikzpicture}%
}
        \caption{Error \(E_3^{(\rho)}\)}
    \end{subfigure}
    \caption{Test 2: Errors w.r.t.~$h$ of NCP1FEM with  \(\rho = 1, h^{1/2}, h\).}
    \label{Fig:orderofcv w.r.t. h & rho NCFEM}
\end{figure}

\begin{figure}[t]
    \centering
    \begin{subfigure}[t]{0.32\textwidth}
        \centering
         \resizebox{1\textwidth}{!}{
%
%
\definecolor{mycolor1}{rgb}{0.06600,0.44300,0.74500}%
\definecolor{mycolor2}{rgb}{0.86600,0.32900,0.00000}%
\definecolor{mycolor3}{rgb}{0.92900,0.69400,0.12500}%
\definecolor{mycolor4}{rgb}{0.52100,0.08600,0.81900}%
\definecolor{mycolor5}{rgb}{0.12941,0.12941,0.12941}%
\begin{tikzpicture}

\begin{axis}[%
width=7cm,
height=7cm,
scale only axis,
xmode=log,
xmin=0.0005,
xmax=5,
xminorticks=false,
xlabel style={font=\color{mycolor5}},
xlabel={$\rho$},
ymode=log,
ymin=0.0001,
ymax=500,
yminorticks=false,
axis background/.style={fill=white},
legend style={legend cell align=left, align=left}
]
\addplot [color=mycolor1, mark=asterisk, mark options={solid, mycolor1}]
  table[row sep=crcr]{%
0.001	54.3\\
0.01	1.12\\
0.1	0.104\\
1	0.00325\\
};
\addlegendentry{mesh4}

\addplot [color=mycolor2, mark=square, mark options={solid, mycolor2}]
  table[row sep=crcr]{%
0.001	5.27\\
0.01	0.278\\
0.1	0.0239\\
1	0.00145\\
};
\addlegendentry{mesh5}

\addplot [color=mycolor3, mark=triangle, mark options={solid, mycolor3}]
  table[row sep=crcr]{%
0.001	0.772\\
0.01	0.0684\\
0.1	0.00494\\
1	0.000361\\
};
\addlegendentry{mesh6}

\addplot [color=mycolor4]
  table[row sep=crcr]{%
0.001	100\\
0.01	10\\
0.1	1\\
1	0.1\\
};
\addlegendentry{slope \(-1\)}

\end{axis}

\end{tikzpicture}%
}
        \caption{Error \(E_1^{(\rho)}\)}
    \end{subfigure}%
    \hfill
    \begin{subfigure}[t]{0.32\textwidth}
        \centering
         \resizebox{1\textwidth}{!}{
%
%
\definecolor{mycolor1}{rgb}{0.06600,0.44300,0.74500}%
\definecolor{mycolor2}{rgb}{0.86600,0.32900,0.00000}%
\definecolor{mycolor3}{rgb}{0.92900,0.69400,0.12500}%
\definecolor{mycolor4}{rgb}{0.52100,0.08600,0.81900}%
\definecolor{mycolor5}{rgb}{0.12941,0.12941,0.12941}%
\begin{tikzpicture}

\begin{axis}[%
width=7cm,
height=7cm,
scale only axis,
xmode=log,
xmin=0.0005,
xmax=5,
xminorticks=false,
xlabel style={font=\color{mycolor5}},
xlabel={$\rho$},
ymode=log,
ymin=0.0005,
ymax=500,
yminorticks=false,
axis background/.style={fill=white},
legend style={legend cell align=left, align=left}
]
\addplot [color=mycolor1, mark=asterisk, mark options={solid, mycolor1}]
  table[row sep=crcr]{%
0.001	0.0853\\
0.01	0.0176\\
0.1	0.0173\\
1	0.0173\\
};
\addlegendentry{mesh4}

\addplot [color=mycolor2, mark=square, mark options={solid, mycolor2}]
  table[row sep=crcr]{%
0.001	0.0388\\
0.01	0.00872\\
0.1	0.00858\\
1	0.00857\\
};
\addlegendentry{mesh5}

\addplot [color=mycolor3, mark=triangle, mark options={solid, mycolor3}]
  table[row sep=crcr]{%
0.001	0.0194\\
0.01	0.00433\\
0.1	0.00428\\
1	0.00428\\
};
\addlegendentry{mesh6}

\end{axis}

\end{tikzpicture}%
}
        \caption{Error \(E_2^{(\rho)}\)}
    \end{subfigure}
    \hfill
    \begin{subfigure}[t]{0.32\textwidth}
        \centering
         \resizebox{1\textwidth}{!}{
%
%
\definecolor{mycolor1}{rgb}{0.06600,0.44300,0.74500}%
\definecolor{mycolor2}{rgb}{0.86600,0.32900,0.00000}%
\definecolor{mycolor3}{rgb}{0.92900,0.69400,0.12500}%
\definecolor{mycolor4}{rgb}{0.52100,0.08600,0.81900}%
\definecolor{mycolor5}{rgb}{0.12941,0.12941,0.12941}%
\begin{tikzpicture}

\begin{axis}[%
width=7cm,
height=7cm,
scale only axis,
xmode=log,
xmin=0.0005,
xmax=5,
xminorticks=false,
xlabel style={font=\color{mycolor5}},
xlabel={$\rho$},
ymode=log,
ymin=0.0005,
ymax=500,
yminorticks=false,
axis background/.style={fill=white},
legend style={legend cell align=left, align=left}
]
\addplot [color=mycolor1, mark=asterisk, mark options={solid, mycolor1}]
  table[row sep=crcr]{%
0.001	0.115\\
0.01	0.0173\\
0.1	0.0171\\
1	0.0171\\
};
\addlegendentry{mesh4}

\addplot [color=mycolor2, mark=square, mark options={solid, mycolor2}]
  table[row sep=crcr]{%
0.001	0.0562\\
0.01	0.00856\\
0.1	0.00846\\
1	0.00846\\
};
\addlegendentry{mesh5}

\addplot [color=mycolor3, mark=triangle, mark options={solid, mycolor3}]
  table[row sep=crcr]{%
0.001	0.0282\\
0.01	0.00426\\
0.1	0.00422\\
1	0.00422\\
};
\addlegendentry{mesh6}

\end{axis}

\end{tikzpicture}%
}
        \caption{Error \(E_3^{(\rho)}\)}
    \end{subfigure}
\caption{Test 2: Errors w.r.t.~\(\rho\) of NCP1FEM on mesh4, mesh5, and mesh6.}
\label{Fig:orderofcv w.r.t. rho NCFEM 1}
\end{figure}

\subsection*{Test 3}
The aim of this test is to observe the evolution of the numerical solution of the schemes for  
\begin{equation}\label{eq:MSF}
    \partial_t u = \text{div} \, \bigg( \frac{\nabla u}{\sqrt{1 + |\nabla u|^2}} \bigg)
\end{equation}
which is imposed with a non-smooth initial condition
\[u_0 = S(x,y) + x(1-x)y(1-y) + 0.25\mathbbm{1}_{B}, \quad (x,y) \in \Omega.\]
where 
\begin{equation*}
    S(x,y) = 0.25 + \ln\bigg( \frac{\cos(y - 0.5)}{\cos(x - 0.5)}\bigg)
    \text{ and }
    B:= \{(x,y): |x-0.5|^2 + |y - 0.5|^2 \leq 0.16\};
\end{equation*}
and boundary conditions \(u|_{\partial \Omega} = S|_{\partial \Omega}.\)
Since \(\Omega\) is convex and the surface boundary \(S|_{\partial \Omega}\) admits a bijective projection with \(\partial \Omega\), 
there exists at most one minimal surface with the boundary \(S|_{\partial \Omega}\).
Moreover, the surface \(S\) has zero mean curvature everywhere on \(\Omega\) (see Lemma \ref{lem:S zero mean curvature}), so \(S\) is the unique minimal surface with the boundary \(S|_{\partial \Omega}\), see, e.g. \cite[Chapter 4.9]{dierkes2010minimal}.
Since the solution of \eqref{eq:MSF} evolves in the direction that decreases the surface area till reaching zero mean curvature, the expected minimal surface from \eqref{eq:MSF} is \(S\).

The parameters for this simulation are: the final time \(T =1\), the time step \(\delta t = 10^{-3}\) and the tolerance \(\textnormal{tol} = 10^{-12}\).
Figures \ref{fig:CFEM snapshots of numerical solution at time t} and \ref{fig:NCFEM snapshots of numerical solution at time t} demonstrate that both schemes of CP1FEM and NCP1FEM effectively minimize the surface area over time.
Figure \ref{fig:comparison numerical sol and exp} confirms that our numerical solutions align closely with the nodal interpolation of \(S\), validating the accuracy of our schemes.

\begin{figure}[t]
    \centering
    \begin{subfigure}[t]{0.32\textwidth}
        \centering
         \includegraphics[width=1\textwidth]{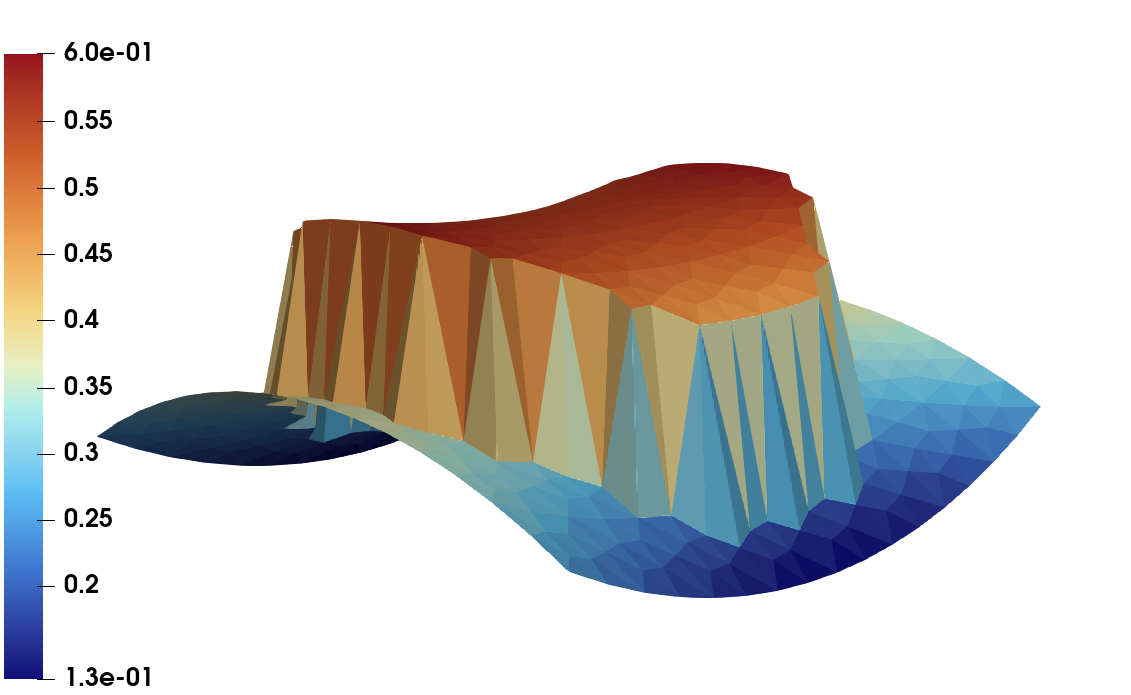}
        \caption{t = 0}
    \end{subfigure}%
    \hfill
    \begin{subfigure}[t]{0.32\textwidth}
        \centering
         \includegraphics[width=1\textwidth]{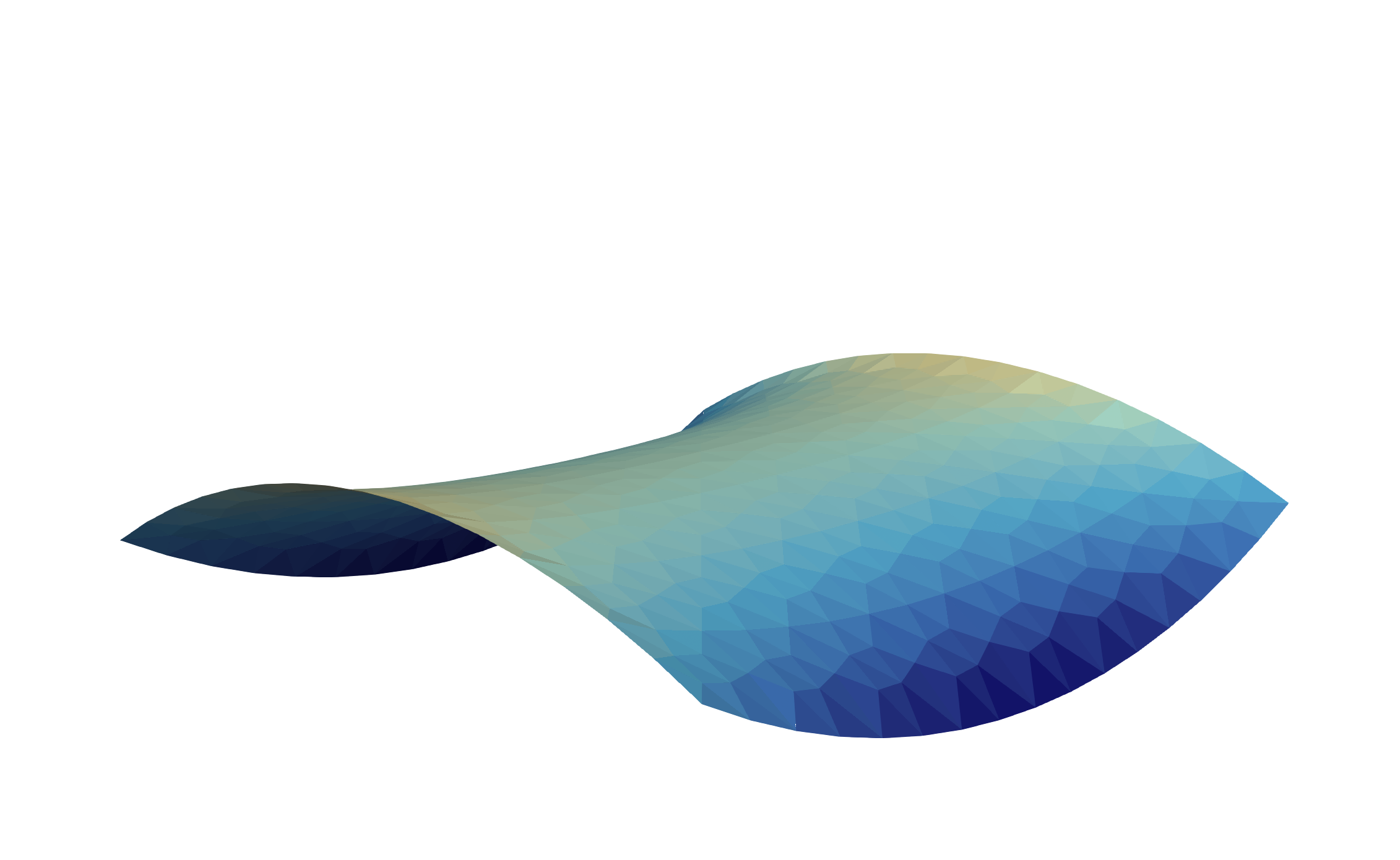}
        \caption{t = 0.1}
    \end{subfigure}
    \hfill
    \begin{subfigure}[t]{0.32\textwidth}
        \centering
         \includegraphics[width=1\textwidth]{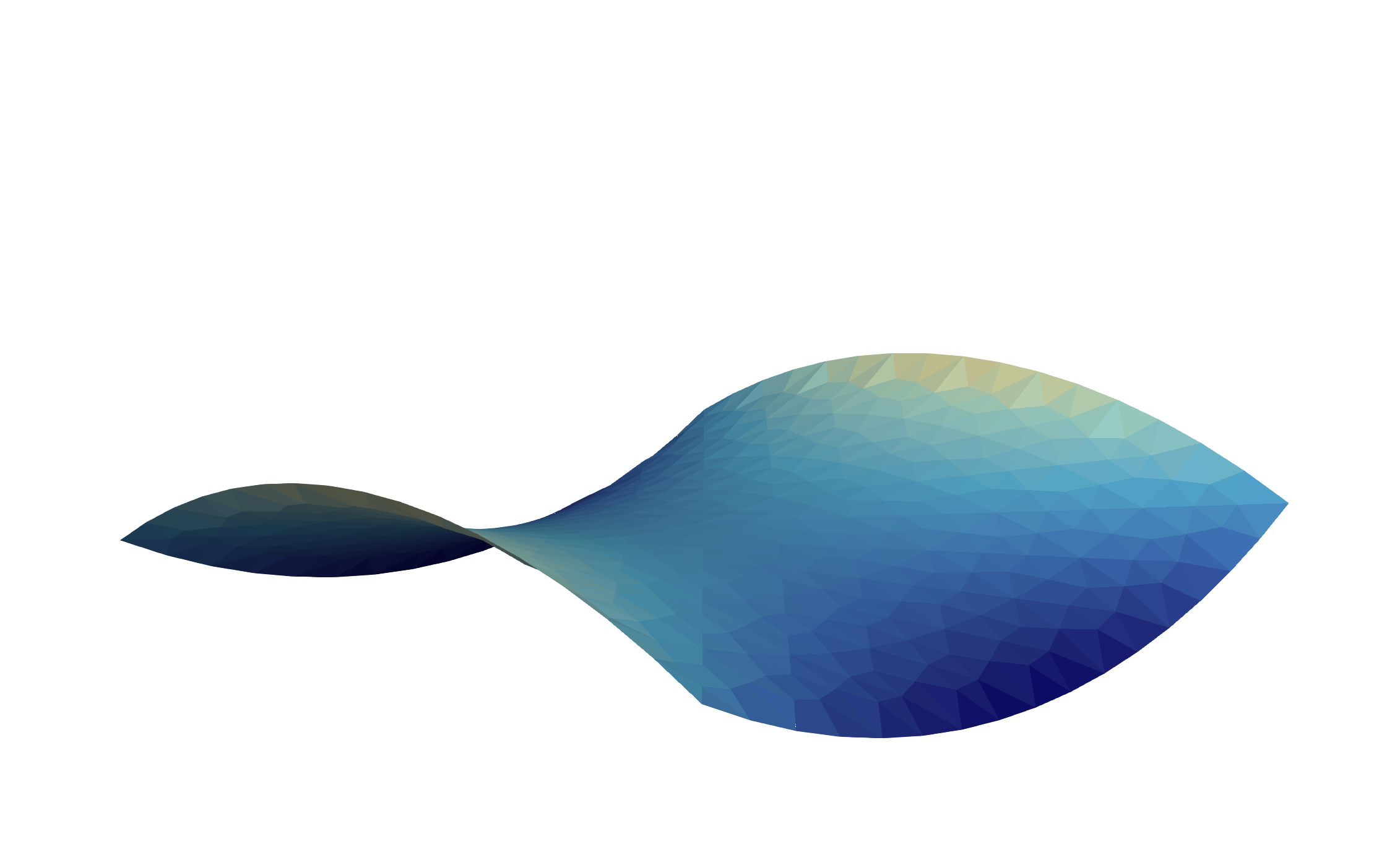}
        \caption{t = 1}
    \end{subfigure}
\caption{Test 3: CP1FEM, snapshots of numerical solution at the time \(t\).}
\label{fig:CFEM snapshots of numerical solution at time t}
\end{figure}

\begin{figure}[t]
    \centering
    \begin{subfigure}[t]{0.32\textwidth}
        \centering
         \includegraphics[width=1\textwidth]{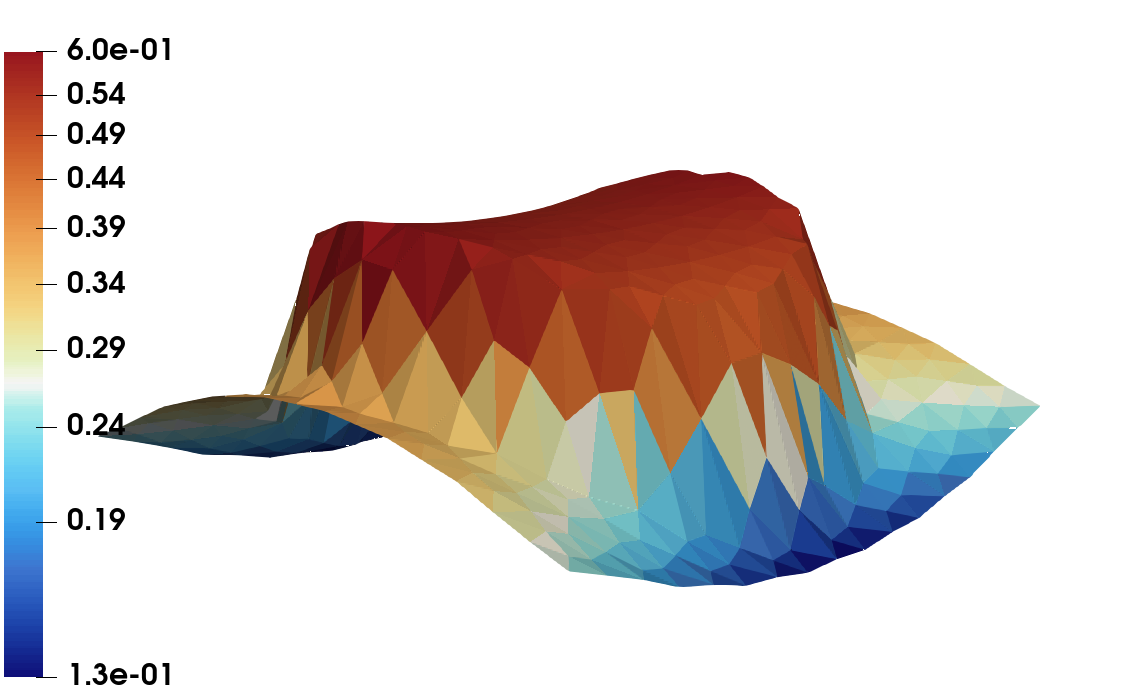}
        \caption{t = 0}
    \end{subfigure}%
    \hfill
    \begin{subfigure}[t]{0.32\textwidth}
        \centering
         \includegraphics[width=1\textwidth]{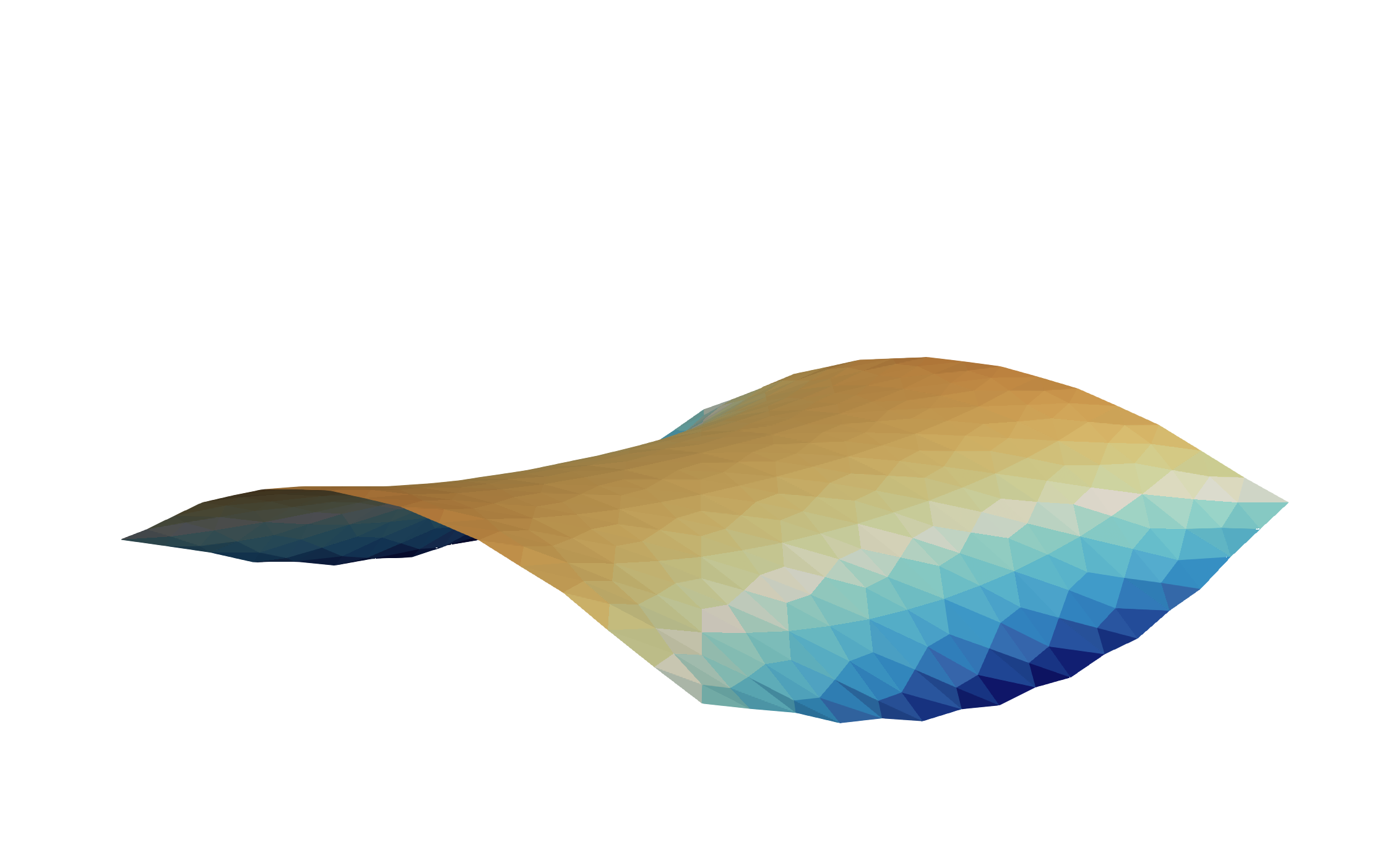}
        \caption{t = 0.1}
    \end{subfigure}
    \hfill
    \begin{subfigure}[t]{0.32\textwidth}
        \centering
         \includegraphics[width=1\textwidth]{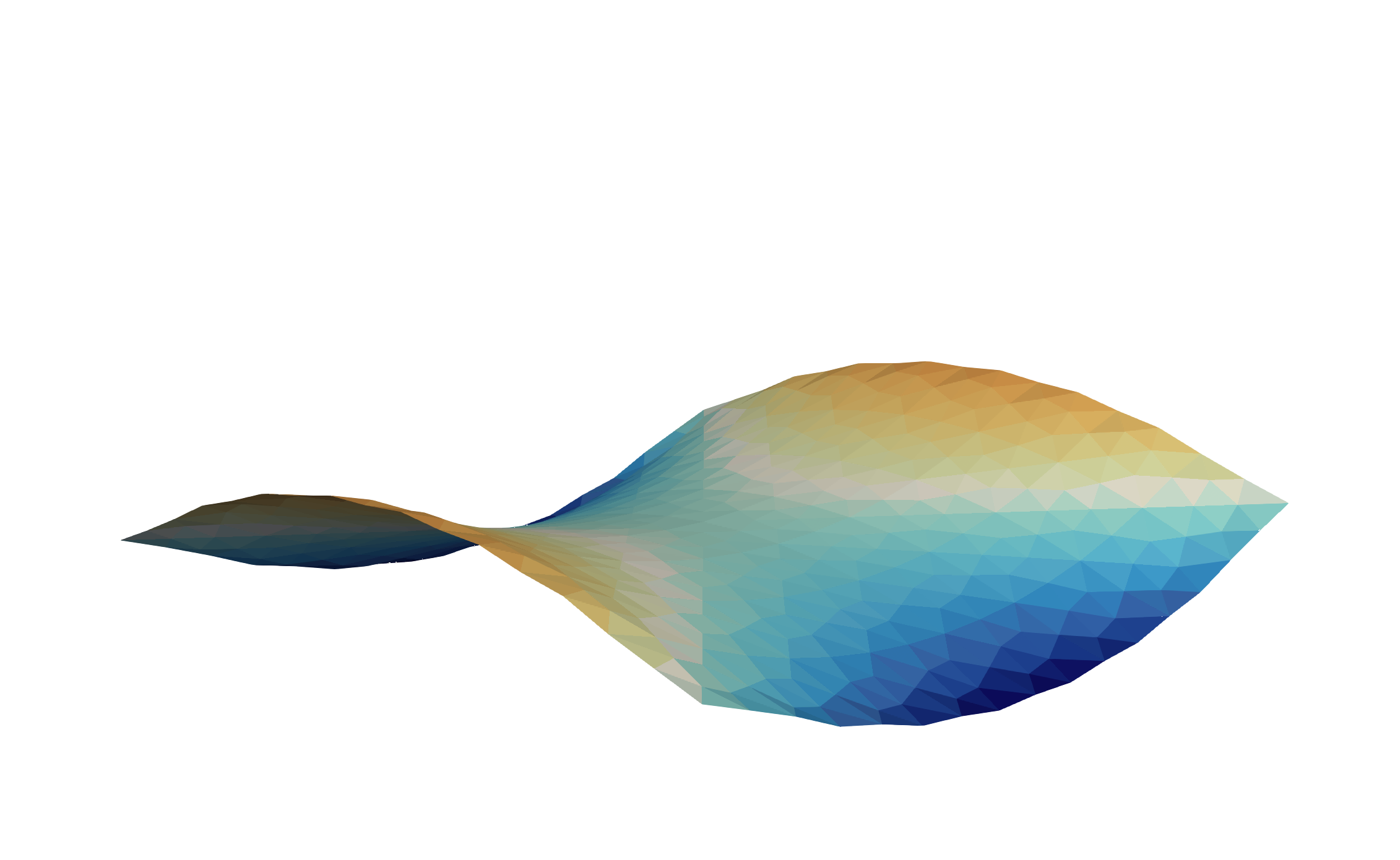}
        \caption{t = 1}
    \end{subfigure}
\caption{Test 3: NCP1FEM, snapshots of numerical solution at the time \(t\).}
\label{fig:NCFEM snapshots of numerical solution at time t}
\end{figure}

\begin{figure}[t]
    \centering
    \begin{subfigure}[t]{0.49\textwidth}
        \centering
         \includegraphics[width=1\textwidth]{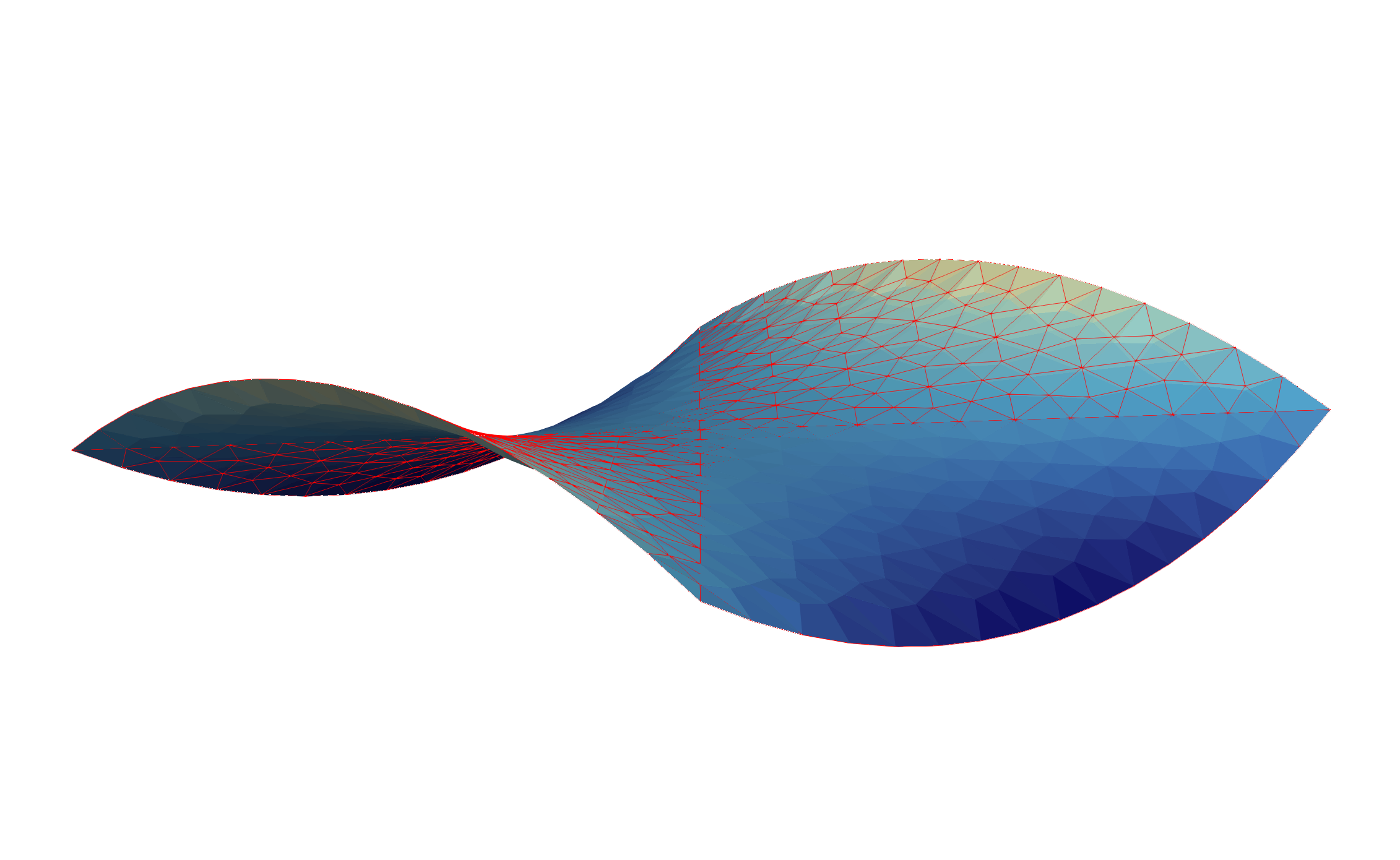}
        \caption{CP1FEM}
    \end{subfigure}%
    \hfill
    \begin{subfigure}[t]{0.49\textwidth}
        \centering
         \includegraphics[width=1\textwidth]{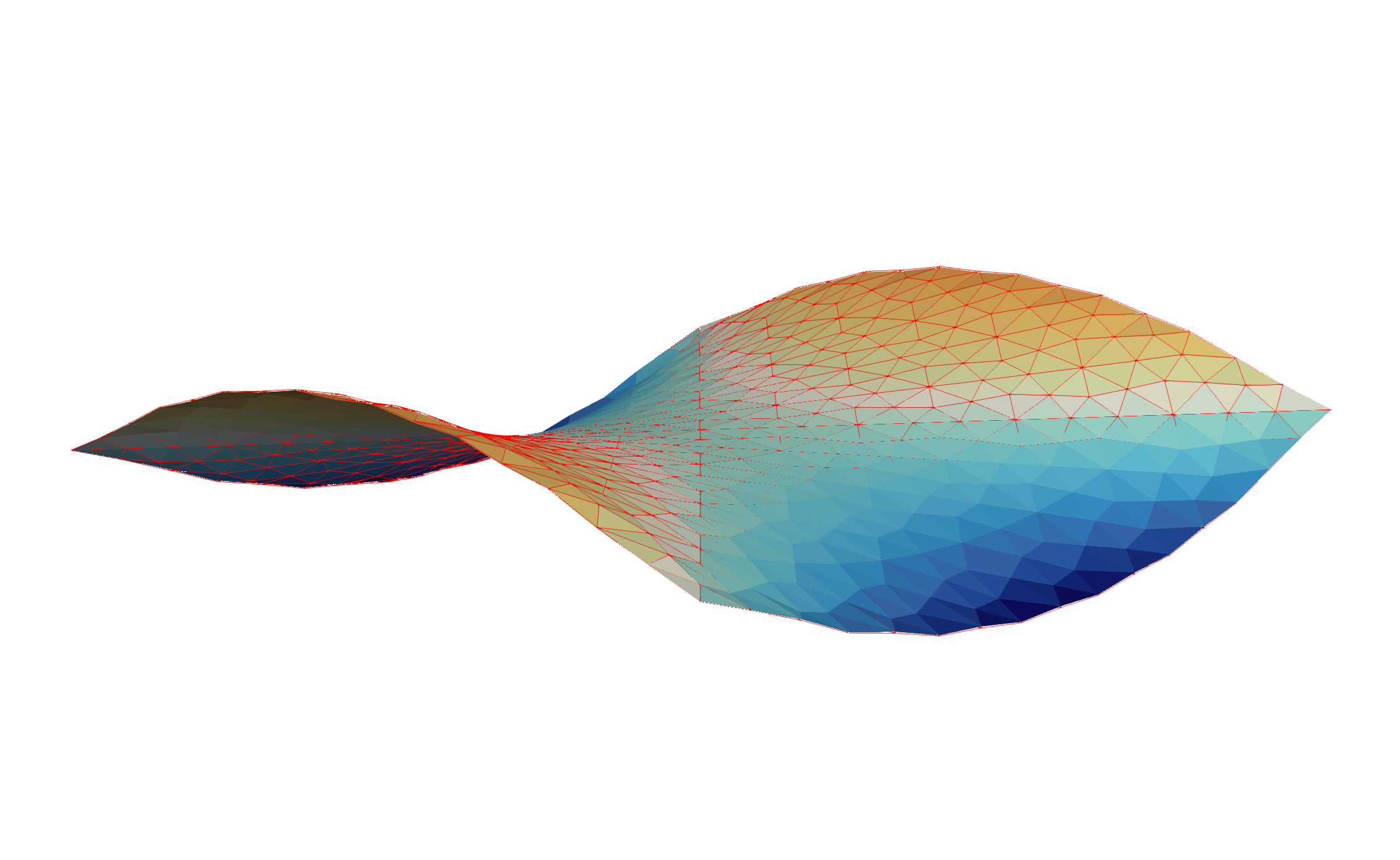}
        \caption{NCP1FEM}
    \end{subfigure}
\caption{Test 3: comparison of numerical solution at the time \(T\) and the surface \(S\) (in red wireframe).}
\label{fig:comparison numerical sol and exp}
\end{figure}

\color{black}

    \section{Conclusion}
    We established error estimates for the gradient scheme, in \(\mathbb{R}^2\), 
    which holds true for all the conforming and nonconforming approximations that are in the framework of the gradient discretisation method.
    Moreover, the minimal surface-like interpolation and its relevant error estimates we proved in this work, have the potential to apply to nonconforming numerical analysis of the geometric flow; for instance, the minimal surface flow and mean curvature flow.
    To the best of our knowledge, these estimates are the first for nonconforming approximations of the model.
    Numerical experiments have been conducted that validate our theoretical results.

    \printbibliography
    \begin{appendices}
    \section{Auxiliary results}
    Let \(F_{\rho}:\mathbb{R}^2 \rightarrow \mathbb{R}\) be defined via 
         \(F_{\rho}(\cdot) =  {\cdot}/{|\cdot|_{\rho}} \)
    with \(|\cdot|_{\rho}\) defined in \eqref{eq:def weighted rho norm}.

    \begin{proposition}
        For any \(u, v  \in \mathbb{R}^2\), 
        \begin{align}
            \frac{1}{|u|_{\rho}} & \leq \frac{1}{\rho},\label{eq-lem:aux 0 0}\\
            |u|_{\rho} - |v|_{\rho}   &\leq |u - v|,\label{eq-lem:aux 0 1}\\
              F_{\rho}(u) - F_{\rho}(v) &\leq \frac{1}{\rho}|u - v|,\label{eq-lem:aux 0 2}\\
              \big(F_{\rho}(u) - F_{\rho}(v) \big)\cdot (u - v) & \geq \bigg(1 - \frac{|v|}{ |v|_{\rho}}\bigg) \frac{|u- v|^2}{|u|_{\rho}}.\label{eq-lem:aux 0 3}
        \end{align}
    \end{proposition}
    \begin{proof}
        Take \(u, v \in \mathbb{R}^2\).

        \noindent \textbf{Proof of \eqref{eq-lem:aux 0 1}.}
        Consider 
        \begin{equation*}
            \begin{aligned}
                 |u|_{\rho} - |v|_{\rho}
                =  
                    \frac{ |u|^2  - |v|^2  }{|u|_{\rho} + |v|_{\rho}} 
                = 
                    \frac{u+v}{|u|_{\rho} + |v|_{\rho}} \cdot (u-v)
                \leq 
                    \frac{| u| + | v | }{|u|_{\rho} + |v|_{\rho}} |u - v| 
                \leq 
                    |u-v|.
            \end{aligned}
        \end{equation*}

        \noindent \textbf{Proof of \eqref{eq-lem:aux 0 2}.}
        Consider
        \begin{equation*}
            \begin{aligned}
                F_{\rho}(u) - F_{\rho}(v) 
                = 
                    \frac{|u|-|v|}{|u|_{\rho}} + |v| \bigg(\frac{1}{|u|_{\rho}} - \frac{1}{|v|_{\rho}}\bigg),
            \end{aligned}
        \end{equation*}
        where 
        \begin{equation*}
            \frac{|u|-|v|}{|u|_{\rho}} = \frac{|u|^2 - |v|^2}{|u|_{\rho}(|u| + |v|)} = \frac{u + v}{|u|_{\rho}(|u| + |v|)} \cdot (u - v) \leq \frac{|u| + |v|}{|u|_{\rho}(|u| + |v|)} \cdot |u - v| \leq \frac{|u - v|}{\rho}
        \end{equation*}
        and 
        \begin{equation*}
            |v| \bigg(\frac{1}{|u|_{\rho}} - \frac{1}{|v|_{\rho}}\bigg) = |v| \frac{|v|^2 - |u|^2}{|u|_{\rho}|v|_{\rho}(|v|_{\rho} + |u|_{\rho})} \leq \frac{|v|}{|v|_{\rho}} \frac{v+u}{|u|_{\rho} (|v|_{\rho} + |u|_{\rho})}\cdot(v - u) \leq \frac{|v - u|}{\rho}. 
        \end{equation*}

        \noindent \textbf{Proof of \eqref{eq-lem:aux 0 3}.}
         We start from the left-hand side
        \begin{align*}
             \big(F_{\rho}(u)  - F_{\rho}(v) \big) \cdot (u - v) 
            = \,&  \frac{(u - v)^2}{ |u|_{\rho}} + \bigg(\frac{1}{|u|_{\rho}} - \frac{1}{ |v|_{\rho}}\bigg) v \cdot (u - v) \\
            = \,& \frac{(u - v)^2}{|u|_{\rho}} + \frac{v^2 - u^2 }{ |u|_{\rho}  |v|_{\rho} (  |u|_{\rho} + |v|_{\rho})}v \cdot (u - v) \\
            = \,& \frac{(u - v)^2}{ |u|_{\rho}} -  \frac{(u - v)^2}{ |u|_{\rho}}  \frac{u + v }{ |u|_{\rho} +  |v|_{\rho}} \cdot \frac{v }{  |v|_{\rho}} \\
            = \,& \bigg(1 - \frac{u + v }{ |u|_{\rho} +  |v|_{\rho}} \cdot \frac{v }{  |v|_{\rho} } \bigg)\frac{(u - v)^2}{ |u|_{\rho}} \\
            \geq \, & \bigg(1 -  \frac{ |v| }{ \sqrt{\rho^2 + |v|^2} } \bigg)\frac{(u - v)^2}{ |u|_{\rho}}. \qedhere
        \end{align*} 
    \end{proof}
    \begin{proposition}
        For any \(u, v \in \mathbb{R}^2\), 
        gradient of \(F_{\rho}(u)\) and directional derivative of \(F_{\rho}(u)\) in the direction of \(v\) are
        \begin{align*}
            &\nabla F_{\rho}(u) 
            = \frac{1}{|u|_{\rho}^3}\begin{bmatrix}
               \rho^2 + u_2^2 & -u_2u_1 \\
                -u_1u_2  & \rho^2 + u_1^2
            \end{bmatrix}
            \quad \text{with} \quad u = (u_1, u_2),\\
            & DF_{\rho}(u)[v] = \nabla F_{\rho}(u) \cdot v 
                =  \frac{v}{|u|_{\rho}} - \frac{(u\cdot v)u}{|u|_{\rho}^3}.
        \end{align*}
        Then for any \(u,v,w\in \mathbb{R}^2\),
        \begin{align}
             &\frac{\rho^2|v|^2 }{|u|_{\rho}^3} \leq DF_{\rho}(u)[v] \cdot v  \leq \frac{2|v|^2}{|u|_{\rho}}, \label{eq-lem:aux 1}\\
            &DF_{\rho}(u)[v] \cdot w  \leq \frac{2|v||w|}{\rho} \label{eq-lem:aux 1.1},\\
            & \big| DF_{\rho}(u)[v] - DF_{\rho}(w)[v] \big|  \leq \bigg( \frac{1}{\rho^2 } + \frac{|u|}{\rho^3}\bigg) |w - u||v|.\label{eq-lem:aux 2}
        \end{align}
    \end{proposition}
    \begin{proof}
        We omit the proof for \eqref{eq-lem:aux 1} and \eqref{eq-lem:aux 1.1} since these follow from basic algebraic manipulation.
        Now we proof for \eqref{eq-lem:aux 2}.
        Take \(u, v, w \in \mathbb{R}^2\), 
        \begin{equation*}
            \begin{aligned}
                \big| DF_{\rho}(u)[v] - DF_{\rho}(w)[v] \big| = \bigg(\frac{1}{|u|_{\rho}} - \frac{1}{|w|_{\rho}}\bigg)v - \bigg(\frac{(u\cdot v)u}{|u|_{\rho}^3} - \frac{(w\cdot v)w}{|w|_{\rho}^3}\bigg) =: T_1 - T_2.
            \end{aligned}
        \end{equation*}
        Consider \(T_1\) first, 
        \begin{equation*}
            \begin{aligned}
                T_1 = \frac{\big(|w|^2 - |u|^2\big)v}{|u|_{\rho} |w|_{\rho} \big(|w|_{\rho} + |u|_{\rho}\big)} \leq \frac{1}{\rho^2}\bigg| \frac{\rho}{|u|_{\rho}}\bigg|\bigg| \frac{\rho}{|w|_{\rho}}\bigg|\bigg| \frac{w + u}{|w|_{\rho} + |u|_{\rho}}\bigg| |w-u||v| \leq \frac{ |w-u||v|}{\rho^2}.
            \end{aligned}
        \end{equation*}
        Now consider \(T_2\), 
        \begin{equation*}
            \begin{aligned}
                  T_2 
                =  (u\cdot v)u \bigg(\frac{1}{|u|_{\rho}^3} - \frac{1}{|w|_{\rho}^3}\bigg)
                    - \frac{(w\cdot v)w -  (u\cdot v)u}{|w|_{\rho}^3}
                =: T_{2,1} - T_{2,2}. 
            \end{aligned}
        \end{equation*}
        Consider the term \(T_{2,1}\), by \eqref{eq-lem:aux 0 0} and \eqref{eq-lem:aux 0 1},
        \begin{align*}
             T_{2,1} = \, & \frac{\big(|w|_{\rho} - |u|_{\rho} \big)\big(|w|_{\rho}^2 + |w|_{\rho}|u|_{\rho} + |u|_{\rho}^2\big)(u\cdot v)u}{|u|_{\rho}^3 |w|_{\rho}^3} \\
             \leq \, & \bigg|\frac{\big(|w|_{\rho}^2 + |w|_{\rho}|u|_{\rho} + |u|_{\rho}^2\big)|u|^2}{|u|_{\rho}^3 |w|_{\rho}^3}\bigg||w - u| |v| \\
             \leq \, & \bigg|\frac{1}{|u|_{\rho}|w|_{\rho}} + \frac{1}{|w|_{\rho}^2} + \frac{|u|}{|w|_{\rho}^3}\bigg||w - u| |v| \\
             \leq \, & \bigg|\frac{1}{\rho^2 } + \frac{|u|}{\rho^3}\bigg||w - u| |v|;
        \end{align*}
        now consider the term \(T_{2,2}\),
        \begin{align*}
            T_{2,2} = \, &  \frac{(w\cdot v) (w-u) - \big[(u - w)\cdot v\big]u }{ |w|_{\rho}^3} \\
            \leq \, & \frac{|w|}{|w|_{\rho}^3} |w - u||v| + \frac{|u|}{|w|_{\rho}^3}|w - u||v| \\
            \leq \, & \frac{1}{\rho^2} |w - u||v| + \frac{|u|}{\rho^3}|w - u||v|.
        \end{align*} 
        Collecting the estimates on \(T_1\) and \(T_2\), we arrive at
        \begin{equation*}
                 \big| DF_{\rho}(u)[v] - DF_{\rho}(w)[v] \big| \leq \bigg(\frac{1}{\rho^2 } + \frac{|u|}{\rho^3}\bigg) |w - u||v|. \qedhere
        \end{equation*}
    \end{proof}

\begin{lemma}\label{lem:S zero mean curvature}
    Surface  
    \[
    S(x,y) = 0.25 + \ln\bigg( \frac{\cos(y - 0.5)}{\cos(x - 0.5)}\bigg) \quad (x,y) \in (0,1)^2
    \]
    has zero mean curvature everywhere on \((0,1)^2\).
\end{lemma}
\begin{proof}
    To show that \(S\) has zero mean curvature,
    it suffices to show that 
    \begin{equation}\label{eq:S mean curvature}
        (1 +  S_y^2)  S_{xx} - 2  S_x S_y S_{xy} + (1 + S_x^2) S_{yy} = 0.
    \end{equation}
    Taking the derivatives of \(S\) with respect to \(x\) and \(y\) yields 
    \(\partial_x S = \tan(x - 0.5)\), 
    \(\partial_{xx} S = \sec^2(x - 0.5)\), 
    \(\partial_y S = -\tan(y - 0.5)\), 
    \(\partial_{yy} S = -\sec^2(y-0.5)\), 
    and \(\partial_{xy} S = 0.\) 
    Plugging these into the left-hand side of \eqref{eq:S mean curvature} and recalling the identity \(1 + \tan^2(\theta) = \sec^2(\theta)\)
    finishes the proof.
\end{proof}

\section{Glossary}

We list here the main notations used in the article, with a reference to their definition.

\medskip

\begin{center}
\begin{tabular}{c|l|l}
  \toprule
  Notation & Description & Defined or appearing in\\
  \midrule
  $|{\cdot}|_{\rho}$   & Regularized Euclidean length & \eqref{eq:def weighted rho norm}\\
  $\mathcal{D}_T$ & Space-time gradient discretisation (GD) & Definition \ref{def:GD definition}\\
  $X_{\mathcal{D}}$ & Space of $\mathcal D_T$ & Definition \ref{def:GD definition}\\
  $\|{\cdot}\|_{\mathcal{D}}$ & Norm on $X_{\mathcal D}$ & \eqref{eq-def:GD definition norm}\\
  $\Pi_{\mathcal{D}}$ & Function reconstruction & Definition \ref{def:GD definition}\\
  $\nabla_{\mathcal{D}}$ & Gradient reconstruction & Definition \ref{def:GD definition}\\
  $S_{\mathcal{D}}$ & Measure of consistency of the GD & Definition \ref{def:GD consistency}\\
  $W_{\mathcal{D}}$ & Measure of limit-conformity of the GD & Definition \ref{def:GD limit conformity}\\
  $h_{\mathcal{D}}$ & Space size of the GD & Definition \ref{def:GD space size}\\
  $e_{\mathcal{D}}^{ini}$ & Consistency error of the initial condition & \eqref{def:eDini}\\
  $\Lambda^{(1)}$ & Constant $<1$ & \eqref{eq-proof-lem::NL interpolation error estimate constant 1}\\
  $\Lambda^{(2)}$ & Free positive constant & Lemma \ref{lem:GS consistency} \\
  $\Lambda^{(3)}$ & Constant $\gtrsim 1$ & \eqref{eq:def.lambda3}\\
  $\Lambda^{(4)}$ & Constant & \eqref{def:lambda4}\\
  \bottomrule
\end{tabular}
\end{center}

\end{appendices}

\end{document}